\documentclass[11pt,a4paper,openany,oneside]{article}
\usepackage[a4paper,left=2cm,right=2cm, top=3cm, bottom=3cm]{geometry}
\usepackage{times}
\usepackage{authblk}
\usepackage[latin1]{inputenc}
\usepackage[russian, english]{babel}
\usepackage{mathrsfs}
\usepackage{stmaryrd}
\usepackage{amssymb,amsmath,amsthm}
\usepackage{hyperref}
\usepackage{graphicx}
\usepackage{epstopdf}
\usepackage{subcaption}
\usepackage{eurosym}
\usepackage{color}
\usepackage{enumerate}
\usepackage{multirow}
\usepackage{stmaryrd}
\usepackage{mathtools}
\usepackage{scalerel}
\usepackage{booktabs}
\usepackage{float}
\usepackage[T2A,T1]{fontenc}
\usepackage{amsfonts}

\theoremstyle{plain}
\newtheorem{thm}{Theorem}[section]

\newtheorem{lem}[thm]{Lemma}
\newtheorem{lemma}[thm]{Lemma}

\theoremstyle{definition}
\newtheorem{defn}{Definition}[section]
\theoremstyle{remark}
\newtheorem{remark}{Remark}

\renewcommand{\C}{\mathbb C}
\newcommand{\R}{\mathbb R}
\newcommand{\N}{\mathbb N}

\newcommand{\chitre}{\chi^{(3)}}

\newcommand{\PT}{\mathcal{P}\mathcal{T}}
\newcommand{\CyrB}{\mbox{\usefont{T2A}{\rmdefault}{m}{n}\CYRB}}
\newcommand{\cyrB}{\mbox{\scriptsize{\usefont{T2A}{\rmdefault}{m}{n}\CYRB}}}
\newcommand{\Zhe}{\mbox{\usefont{T2A}{\rmdefault}{m}{n}\CYRZH}}

\def\ri{{\rm i}}

\def\astB{\ast_\B}

\def\kj{k^{(j)}}

\newcommand{\ko}{{k^0}}
\def\kalpha{k^{(\alpha)}}
\def\eps{\varepsilon}
\def\kbeta{k^{(\beta)}}
\def\kgamma{k^{(\gamma)}}
\def\sumjN{\sum_{j=1}^N}
\def\sumabgN{\sum_{\alpha,\beta,\gamma=1}^N}
\def\sumabN{\sum_{\alpha,\beta=1}^N}

\def\sumabcd{\sum_{a,b,c,d=1}^3}

\def\Pke{P_k^\epsilon}
\def\Qke{Q_k^\epsilon}
\def\ePk{\prescript{\epsilon}{}{P_k}}
\def\eQk{\prescript{\epsilon}{}{Q_k}}

\def\bi{\begin{itemize}}
\def\ei{\end{itemize}}

\newcommand{\B}{\mathbb{B}}

\newcommand{\Z}{\mathbb{Z}}
\newcommand{\mK}{\mathbb K}
\newcommand{\cA}{{\mathcal A}}

\newcommand{\cB}{{\mathcal B}}

\newcommand{\cG}{{\mathcal G}}
\newcommand{\cH}{{\mathcal H}}

\newcommand{\cK}{{\mathcal K}}
\newcommand{\cL}{{\mathcal L}}

\newcommand{\cN}{{\mathcal N}}
\newcommand{\cO}{{\mathcal O}}

\newcommand{\cQ}{{\mathcal Q}}

\newcommand{\cT}{{\mathcal T}}

\newcommand{\cW}{{\mathcal W}}
\newcommand{\cX}{{\mathcal X}}

\renewcommand{\dim}{{\rm dim}\,}

\newcommand{\pa}{\partial}

 \def\dd{\, {\rm d}}

\newcommand{\hA}{\widehat A}
\newcommand{\hB}{\widehat B}
\newcommand{\hC}{\widehat C}

\newcommand{\hL}{\widehat L}
\newcommand{\hR}{\widehat R}
\newcommand{\hv}{\widehat v}
\newcommand{\ha}{\widehat a}
\newcommand{\hb}{\widehat b}
\newcommand{\hbj}{\widehat b_j}

\newcommand{\hcN}{\widehat\cN}

\newcommand{\hbA}{\widehat{\overline A}}
\newcommand{\hbB}{\widehat{\overline B}}
\newcommand{\hbC}{\widehat{\overline C}}

\newcommand{\Bp}{B^*}

\newcommand{\hBp}{\widehat{B}^*}
\newcommand{\hbBp}{\widehat{\overline B}^*}
\newcommand{\hCp}{\widehat{C}^*}

\newcommand{\tF}{\widetilde F}

\newcommand{\tf}{\widetilde f}
\newcommand{\tilg}{\widetilde g}
\newcommand{\tu}{\widetilde u}
\newcommand{\tv}{\widetilde v}
\newcommand{\tw}{\widetilde w}

\newcommand{\tvB}{\widetilde v_B}
\newcommand{\tvC}{\widetilde v_C}

\newcommand{\tvBj}{\widetilde v_{B_j}}

\newcommand{\tbu}{\widetilde{\overline u}}
\newcommand{\tbv}{\widetilde{\overline v}}

\newcommand{\Ae}{\hA^\varepsilon}
\newcommand{\bAegamma}{\widehat{\overline{A^\varepsilon_\gamma}}}
\newcommand{\Aej}{\hA^\varepsilon_j}
\newcommand{\hae}{\ha^\varepsilon}
\newcommand{\haej}{\ha^\varepsilon_j}

\newcommand{\trans}{{\mathsf T}}

\DeclareMathOperator{\Id}{Id}

\DeclareMathOperator{\spann}{span}

\DeclareMathOperator*{\esssup}{ess\,sup}

\DeclareMathOperator{\Real}{Re}
\DeclareMathOperator{\Imag}{Im}

\DeclareOldFontCommand{\it}{\normalfont\itshape}{\mathit}

\newcommand{\icol}[1]{% inline column vector
	\left(\begin{smallmatrix}#1\end{smallmatrix}\right)%
}

\def\barl#1{\overline{#1}}% Conjugate

\newcommand{\dist}{\text{\rm dist}}
\newcommand{\supp}{\text{\rm supp}}

\newcommand{\bspm}{\left(\begin{smallmatrix}}\newcommand{\espm}{\end{smallmatrix}\right)}
\newcommand{\bpm}{\begin{pmatrix}}\newcommand{\epm}{\end{pmatrix}}

\makeatletter
\newcommand*\bigcdot{\mathpalette\bigcdot@{.3}}
\newcommand*\bigcdot@[2]{\mathbin{\vcenter{\hbox{\scalebox{#2}{$\m@th#1\bullet$}}}}}
\makeatother

\def\blem{\begin{lemma}}\def\elem{\end{lemma}}
\def\bthm{\begin{theorem}}\def\ethm{\end{theorem}}
\def\bcor{\begin{corollary}}\def\ecor{\end{corollary}}
\def\beq{\begin{equation}}\def\eeq{\end{equation}}
\def\beqq{\begin{equation*}}\def\eeqq{\end{equation*}}
\def\bal{\begin{align}}\def\eal{\end{align}}
\def\bpf{\begin{proof}}\def\epf{\end{proof}}
\def\bex{\begin{example}}\def\eex{\end{example}}
\def\brem{\begin{remark}}\def\erem{\end{remark}}
\def\bass{\begin{assumption}}\def\eass{\end{assumption}}
\def\bprop{\begin{proposition}}\def\eprop{\end{proposition}}
\def\bdefi{\begin{defn}}\def\edefi{\end{defn}}

\makeatletter
\DeclareRobustCommand\widecheck[1]{{\mathpalette\@widecheck{#1}}}
\def\@widecheck#1#2{%
	\setbox\z@\hbox{\m@th$#1#2$}%
	\setbox\tw@\hbox{\m@th$#1%
		\widehat{%
			\vrule\@width\z@\@height\ht\z@
			\vrule\@height\z@\@width\wd\z@}$}%
	\dp\tw@-\ht\z@
	\@tempdima\ht\z@ \advance\@tempdima2\ht\tw@ \divide\@tempdima\thr@@
	\setbox\tw@\hbox{%
		\raise\@tempdima\hbox{\scalebox{1}[-1]{\lower\@tempdima\box
				\tw@}}}%
	{\ooalign{\box\tw@ \cr \box\z@}}}
\makeatother

\numberwithin{equation}{section}

\begin{document}

\title{Justification of the Asymptotic Coupled Mode Approximation of Out-of-Plane Gap Solitons in Maxwell Equations}
\date{\today}
\author{Tom\'a\v{s} Dohnal and Giulio Romani}
\affil{Institut f\"{u}r Mathematik,  Martin-Luther-Universit\"{a}t Halle-Wittenberg, 06099 Halle (Saale), Germany\\
\small{tomas.dohnal@mathematik.uni-halle.de, giulio.romani@mathematik.uni-halle.de}}

\maketitle

\abstract{In periodic media gap solitons with frequencies inside a spectral gap but close to a spectral band can be formally approximated by a slowly varying envelope ansatz. The ansatz is based on the linear Bloch waves at the edge of the band and on effective coupled mode equations (CMEs) for the envelopes. We provide a rigorous justification of such CME asymptotics in two-dimensional photonic crystals described by the Kerr nonlinear Maxwell system. We use a Lyapunov-Schmidt reduction procedure and a nested fixed point argument in the Bloch variables. The theorem provides an error estimate in $H^2(\R^2)$ between the exact solution and the envelope approximation. The results justify the formal and numerical CME-approximation in [Dohnal and D\"orfler, Multiscale Model. Simul., p. 162-191,
11 (2013)].}

\vskip0.2truecm

%-----------------------------------------------
%-----------------------------------------------

\section{Introduction}\label{S:intro}

Maxwell's equations in Kerr nonlinear dielectric materials without free charges are described by
\begin{equation}\label{Max}
\begin{split}
\mu_0\partial_tH=-\nabla\times E,\qquad\epsilon_0\partial_t D=\nabla\times H,\qquad\nabla\cdot D=\nabla\cdot H=0,
\end{split}
\end{equation}
where $E=(E_1,E_2,E_3)$ and $H=(H_1,H_2,H_3)$ are the electric and the magnetic field respectively, $D=(D_1(E),D_2(E),D_3(E))$ is the electric displacement field and $\epsilon_0$ and $\mu_0$ are the permittivity and the permeability of the free space, respectively. We assume the constitutive relations 
\begin{equation*}\label{constit_rel}
D(x,t)=\epsilon_0(\epsilon(x_1,x_2)E(x,t)+\mathcal F(E)(x,t)), \quad x\in\R^3,\; t\in\R,
\end{equation*}
where
\begin{equation}\label{cF}
\mathcal F_d(E)(x,t):=\sum_{a,b,c\,=1}^3\chi^{(3)}_{a,b,c,d}(x_1,x_2)(E_aE_bE_c)(x,t), \quad d\in \{1,2,3\}.
\end{equation}
We model a two dimensional photonic crystal and hence assume that the dielectric function (relative permittivity) $\epsilon:\R^2\to\R$ and the cubic electric susceptibility $\chi^{(3)}:\R^2\to\R^{3\times 3\times 3\times 3}$ are periodic and $\epsilon$ is positive. The periodicity is specified by two linearly independent lattice vectors $a^{(1)},a^{(2)}\in\R^2$ defining the \textit{Bravais lattice} $\Lambda:=\spann_\Z\{a^{(1)},a^{(2)}\}$ of the crystal. Then the required periodicity reads
\beq\label{E:eps-per}
\epsilon(x)=\epsilon(x+R), \;\;\, \chi^{(3)}(x)=\chi^{(3)}(x+R) \qquad\mbox{for any }\,x\in\R^2\;\mbox{ and }\,R\in\Lambda.\eeq
Since $\epsilon=\epsilon(x_1,x_2)$ and $\chi^{(3)}=\chi^{(3)}(x_1,x_2)$, the material is homogeneous in the $x_3$-direction.
In the following $\cQ\subset\R^2$ denotes the Wigner-Seitz periodicity cell, defined as the set of points in $\R^2$ which are closer to $0$ than to any other lattice point in $\Lambda$.  More precisely,
$$\cQ:=\{x\in\R^2\,|\,|x|<|x-Z|,\,Z\in\Lambda\setminus\{0\}\}\cup S,$$
where the connected subset $S\subset\partial\cQ$ is chosen so that $\bigcup_{Z\in\Lambda}(\cQ+Z)=\R^2$ and $\cQ\cap(\cQ+Z)=\emptyset$ for all $Z\in\Lambda$. We often use the term $\Lambda-$periodic to mean the periodicity as in \eqref{E:eps-per}.

We consider monochromatic waves propagating in the homogeneous $x_3$-direction, i.e. out of the plane of periodicity of the 2D crystal, and use the ansatz
\begin{equation}\label{monochromatic}
(E,H)(x,t)=e^{\ri(\kappa x_3-\omega t)}(u,h)(x_1,x_2;\omega)+\text{c.c.}
\end{equation}
where $\kappa\in\R$ and c.c. denotes the complex conjugate. We look for profiles $u,h$ localized in both $x_1$ and $x_2$ and with $\omega$ in a frequency gap. The resulting solutions  are called \textit{out-of-plane gap solitons}. Inserting such a monochromatic ansatz into the nonlinearity \eqref{cF} and neglecting the higher harmonics\footnote{Neglecting higher harmonics is a common approach in theoretical studies of weakly nonlinear optical waves \cite{shen1984}. Alternatively, one can use a time averaged model for the nonlinear part of the displacement field, see \cite{Stuart_93,Sutherland03,DD}, where no higher harmonics appear.}, one obtains
\begin{equation*}\label{F}
\mathcal{F}_d(E)(x,t)=\sum_{a,b,c=1}^3\chi^{(3)}_{a,b,c,d}(x_1,x_2)(\overline{u}_a u_b u_c + u_a \overline{u}_b u_c+ u_au_b\overline{u}_c)(x_1,x_2)e^{\ri(\kappa x_3-\omega t)}+\text{c.c.}, \quad d\in \{1,2,3\}\,.
\end{equation*}
We define 
$$
\begin{aligned}
F_d(u)&:=\sum_{a,b,c=1}^3\chi^{(3)}_{a,b,c,d}(\overline{u}_a u_b u_c + u_a \overline{u}_b u_c+ u_au_b\overline{u}_c)=\sum_{a,b,c=1}^3\underline{\chi}^{(3)}_{a,b,c,d}u_au_b\overline{u}_c, \quad d\in \{1,2,3\},
\end{aligned}
$$
where
$$\underline{\chi}^{(3)}_{a,b,c,d}:=\chi^{(3)}_{c,b,a,d}+\chi^{(3)}_{a,c,b,d}+\chi^{(3)}_{a,b,c,d}\,.$$
We rescale the frequency by defining $\tilde\omega=\frac\omega c$, where $c=(\mu_0\epsilon_0)^{-1/2}$, but drop the tilde again for better readability. Then, with the ansatz in \eqref{monochromatic} Maxwell's equations \eqref{Max} become
\begin{equation}\label{System}
\begin{cases}
\nabla'\times u-\ri\mu_0c\omega h=0,\\
\nabla'\times h+\ri\epsilon_0c\omega\epsilon(\cdot)u=\ri\epsilon_0c\omega F(u),
\end{cases}
\end{equation}
where $\nabla':=\icol{\partial_1\\\partial_2\\\ri\kappa}$ is the restriction of the standard $\nabla$ applied to our 2D-ansatz \eqref{monochromatic}. Notice that, indeed, the divergence equations in \eqref{Max} are automatically satisfied by our ansatz. Equivalently, we may write \eqref{System} as a second-order equation for the electric field $u$,
\begin{equation}\label{eq}
(L^{(E)}-\omega^2\epsilon(\cdot))u:=\nabla'\times\nabla'\times u-\omega^2\epsilon(\cdot)u=\omega^2F(u),
\end{equation}
and then, having determined a solution $u$, the magnetic field can be recovered by
\begin{equation*}
h=-\tfrac\ri{\mu_0c\omega}\left(\nabla'\times u\right).
\end{equation*}
For any $\omega$ in a spectral gap of the linear problem $(L^{(E)}-\omega^2\epsilon(\cdot))u=0$ equation \eqref{eq} is expected to have localized solutions $u$ with $u(x_1,x_2)\to 0$ as $|(x_1,x_2)|\to \infty$, called \textit{gap solitons}. This has been proved variationally for other problems, e.g. the periodic Gross-Pi\-ta\-ev\-skii equation, see \cite{Pankov05}, or equation \eqref{eq} with other (not periodic) coefficients $\epsilon$ and $\chi^{(3)}$, see e.g. \cite{BDPR2014,Mederski_2016}. From the physics point of view, gap solitons are phenomenologically interesting as they achieve a balance between the periodicity induced dispersion and the focusing or defocusing of the nonlinearity. In addition, they exist for frequencies in spectral gaps, i.e. where no linear propagation is possible. Examples of physics references for gap solitons in two dimensions are \cite{Fleischer2003} or \cite[Sec. 16.6]{slusher2003nonlinear}.

In \cite{DD} an approximation of gap solitons of \eqref{eq} with periodic coefficients and for $\omega$ in an asymptotic vicinity of a gap edge was formally obtained using a slowly varying envelope approximation. In particular, envelopes of such gap solitons satisfy a system of nonlinear equations with constant coefficients, so-called \textit{couple mode equations} (CMEs), posed in a slow variable. The advantage is that such a system can be numerically solved with less effort than the original Maxwell system \eqref{eq}, which is posed in the ``fast'' variable $x$. Then, the solution of \eqref{eq} for $\omega$ near a band edge would be asymptotically approximated by the sum of linear Bloch waves at the edge modulated by the corresponding envelopes. The aim of this paper is to give a rigorous justification of this approximation.
\vskip0.2truecm

Let us now describe the approximation in more detail. First, recall that the (first) Brillouin zone (here denoted by $\B\subset\R^2$) is the Wigner-Seitz periodicity cell for the reciprocal lattice $\Lambda^*:=\spann_\Z\{b^{(1)},b^{(2)}\}\subset\R^2$. The vectors $b^{(1)},b^{(2)}$ satisfy $a^{(i)}\cdot b^{(j)}=2\pi\delta_{ij}$ for $i,j\in\{1,2\}$, with $\delta_{ij}$ being the Kronecker-delta.

Let now $\omega_*$ be a boundary point of the (real) spectrum of the pencil $L^{(E)}-\omega^2\epsilon$, i.e. such that 
there is a choice of $\Omega\in\{-1,+1\}$ for which $\omega_*+\tau\Omega$ lies outside and $\omega_*-\tau\Omega$ inside the spectrum for all $\tau>0$ small enough. We are interested in studying solutions of equation \eqref{eq} when $\omega$ lies in a band gap and is asymptotically close to $\omega_*$. Hence, in our main result (Theorem \ref{Thm_main}), we choose a small parameter $0<\eps\ll 1$ and set
\begin{equation}\label{omega1}
\omega=\omega_*+\Omega\varepsilon^2,
\end{equation}
where $\Omega\in\{-1,+1\}$ is chosen such that $\omega$ lies outside the spectrum.
We aim to study the existence of a solution of \eqref{eq} close to the \textit{slowly varying envelope} ansatz
\begin{equation}\label{ansatz_phys_var}
u_\text{ans}(x)=\varepsilon\sum_{j=1}^N A_j(\varepsilon x)u_{n_*}(x, k^{(j)}).
\end{equation}
Here $(A_j)_{j=1}^N$ are localized envelopes to be determined below and the function $u_{n_*}(x,k)$, for $n_*\in\N$, is a Bloch wave of $L^{(E)}-\omega_*^2\epsilon$. Writing 
$$\nabla'_k:=\nabla'+\ri k \quad \text{for} \ k\in\R^2,$$ 
a \textit{Bloch wave} $u_n(\cdot,k):\R^2\times\R^2\to\C^3$ is defined as $$u_n(x,k):=p_n(x,k)e^{\ri k\cdot x}, \qquad n\in\N,$$ 
where $p_n(\cdot,k)$ is a solution of the periodic eigenvalue problem
\begin{equation}\label{eq_linear_Bloch}
\begin{split}
\nabla'_k\times\nabla'_k\times p_n(x,k)&=\omega_n(k)^2\epsilon(x)p_n(x,k),\qquad\mbox{for all}\,\,x\in\R^2,\\
p_n(x+R,k)&=p_n(x,k)\qquad\mbox{for all}\,\,x\in\R^2\,\,\mbox{and all}\,\,R\in\Lambda,
\end{split}
\end{equation}
which is to be solved for the eigenpair $(\omega_n(k),p_n(\cdot,k))$. Let us choose the normalization
$$\int_\cQ\epsilon(x)|p_n(x,k)|^2\dd x=1, \ n\in \N,$$ 
see \eqref{orthog_pn} below. Note that for a geometrically simple eigenvalue the eigenfunction $p_n(\cdot,k)$ is unique up to a phase factor $e^{\ri \alpha},\alpha\in\R$. We call $p_n(\cdot,k)$ a \textit{Bloch eigenfunction}.  Clearly, Bloch waves are quasiperiodic
$$u_n(x+R,k)=u_n(x,k)e^{\ri k\cdot R}\qquad\mbox{for all}\,\,x\in\R^2\,\,\mbox{and all}\,\,R\in \Lambda.$$
Note that, strictly speaking, a ``Bloch wave'' is the time dependent function $u_n(x,k)e^{\ri \omega_n(k)t}$ but, for the purpose of this paper, we use this name for the factor $u_n(x,k)$.

We assume that the band structure $k\mapsto \{\omega_n(k): n\in \N\}$ attains the value $\omega_*$ at finitely many points, denoted $k^{(1)},\dots,k^{(N)} \in \B$, see \eqref{ansatz_phys_var}. In other words, there exist indices $(n_j)_{j=1}^N$ for which $\omega_{n_j}(\kj)=\omega_*$ holds, with $k\mapsto\omega_{n_j}(k)$ defined by \eqref{eq_linear_Bloch}. Notice that since at each $k\in \B$ the eigenvalues $\omega_n(k)$ are ordered by magnitude, we necessarily have $\omega_{n_*}(k^{(j)})=\omega_*$ for some $n_*\in\N$ and all $j\in\{1,\dots,N\}$. Due to the additional assumption of geometric simpleness at $k\in \{k^{(1)},\dots,k^{(N)}\}$ at the level $\omega_*$, there is only one Bloch eigenfunction (up to a complex phase factor) at  $k\in \{k^{(1)},\dots,k^{(N)}\}$ at the level $\omega_*$. We denote this eigenfunction by $p_{n_*}(\cdot,k^{(j)})$. For a precise formulation of the corresponding assumptions see assumption (A3) in Sec. \ref{Section_AssMR}. Also note that the spectrum of the pencil $L^{(E)}-\omega^2\epsilon$ equals the union of the ranges of the functions $\omega_n$ over all $n$, see Sec. \ref{S:spec-H} and \ref{S:spec-E}.

\begin{figure}[h!]
\begin{center}
    \begin{minipage}[c]{0.25\textwidth}
    \centering
    (a)\\
		\includegraphics[width=0.85\linewidth]{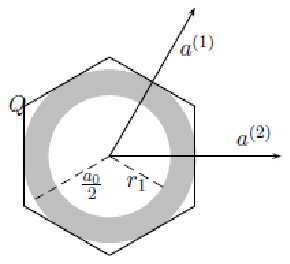}
    \end{minipage}
   \begin{minipage}[c]{0.25\textwidth}
    \centering
    (b)\\
		\includegraphics[width=0.77\linewidth]{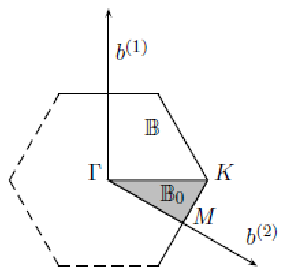}
    \end{minipage}
		\begin{minipage}[c]{0.48\textwidth}
    \centering
    (c)\\
		\includegraphics[width=\linewidth]{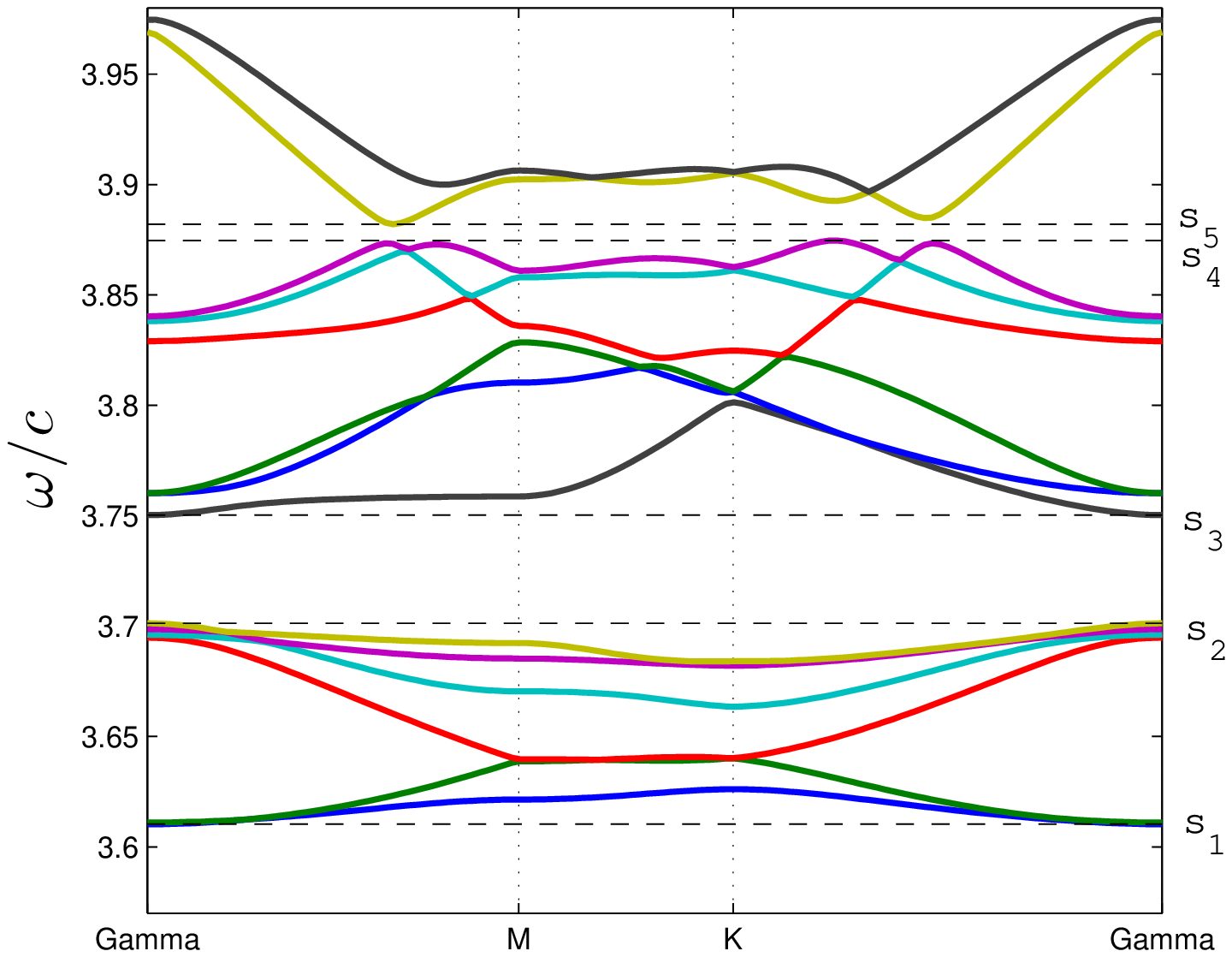}
		\end{minipage}
	\end{center}
 \caption{\label{F:crystal}\small%
(taken from \cite{DD};  Copyright $\copyright$2013 Society for Industrial and Applied Mathematics.  Reprinted with permission.  All rights reserved.) (a) Hexagonal periodicity cell $\cQ$ with a cylindrical material structure and the lattice vectors $a^{(1)},a^{(2)}$. (b) The corresponding Brillouin zone $\B$ (not to scale) with the reciprocal lattice vectors $b^{(1)},b^{(2)}$ and a shaded irreducible Brillouin zone $\B_0$. (c) Band structure along $\pa\B_0$ for $r_1= \tfrac{1.31}{4.9}a_0$, $\epsilon(x)=2.1025$ for $r_1\leq |x|\leq a_0/2$, and $\epsilon(x)=0$ otherwise. The distance $a_0>0$ between neighboring cylinders can be chosen arbitrary.}
\end{figure}
In Fig. \ref{F:crystal} we plot the material structure, the Brillouin zone, and the band structure for an example adopted from \cite{DD}. Only the band structure along the boundary of an ``irreducible'' Brillouin zone $\B_0$ is plotted, which is standard practice in the physics literature. It was checked in \cite{DD} that the level sets of the first five band edges do not include any points from the interior of $\B_0$. Five spectral edges $s_1,\dots, s_5$ are labeled. Note that, for instance, the edge $\omega_*=s_3$ has $N=1$ as the level set includes only the point $k^{(1)}=\Gamma=(0,0)$. At $\omega_*=s_5$ we have $N=6$ because the minimal point along the line $\Gamma-M$ is repeated 6 times in the full Brillouin zone $\B$ due to a discrete rotational symmetry of the lattice.

As shown in \cite{DD}, in order for the residual $L^{(E)}u_\text{ans}-\omega^2\epsilon u_\text{ans} -\omega^2F(u_\text{ans})$ to be small, the functions $\left(A_j\right)_{j=1}^N$ in \eqref{ansatz_phys_var} have to satisfy the second-order CMEs
\begin{equation}\label{CMEs_phys}
\Omega A_j+\frac12\big(\partial_{k_1}^2\omega_{n_*}(\kj)\partial_{y_1}^2+\partial_{k_2}^2\omega_{n_*}(\kj)\partial_{y_2}^2+2\partial_{k_1}\partial_{k_2}\omega_{n_*}(\kj)\partial_{y_1}\partial_{y_2}\big)A_j+\cN_j=0, \quad j=1,\dots,N
\end{equation}
in $\R^2$, where $y:=\varepsilon x$ is the slow variable and the nonlinear term $\cN_j$ is given by
\begin{equation}\label{Nj}
\cN_j=\sum_{(\alpha,\beta,\gamma)\in\sigma_j} I_{\alpha,\beta,\gamma}^{\,j}A_\alpha A_\beta \overline{A_\gamma}.
\end{equation}
where 
\begin{equation}\label{sigma_j}
\sigma_j :=\{(\alpha,\beta,\gamma)\in \{1,\dots,N\}^3: \kalpha+\kbeta-\kgamma-\kj \in\Lambda^*\}.
\end{equation}
The coefficients $I_{\alpha,\beta,\gamma}^{\,j}$ are determined by the Bloch wave $u_{n_*}$ at the points $\kj$, in detail
\begin{equation} \label{I_abg^j}
\begin{array}{rl}
I_{\alpha,\beta,\gamma}^{\,j}
:= \displaystyle
\frac{\omega_*}{2}\sum_{a,b,c,d=1}^3 \langle
\chi^{(3)}_{a,b,c,d}u_{n_*,a}(\cdot,k^{(\alpha)})u_{n_*,b}(\cdot,k^{(\beta)})
\barl{u_{n_*,c}}(\cdot,k^{(\gamma)}),u_{n_*,d}(\cdot,k^{(j)})\rangle.
\end{array}
\end{equation}

The formal derivation of the CMEs \eqref{CMEs_phys} as an effective model for the envelopes $A_j$ can be summarized as follows. First, ansatz \eqref{ansatz_phys_var} is inserted into \eqref{eq} and for each $j$ the terms proportional to $e^{\ri \kj\cdot x}$ times a $\Lambda$-periodic function are collected. Then, setting the $L^2(\cQ)$-inner product of the leading order part of these terms with $u_{n_*}(\cdot,\kj)$ to zero, produces the $j$-th equation in \eqref{CMEs_phys}. In the inner product the variable $y:=\eps x$ is considered independent of $x$.

For several examples with the coefficients $\nabla^2\omega_{n_*}(\kj)$ and $I^j_{\alpha,\beta,\gamma}$ obtained from actual Bloch waves of the corresponding Maxwell problem, localized solutions were found numerically in \cite{DD}. Fig. \ref{F:GS-approx} (a), (b) shows an example solution of CMEs \eqref{CMEs_phys} corresponding to $\omega_*=s_5$ in Fig. \ref{F:crystal}.  We also plot the total intensity $|u_\text{ans}|^2$ of the corresponding formal approximation \eqref{ansatz_phys_var} at $\omega=\omega_*-\eps^2$ with $\eps=0.1$  in
Fig. \ref{F:GS-approx} (c). All plots in Fig. \ref{F:GS-approx} are adopted from \cite{DD}.
\begin{figure}[h!]
\begin{center}
\includegraphics[scale=1.6]{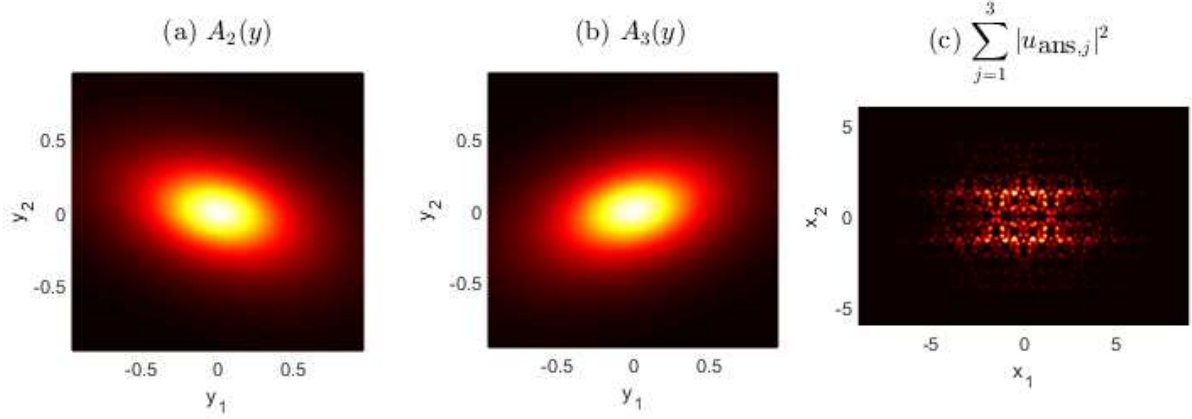}%.png
\end{center}
 \caption{\label{F:GS-approx}\small%
(taken from \cite{DD};  Copyright $\copyright$2013 Society for Industrial and Applied Mathematics.  Reprinted with permission.  All rights reserved.) (a), (b) Components $A_2$ and $A_3$ of CMEs for the case $\omega_*:=s_5$ in Fig. \ref{F:crystal}, where $N=6$. A solution with the symmetry $A_1=A_4=0$, $A_5=A_2$, and $A_6=A_3$ was chosen. (c) The approximation of the square modulus of a gap soliton at $\omega=\omega_*-\eps^2$ with $\eps=0.1$ as given by $|u_{\text{ans},1}|^2+|u_{\text{ans},2}|^2+|u_{\text{ans},3}|^2$, see \eqref{ansatz_phys_var}.}
\end{figure}

The main result of this paper is that for $\eps>0$ small enough the existence of suitable solutions of CMEs \eqref{CMEs_phys} implies the existence of gap solitons of \eqref{eq} with $\omega$ given by \eqref{omega1}. These gap solitons are approximated by the ansatz \eqref{ansatz_phys_var}. The following theorem uses assumptions (A1)-(A7) on the band structure and on the functions $\epsilon$ and $\chitre$, see Sec. \ref{Section_AssMR}, the non-degeneracy property, see Def. \ref{D:nondegen} as well as the $\PT$ (parity-time reversal) symmetry defined below.
\begin{defn}
A function $A\in L^2(\R^n)$ is called $\PT$-symmetric if $A=\overline{A(-\,\cdot)}$.	
\end{defn}

\begin{thm}\label{Thm_main}
	Let $\kappa\in \R\setminus \{0\}$. Suppose $A=(A_j)_{j=1}^N\in H^{s_A}(\R^2,\C^N)$ is a $\PT$-symmetric non-degenerate solution of the CMEs \eqref{CMEs_phys} with $s_A>1$. Then, under assumptions (A1)-(A7), see Sec. \ref{Section_AssMR}, there are constants $c>0,\Omega \in \{-1,1\}$, and $\varepsilon_0>0$ such that for each $\varepsilon\in (0,\varepsilon_0)$ there exists a $\PT$-symmetric solution $u\in H^2(\R^2,\C)$ of the reduced Maxwell equation \eqref{eq} with $\omega$ as in \eqref{omega1}, which satisfies
	\begin{equation*}
	\|u-u_\text{ans}\|_{H^2(\R^2)}\leq c\,\varepsilon,
	\end{equation*}
	where $u_\text{ans}$ is defined by \eqref{ansatz_phys_var}.
\end{thm}

Before immersing ourselves in the details of the proof, let us make some important remarks.
\begin{remark}
The assumption $\kappa\neq 0$ is needed in the Helmholtz decomposition in Lemma \ref{L:helmh}.
\end{remark}
\begin{remark}
	Note that $\|u_\text{ans}\|_{L^2(\R^2)}=O(1) \ (\eps \to 0)$ since $\|A_j(\eps \cdot)\|_{L^2(\R^2)}=\eps^{-1}\|A_j\|_{L^2(\R^2)}$. Analogously one has $\|u_\text{ans}\|_{H^2(\R^2)}=O(1)$. The next correction term in the asymptotics of the solution $u$ is expected to have the form $\eps^2\sum_{j=1}^NA^{(2)}_j(\eps x)r_j(x)$ with suitable (smooth) functions $A^{(2)}_j$ and $r_j$. It is the correction compensating for the residual at the formal order $O(\eps^3)$ after satisfying the coupled mode equations, see Sec. \ref{Section_approx}. Hence, the expected correction term is $O(\eps)$ in $H^2(\R^2)$ such that the error estimate of Theorem \ref{Thm_main} is expected to be optimal. 
\end{remark}
\begin{remark}
	Theorem \ref{Thm_main} can also be considered as a result on the bifurcation of gap solitons from the zero solution at $\omega=\omega_*$. A sufficient condition is the existence of a $\PT$-symmetric and non-degenerate solution $A\in H^{s_A}(\R^2,\C^N)$ of the effective CME equations. 
\end{remark}

\begin{remark}
	The CMEs in \eqref{CMEs_phys} are a system of coupled nonlinear Schr\"odinger equations and have the same structure as those for stationary gap solitons of the 2D scalar Gross-Pi\-ta\-ev\-skii equation with a periodic potential, see \cite{DPS09,DU}. In \cite{DD} several localized solutions of CMEs \eqref{CMEs_phys} with coefficients determined by Bloch waves of the Maxwell system were found numerically. The current paper does not discuss the existence of localized nontrivial solutions to CMEs. Existence results based on bifurcation theory and variational analysis can be found, e.g., in \cite{LW05,Mandel-2014,Mandel-2015}. 
\end{remark}
\begin{remark}
The $\PT$-symmetry has been extensively studied by the physics community in the recent years, mainly with emphasis on localized solutions, as it serves as a model for a balance between gain and loss in the structure. It has been shown to have a lot of applications, e.g. in Bose-Einstein condensates \cite{KKZ}, non-Hermitian systems \cite{Bender}, quantum mechanics, optics \cite{RMEl}, or surface plasmons polaritons \cite{Oulton, Barton2018}. For a survey on the topic we refer to \cite{El}. Mathematically, the restriction of a fixed point argument to a $\PT$-symmetric (or more generally anti-linearly symmetric) subspace has been used to obtain real nonlinear eigenvalues, see, e.g., \cite{Rubinstein-2010-195,DS,DP,DR}. In particular, in our justification result, such symmetry, assumed on the functions $\epsilon$ and $\chitre$ and then reflected in the band structure, is exploited to remove shift and space invariances in perturbed CMEs. This enables us to invert the linearized operator when working in the symmetric subspace.
\end{remark}
\begin{remark}
	The proof of Theorem \ref{Thm_main} is based on a generalized Lyapunov-Schmidt decomposition in Bloch variables and on fixed point arguments. The CMEs can be seen as the effective bifurcation system of the Lyapunov-Schmidt decomposition. This approach has been used, e.g., for wave packets of the Gross-Pi\-ta\-ev\-skii equation with periodic coefficients in \cite{DPS09,DU,DU-err,DP}. 
 \end{remark}
\begin{remark}
The term ``Coupled Mode Equations'' in the context of asymptotics of wave packets is often used also for a different system, namely for time dependent first order envelope equations. These are derived when the wave packet is built using Bloch waves with nonzero group velocities, see \cite{GWH01,SU01,GMS08,DW20}.
 \end{remark}

\vskip0.2truecm
The rest of the paper is organized as follows. In Sec. \ref{Section:Linear}, after introducing the suitable functional setting, we investigate the linear problem $L^{(E)}u-\omega^2\epsilon u =0$, its spectrum and the Bloch waves. Then, using the Bloch transform, we formulate \eqref{eq} in the Bloch variables. In addition, important regularity estimates on the Bloch eigenfunctions are also established here. Next, precise formulations of our assumptions are given in Sec. \ref{Section_AssMR}. The proof of Theorem \ref{Thm_main} is provided in Sec. \ref{S:NL} and split into several subsections according to our Lyapunov-Schmidt decomposition of the solution. We trim the solution $u$ by rest terms, which are proved to be small enough in Sec. \ref{Section_w0}-\ref{Section_B}, and we finally show that the leading order part is $\varepsilon$-close to our ansatz in Sec. \ref{Section_approx}. The Appendix collects some auxiliary Lemmas which are used in our analysis.

%-----------------------------------------------------
%-----------------------------------------------
\section{Function Spaces, Spectrum, Bloch Transformation and Linear Estimates}\label{Section:Linear}
In this section we firstly investigate the eigenvalue problem
\begin{equation}\label{eq_linear}
\nabla'\times\nabla'\times u=\omega^2\epsilon(x)u, \quad x\in \R^2
\end{equation}
and the corresponding Bloch eigenvalue problem on the periodicity cell
\begin{equation}\label{eval-eq-p}
\nabla_k'\times\nabla_k'\times p=\omega_j^2(k)\epsilon(x)p,\quad x\in \cQ,
\end{equation}
where $p$ is $\Lambda$-periodic, $k\in \R^2$ and (with $\kappa\in \R$ being a fixed parameter)
$$\nabla'_k:=(\pa_1+\ri k_1, \pa_2 +\ri k_2, \ri \kappa)^T.$$
Secondly, we prove estimates on a linear inhomogeneous problem on the periodicity cell. This problem is obtained by applying the Bloch transformation to an inhomogeneous version of \eqref{eq_linear}, which plays a central role in a Banach fixed point iteration for the nonlinear equation in Sec. \ref{S:NL}.

The Bloch transformation and its properties are reviewed subsequently.

%-----------------------------
\subsection{Function Spaces}

We start by defining some function spaces which we use below. Because of the presence of the \textit{curl} operator in the Maxwell system \eqref{Max} we will make use of $H(\mbox{curl})$ spaces with the curl defined using the above gradient $\nabla'$. Let us first define 
$$
\begin{aligned}
L_\#^2(\cQ,\C^3)&:=\{v\in L^2_\text{loc}(\R^2,\C^3): v\ \text{ is $\Lambda$-periodic}\},\\
H_\#^s(\cQ,\C^3)&:=\{v\in H^s_\text{loc}(\R^2,\C^3): v\ \text{ is $\Lambda$-periodic}\}, \ s>0.
\end{aligned}
$$ 
The notation $L^2_\#(\cQ,\C^3), H_\#^s(\cQ,\C^3)$ is chosen to make clear that the elements need to be defined on the periodicity cell $\cQ$ and periodically extendable in an $L^2_\text{loc}$ resp. $H^s_\text{loc}$ fashion onto $\R^2$. Note that for a vector field $u:\cQ\to\C^3$ we define
$$\|u\|^2_{H^s(\cQ)}:=\sum_{j=1}^3\sum_{|\alpha|\leq s}\|D^\alpha u_j\|^2_{L^2(\cQ)}$$
with $\alpha\in\N_0^3$ being the standard multi-index and $s\in\N_0$.

Next we define
$$H_\#(\text{curl},\cQ):=\{v\in L_\#^2(\cQ,\C^3)\,|\,\nabla'\times v\in L^2_\#(\cQ,\C^3)\}$$
and
$$H_\#(\text{curl}^2,\cQ):=\{v\in H_\#(\text{curl},\cQ)\,|\,\nabla'\times\nabla'\times v\in L_\#^2(\cQ,\C^3)\}.$$
We will sometimes use the short notation $H_\#(\text{curl})$ or $H_\#(\text{curl}^2)$.

Note that in the majority of our calculations the gradient $\nabla'$ is replaced by $\nabla_k'$. However, this makes no difference in the definition of the function spaces. Indeed,  because
$$H_\#(\text{curl},\cQ) = \{v\in L_\#^2(\cQ,\C^3)\,|\,\nabla'_k\times v\in L_\#^2(\cQ,\C^3)\}$$
and
$$H_\#(\text{curl}^2,\cQ) = \{v\in H_\#(\text{curl},\cQ)\,|\,\nabla'_k\times\nabla'_k\times v\in L_\#^2(\cQ,\C^3)\}$$
for any $k \in\R^2$, we do not need to define new function spaces for problems involving the gradient $\nabla'_k$.

For later use, we note the identity
\begin{equation*}\label{divk_curlk}
\nabla'_k\cdot\nabla'_k\times u=0\qquad\mbox{for all}\,\,u\in H_\#(\text{curl},\cQ),
\end{equation*}
which can be easily checked.

%-----------------------------
\subsection{Spectral Problem for the \texorpdfstring{$\boldsymbol{H}-$}{H-}field}\label{S:spec-H}
We build our linear theory on the results of \cite{Dauge_et_al} for the spectral problem for the $H$-field
\begin{equation*}\label{eq_linear_H}
\nabla'\times\left(\frac1\epsilon\nabla'\times v\right)=\omega^2 v.
\end{equation*}
It follows from the Bloch theory (see \cite{kuchment93,DLPSW_2011}) that the spectrum of  
$$L^{(H)}: H({\rm curl}_\epsilon^2)\to L^2(\R^2,\C^3), \quad\; L^{(H)}:=\nabla'\times\left(\frac{1}{\epsilon}\nabla'\times \cdot \right),$$
where
$$H({\rm curl}_\epsilon^2):=\bigg\{v\in L^2(\R^2,\C^3)\,|\,\nabla'\times v\in L^2(\R^2,\C^3), \nabla'\times\bigg(\frac1\epsilon\nabla'\times v\bigg)\in L^2(\R^2,\C^3)\bigg\},$$
is obtained as the union (over all $k\in \B$) of the spectra of
$$L_k^{(H)}: H_\#({\rm curl}_\epsilon^2)\to L_\#^2(\cQ,\C^3), \quad\; L^{(H)}_k:=\nabla_k'\times\left(\frac1\epsilon\nabla_k'\times \cdot \right),$$
where 
$$H_\#({\rm curl}_\epsilon^2):=\bigg\{v\in L^2_\#(\cQ,\C^3)\,|\,\nabla_k'\times v\in L^2_\#(\cQ,\C^3), \nabla_k'\times\bigg(\frac1\epsilon\nabla_k'\times v\bigg)\in L^2_\#(\cQ,\C^3)\bigg\}.$$
We emphasize that $L^{(H)}_k$ acts on periodic functions on the periodicity cell $\cQ$.

Let $k\in \B$ be fixed. With the form domain of $L_k^{(H)}$ being
\begin{equation}\label{DomainVk}
V_k:=\{v\in L^2_\#(\cQ,\C^3)\,|\,\nabla_k'\times v\in L^2_\#(\cQ,\C^3), \ \nabla_k'\cdot v =0\},
\end{equation}
the authors of \cite{Dauge_et_al} prove that the spectrum is discrete and satisfies
$$\sigma(L_k^{(H)})=\{\omega_1^2(k),\omega_2^2(k),\dots\}\subset [0,\infty),$$
where
$$\omega_1^2(k)\leq \omega_2^2(k)\leq \dots$$
The corresponding eigenfunctions $(q_j(\cdot,k))_{j\in\N}\subset V_k$ satisfy
\beq\label{E:ev-eq-weak-vj}
a_k(q_j(\cdot,k),\varphi)=\omega_j^2(k)\langle q_j(\cdot,k),\varphi\rangle \qquad \forall \varphi \in V_k,
\eeq
where 
$$a_k(\psi,\varphi):=\int_\cQ \frac1\epsilon\nabla_k'\times\psi\cdot\overline{\nabla_k'\times\varphi}\dd x \quad \text{and} \quad \langle\psi,\varphi\rangle:=(\psi,\varphi)_{L^2(\cQ)}:=\int_\cQ\psi\cdot\overline\varphi\dd x.$$ 
Moreover, they can be chosen $L^2(\cQ)$-orthonormal, i.e.
$$\langle q_i(\cdot,k),q_j(\cdot,k)\rangle =\delta_{ij} \quad \forall i,j\in \N.$$

It follows that for $\lambda$ in the resolvent set, i.e. $\lambda \in \C\setminus \{\omega_1^2(k),\omega_2^2(k),\dots\}$, and $g\in L^2_\#(\cQ,\C^3)$ there is a unique $v\in V_k$ such that
\beq\label{E:veq-inhom-weak}
a_k(v,\varphi)-\lambda\langle v,\varphi\rangle=\langle g,\varphi\rangle \quad\;\; \forall \varphi\in V_k.
\eeq
Moreover, there is a constant $c>0$ such that 
\beq\label{E:v-est-weak}
\|v\|_{H^1(\cQ)}\leq c\,\|g\|_{L^2(\cQ)}.
\eeq
Here the equivalence of the $H({\rm curl})$ and $H^1$-norms on $V_k$ has been used. In fact, the eigenfunctions automatically satisfy the regularity
\beq\label{E:vj-reg}
q_j(\cdot,k)\in H_\#({\rm curl}_\epsilon^2)\cap V_k.
\eeq
To show this, it suffices to prove that \eqref{E:ev-eq-weak-vj} holds for all $\varphi \in C^\infty_c(\cQ)$. Then the weak curl $\nabla_k'\times$ of $\frac{1}{\epsilon}\nabla_k'\times q_j(\cdot,k)$ equals $\omega_j^2(k)q_j(\cdot,k)$, which is in $L_\#^2(\cQ,\C^3)$. Due to the Helmholtz decomposition in Lemma \ref{L:helmh} we have
$$H_\#({\rm curl},\cQ)=V_k \oplus \nabla'_k H_\#^1(\cQ),$$
so, by a density argument, it remains to show that \eqref{E:ev-eq-weak-vj} holds for all $\varphi \in \nabla'_k C^\infty_\#(\overline\cQ)$, where as usual the subscript $\#$ denotes periodicity. Substituting $\varphi = \nabla_k'\psi$ with $\psi\in C^\infty_\#(\overline\cQ)$, we clearly have $\nabla_k'\times \nabla_k'\psi =0$, as well as 
\begin{equation*}
\int_\cQ q_j(\cdot,k)\cdot\overline{\nabla_k'\psi}\dd x=\int_{\partial\cQ}q_j(\cdot,k)\overline\psi\cdot\nu-\int_\cQ\nabla_k'\cdot q_j(\cdot,k)\overline\psi=0,
\end{equation*}
where $\nu:=(\nu_1,\nu_2,0)^\trans$ and $(\nu_1,\nu_2)^\trans$ being a.e. defined as the unit outer normal vector of $\partial\cQ$. Indeed, first $\nabla_k'\cdot q_j(\cdot,k)=0$ because $q_j(\cdot,k)\in V_k$. Second, as $V_k\subset H^1_{loc}(\R^2)$, the boundary term is well-defined and, by periodicity of both $q_j(\cdot,k)$ and $\psi$, the contributions of the boundary integral on opposite sides of the periodicity cell $\cQ$ cancel out.

Due to the regularity in \eqref{E:vj-reg} we conclude
\beq\label{E:vj-eq-L2}
\nabla_k'\times\bigg(\frac1\epsilon\nabla_k'\times q_j(\cdot,k)\bigg) = \omega_j^2(k) q_j(\cdot,k) \quad\;\; \text{in }\;\, L_\#^2(\cQ,\C^3) \;\,\text{for each }j\in\N.
\eeq

For the spectrum of the operator $L^{(H)}$ in $L^2(\R^2)$ one has 
$$\sigma(L^{(H)})=\bigcup_{k\in \B}\sigma(L^{(H)}_k)=\bigcup_{k\in \B,\,n\in \N}\omega_n(k),$$
see \cite{DLPSW_2011}.

%--------------------------------------------
\subsection{Spectral Problem for the \texorpdfstring{$\boldsymbol{E}-$}{E-}field}\label{S:spec-E}
Let $k\in \B$ be fixed. As we show now, for each eigenfunction $q_j$ of \eqref{E:ev-eq-weak-vj} the function (for $\omega_j(k)\neq 0$)
\begin{equation}\label{pj}
p_j(x,k):=\frac\ri{\epsilon(x)\omega_j(k)}\nabla_k'\times q_j(x,k)
\end{equation}
is an $H_\#({\rm curl}^2)$ eigenfunction of the eigenvalue problem for the $E-$field.
Due to \eqref{E:vj-reg} we first have $p_j(\cdot,k)\in H_\#({\rm curl})$. Next, \eqref{E:vj-eq-L2} implies $\nabla_k'\times p_j(\cdot,k)=\ri\omega_j(k)q_j(\cdot,k)\in H_\#({\rm curl})$ such that $p_j(\cdot,k)\in H_\#({\rm curl}^2)$ and
\begin{equation}\label{E:ev-eq-weak-uj}
L^{(E)}_k p_j(\cdot,k):=\nabla_k'\times\nabla_k'\times p_j(\cdot,k)=\epsilon\omega_j^2(k)p_j(\cdot,k)\quad\; \text{in } L^2_\#(\cQ).
\end{equation}

The sequence $(p_j(\cdot,k))_j$ satisfies the orthogonality 
\beq\label{orthog_pn}
\langle p_i(\cdot,k),p_j(\cdot,k)\rangle_\epsilon:=\langle p_i(\cdot,k),\epsilon p_j(\cdot,k)\rangle=\delta_{ij} \quad \forall i,j\in\N
\eeq
because
\begin{equation*}
\begin{split}
\langle p_i(\cdot,k),p_j(\cdot,k)\rangle_\epsilon&=\frac1{\omega_i(k)\omega_j(k)}\int_{\cQ}\frac1\epsilon\nabla_k'\times q_i(\cdot,k)\cdot\overline{\nabla_k'\times q_j(\cdot,k)}\dd x\\
&=\frac1{\omega_i(k)\omega_j(k)}a_k(q_i(\cdot,k),q_j(\cdot,k))=\delta_{ij}.
\end{split}
\end{equation*}

Besides the periodicity in $x$, the functions $p_n, n\in \N$ are quasiperiodic in $k$, namely
\begin{equation*}
p_n(x,k+K)=p_n(x,k)e^{-\ri K\cdot x}\qquad\mbox{for all}\,\,x\in\R^2\,\,\mbox{and}\,\,K\in\Lambda^*.
\end{equation*}

Two symmetries of the eigenfunctions $p_n$ will be used in the analysis. Firstly, because the eigenvalue problem is invariant under the complex conjugation combined with replacing $k$ by $-k$, one sees that $\omega_n(k)=\omega_n(-k)$ for all $n\in\N$, being they real. Similarly, one also deduces that $\overline{p_n(x,k)}$ is an eigenfunction of $L^{(E)}_{-k}$ if and only if $p_n(x,k)$ is an eigenfunction of $L^{(E)}_{k}$. This implies that the eigenfunction $p_n(x,-k)$ can be chosen to agree with $\overline{p_n(x,k)}$ for all $k\in\R^2\setminus\{0\}$. Notice that at $k=0$ the operator $L^{(E)}_0$ is real, so a real eigenfunction can always be chosen. Hence we have
\beq\label{E:sym-minus-k}
p_n(x,-k)=\overline{p_n(x,k)} \qquad \mbox{for all}\,\,x,k\in\R^2, n\in \N.
\eeq
Secondly, if $\epsilon(x)=\epsilon(-x)$ and if $p_n(x,k)$ is an eigenfunction of $L^{(E)}_k$, it is easy to show that $\overline{p_n(-x,k)}$ is an eigenfunction of $L^{(E)}_k$, too. Therefore, if $\omega_n(k)^2$ is a geometrically simple eigenvalue of \eqref{E:ev-eq-weak-vj}, there is always a choice of the phase of the normalized eigenfunction $p_n$ such that the $\PT-$symmetry
\beq\label{E:sym-minus-x}
p_n(-x,k)=\overline{p_n(x,k)} \qquad \mbox{for all}\,\,x\in\R^2
\eeq
holds.
\vskip0.2truecm

The map $k\mapsto\omega_n(k)$ with $\omega_n\geq 0$ is called the $n$\textit{-th eigenvalue} and the map $(k,n)\mapsto\omega_n(k)$ the \textit{band structure}. Clearly, since the spectrum (for each $k$) is given by $\{\omega_1(k)^2, \omega_2(k)^2,\dots\}$, there are also the negative eigenvalues $\omega_{-n}:=-\omega_n$, but they play no role in our analysis. Notice also that the band structure is the same for both operators $L^{(H)}$ and $L^{(E)}$.

%--------------------------------------------
\subsection{Inhomogeneous Linear Equation for the \texorpdfstring{$\boldsymbol{E}-$}{E-}field}

Our asymptotic and nonlinear analysis is performed for the $E-$field and in the fixed point argument we need to solve the inhomogeneous problem
\beq\label{E:u-inhom-f}
L^{(E)}_k u -\omega^2 \epsilon u = f
\eeq
with $\omega^2$ in the resolvent set of $\epsilon^{-1}L^{(E)}_k$, i.e. of $L^{(H)}_k$.
In our application we have $f\in H^2_\#(\cQ)$ and the fixed point argument requires the estimate $\|u\|_{H^2(\cQ)}\leq c\|f\|_{H^2(\cQ)}.$ We prove this estimate next. 
\blem Let $k\in \B$, $\epsilon \in W^{2,\infty}(\cQ)$, $\epsilon^{-1}\in L^\infty(\cQ)$, $\omega^2 \in \C\setminus \{\omega_1^2(k),\omega_2^2(k),\dots\}$, and $f\in H^2_\#(\cQ)$. Then \eqref{E:u-inhom-f} has a unique solution $u\in H_\#({\rm curl}^2)$ such that 
\beq\label{E:u-H2est}
\|u\|_{H^2(\cQ)} \leq c \|f\|_{H^2(\cQ)}
\eeq
holds.
\elem

\begin{remark}
	Note that \eqref{E:u-H2est} is clearly not optimal as an estimate of the solution of \eqref{E:u-inhom-f}. An optimal estimate includes just the $L^2$-norm on the right-hand side. However, as our nonlinear analysis below employs an estimate of the form $\|u\|_{H^2}\leq c\|f\|_{H^2}$, this suboptimality is not of essence.
\end{remark}

\bpf
For $\omega^2\in\C\setminus\{\omega_1^2(k),\omega_2^2(k),\dots\}$ we choose $\omega$ such that e.g. $\arg\omega\in\big(-\frac\pi2,\frac\pi2\big]$. We define first
$$\tilde f:=-\frac\ri\omega\nabla'_k\times \frac{f}{\epsilon}$$
and solve $L^{(H)}_k v -\omega^2 v = \tilde f$ in the weak sense, see \eqref{E:veq-inhom-weak}. Due to \eqref{E:v-est-weak} we get $\|v\|_{H^1(\cQ)}\leq c\|\tilde f\|_{L^2(\cQ)}$. Since $\epsilon \in W^{1,\infty}(\cQ)$ and $\epsilon^{-1}\in L^\infty(\cQ)$, we get
\beq\label{v:est-ftil}
\|v\|_{H^1(\cQ)}\leq c\,\|f\|_{H({\rm curl})}.
\eeq
Moreover, similarly to \eqref{E:vj-reg}, using the Helmholtz decomposition of Lemma \ref{L:helmh}, we get $v\in H_\#({\rm curl}_\epsilon^2)$ and
\beq\label{E:Lv-inhom}
L^{(H)}_k v -\omega^2 v = \tilde f \quad \text{in }L_\#^2(\cQ,\C^3).
\eeq

Next, we set 
\beq\label{E:u-from-v}
u:=\frac{\ri}{\omega \epsilon}\nabla_k'\times v -\frac{1}{\omega^2\epsilon}f.
\eeq
Then $\nabla_k'\times u = \ri \omega v$ and using \eqref{v:est-ftil} as well as the assumptions on $\epsilon$, we have
\beq\label{E:u-Hcurl-est}
\|u\|_{H({\rm curl})}\leq c \|f\|_{H({\rm curl})}.
\eeq

Moreover, since $v\in H_\#({\rm curl})$, we have from $\nabla_k'\times u = \ri \omega v$ also $u\in H_\#({\rm curl}^2)$ and thus, applying $\nabla'_k\times$ to \eqref{E:u-from-v}, we obtain that \eqref{E:u-inhom-f} holds as an equation in $L^2_\#(\cQ)$.

Next, we derive the desired $H^2$-estimate on $u$. We start with $H^1$. Because 
\begin{equation}\label{2.19bis}
\|u\|_{H^1(\cQ)}\leq c\left(\|u\|_{H({\rm curl})}+\|\nabla_k'\cdot u\|_{L^2(\cQ)}\right)
\end{equation}
and because of \eqref{E:u-Hcurl-est} it remains to estimate the divergence. Since $\nabla_k'\cdot u = \frac1\epsilon(\nabla_k'\cdot( \epsilon u) -(\nabla_k' \epsilon) \cdot u)$, from \eqref{E:u-from-v} we infer
\beq\label{E:divk-u}
\nabla_k'\cdot u  = -\frac1\epsilon\left(\frac1\omega\nabla_k'\cdot f-(\nabla_k' \epsilon)\cdot u\right).
\eeq
By $\epsilon\in W^{1,\infty}$ and $\epsilon^{-1}\in L^\infty$, we get then
\begin{equation}\label{2.20bis}
\|\nabla_k'\cdot u\|_{L^2(\cQ)}\leq c\,(\|f\|_{H^1(\cQ)}+\|u\|_{L^2(\cQ)})\leq c\,\|f\|_{H^1(\cQ)},
\end{equation}
where the last inequality holds by \eqref{E:u-Hcurl-est}. Next,
$$
\|u\|_{H^2(\cQ)} \leq c \left(\|u\|_{L^2(\cQ)}+\|\nabla_k'\times u\|_{H^1(\cQ)}+\|\nabla_k'\cdot u\|_{H^1(\cQ)}\right)
$$
and, using  again \eqref{E:u-from-v}, \eqref{E:divk-u}, and $\epsilon\in W^{2,\infty}$ we obtain
$$
\|u\|_{H^2(\cQ)} \leq c \bigg(\|u\|_{H^1(\cQ)}+\bigg\|\nabla_k'\times \bigg(\frac1\epsilon\nabla_k'\times v\bigg)\bigg\|_{H^1(\cQ)}+\|\nabla_k'\times f\|_{H^1(\cQ)}+\|\nabla_k'\cdot f\|_{H^1(\cQ)}\bigg)
$$
The estimate \eqref{E:u-H2est} is finally deduced by \eqref{v:est-ftil},\eqref{E:Lv-inhom} and \eqref{E:u-Hcurl-est},\eqref{2.19bis}, \eqref{2.20bis}. 
\epf

%--------------------------------------------
\subsection{Bloch Transformation}

To take advantage of the fact that the coefficients of our problem \eqref{eq} are periodic, we will work in Bloch variables, i.e. we will employ the Bloch transform to change the problem into a family of problems on the periodicity cell $\cQ$, parametrized by the wave vector $k \in\B$. The above discussion (Sec. \ref{S:spec-H}, \ref{S:spec-E}) implies that the resulting equation has (for each $k$) a linear operator with a discrete spectrum.

The \textit{Bloch transform} $\cT:L^2(\R^2)\to L^2(\B, L^2_\#(\cQ))$ so that $v\mapsto\tv$ and its inverse are formally defined as
\begin{equation*}
\tv(x,k)=(\cT v)(x,k):=\sum_{K\in\Lambda^*}\hv(k+K)e^{\ri K\cdot x},\qquad v(x)=(\cT^{-1}\tv)(x)=\int_\B\tv(x,k)e^{\ri k\cdot x}\dd k
\end{equation*}
for all $x,k\in\R^2$, see e.g. \cite{Schneider98} or \cite[Chap.7]{Optic_book}. For the domain and range of $\cT$ see \eqref{Bloch_isomorphismum}. Here $\hv$ denotes the \textit{Fourier transform} of $v\in L^1(\R^2)$
\begin{equation*}
\hv(k):=\frac1{(2\pi)^2}\int_{\R^2}v(x)e^{-\ri k\cdot x}\dd x,
\end{equation*}
which is extended to $L^2(\R^2)$ functions as usual.

The definition of $\tilde{v}$ yields naturally the periodicity in $x$ and the quasi-periodicity in $k$, i.e.
\begin{equation}\label{properties_period-qperiod}
\begin{split}
\tv(x+R,k)&=\tv(x,k)\qquad\qquad\quad\mbox{for all}\,\,R\in\Lambda, x\in \R^2, k\in \R^2,\\
\tv(x,k+K)&=e^{-\ri K\cdot x}\tv(x,k)\qquad\mbox{for all}\,\,K\in\Lambda^*, x\in \R^2, k\in \R^2.
\end{split}
\end{equation}
Moreover, the product of two functions $f,g\in L^2(\R^2)$ for which also $fg\in L^2(\R^2)$ is transformed by $\cT$ into a convolution of the transformed functions:
\begin{equation}\label{convolution}
(\cT(fg))(x,k)=\int_\B \tf(x,k-l)\tilg(x,l)\dd l=:(\tf\ast_\B\tilg)(x,k),
\end{equation}
where the quasiperiodicity property in \eqref{properties_period-qperiod} is used if $k-l\not\in\B$. For the same reason the convolution in $\B$ can be substituted by a convolution on any shifted Brillouin zone, i.e.
$$(\tf\ast_\B\tilg)(x,k)=\int_{\B+k_*}\tilde f(x,k-l)\tilde g(x,l)\dd l \qquad \forall k_*\in \R^2.$$ 
%(cf.\cite[p.8, first eq.]{DP}})
\noindent If $f$ enjoys periodicity with respect to the same lattice $\Lambda$, then
\begin{equation}\label{Bloch_periodic}
(\cT(fg))(x,k)=f(x)(\cT g)(x,k)
\end{equation}
for all $x\in\R^2$ and $k\in\B$.

%--------------
\paragraph{The function spaces for the Bloch transform.}

Let $H^s(\R^n)$ with $s>0$ be the standard (possibly fractional) Sobolev space. The Bloch transform  
\begin{equation}\label{Bloch_isomorphismum}
\cT:H^s(\R^n,\C)\to \cX_s:=L^2(\B,H^s_\#(\cQ,\C))
\end{equation}
is an isomorphism for $s\geq 0$ \cite{Schneider98,ReedSimon}. The norm in $\cX_s$ is defined as
\begin{equation*}
\|\tu\|_{\cX_s}=\left(\int_\B\|\tu(\cdot,k)\|^2_{H^s(\cQ)}\dd k\right)^{\frac12},
\end{equation*}
where $\cQ$ is an arbitrary interval in $\R^n$ and $\B$ the corresponding reciprocal periodicity cell. We work, of course, in $n=2$ with $\cQ$ and $\B$ as defined in Sec. \ref{S:intro}. For vector-valued functions $u\in H^s(\R^n,\C^m)$, $m\in\N$ the transform $\cT$ is defined componentwise and the space $\cX_s$ is $\cX_s:=L^2(\B,H^s_\#(\cQ,\C^m))$ with the norm $\|f\|_{\cX_s}:=\max_{1\leq j\leq m}\|f_j\|_{\cX_s}$.

Note that due to the quasi-periodicity of $\tilde u$ in $k$, the $\cX_s$-norm is equivalent to 
$$\left(\int_{\B+k_*}\|\tu(\cdot,k)\|^2_{H^s(\cQ)}\right)^{\frac12}$$
for any $k_*\in \R^2.$ We take advantage of this property in our estimates below.

Because of the polynomial nonlinearity in \eqref{eq} and our approach employing a fixed point argument we require our function space to have the algebra property with respect to the pointwise multiplication. We recall that if $s>n/2$, the Sobolev space $H^s(\R^n)$ enjoys this property. Moreover, it embeds into the space of bounded and continuous functions decaying to $0$ at $\infty$. 

In the Bloch variables, where multiplication is transformed into a convolution, we need the algebra property with respect to the convolution. Combining the algebra property of $H^s(\R^n), s>n/2,$ and \eqref{convolution}, we get the following algebra property for our working space $\cX_s$:
\begin{equation}\label{algebra:X_2}
\|\tf\ast_\B\tilg\|_{\cX_s}\leq c\,\|\tf\|_{\cX_s}\|\tilg\|_{\cX_s}\qquad\mbox{for any}\,\,\tf,\tilg\in\cX_s\quad\mbox{if}\,\,s>n/2.
\end{equation}
We introduce also the weighted spaces $L^2_s(\R^n)$ defined as
\begin{equation}\label{Ls}
L^2_s(\R^n):=\big\{f\in L^2(\R^n)\,|\,\|f\|_{L^2_s(\R^n)}^2:=\scaleobj{.75}{\int_{\R^n}}(1+|x|)^{2s}|f(x)|^2\dd x<\infty\big\}.
\end{equation}
Recall that the Fourier transform is an isomorphism from $H^s(\R^n)$ to $L^2_s(\R^n)$ for $s\geq 0$.

%---------------------------------------------------
%---------------------------------------------------

\section{Assumptions}\label{Section_AssMR}
We start with the following basic assumptions on the coefficients and on the band structure.
\begin{enumerate}
\item[(A1)] $\epsilon:\R^2\to\R$ and $\chi^{(3)}:\R^2\to\R^{3\times3\times3\times3}$ are $\Lambda$-periodic and real-valued and $\epsilon>0$;
\item[(A2)] the spectrum $\cup_{n\in \N,k\in \B}\{\omega_n(k)\}\subset \R$ possesses a gap;
\item[(A3)] the points $k^{(1)}, \dots, k^{(N)}\in\B$ are distinct and constitute the level set $W_{\omega_*}\subset \B$ of one of the gap edges, denoted by $\omega_*$ and the eigenvalues at the level $\omega_*$ are all geometrically simple, i.e.
\begin{equation*}
\omega_*=\omega_n(k),\; n \in \N,\; k \in \B\quad\Rightarrow\quad k\in\{k^{(1)}, \dots, k^{(N)}\},
\end{equation*}
and
\begin{equation*}
\dim\ker\big(L^{(E)}_{\kj}-\epsilon\omega^2_*I\big)=1.
\end{equation*}
Hence, due to the monotonicity $\omega_n(k) \leq \omega_{n+1}(k)$, we have 
$$\exists n_*\in \N: \ \omega_{n_*}(k^{(j)})=\omega_* \ \forall j\in\{1,\dots,N\}.$$
\item[(A4)] the eigenvalue $\omega_{n_*}$ is twice continuously differentiable at $\kj$ and $\nabla^2\omega_{n_*}(k^{(j)})$, the Hessian of $\omega_{n_*}$ at $k=\kj$, is definite for each $j\in\{1,\dots,N\}$.
\end{enumerate}
The formal asymptotic analysis of gap solitons in \cite{DD}
used assumptions (A1),(A2), and (A4). In assumption (A3) multiple eigenvalues were allowed at the points $k^{(j)},j=1,\dots,N$. Here, in order to be able to prove symmetries of the Bloch waves at $k^{(j)},j=1,\dots,N$, which are needed in the restriction of the nonlinear problem to a symmetric subspace, we require the geometric simpleness. The unique (up to a phase factor $e^{\ri\alpha}$) normalized eigenfunction at $k=k^{(j)}$ and $\omega =\omega_*$ is denoted by $p_{n_*}(\cdot,k^{(j)})$.

Note that according to the mathematical folklore, simple eigenvalues depend smoothly on the coefficients of the operator. Nevertheless, we are not aware of an existing result applicable to our operator $L_k^{(H)}$ or $L_k^{(E)}$ such that we assume the $C^2$-regularity in (A4). In addition our proof requires the Lipschitz continuity throughout $\B$, which we prove in the Appendix, see Lemma \ref{eigv_Lip}.

Clearly, $\omega_*$ must be the maximum or minimum of the eigenvalue $\omega_{n_*}$. Hence, based on (A4), $\nabla^2\omega_{n_*}(k^{(j)})$ is either positive definite for all $j$ or negative definite for all $j$.
Note that the assumption that $\omega_*$ is attained at  the points $k=\kj$, $j=1,\dots,N$ by the same eigenvalue $\omega_{n_*}$ is in accordance with the numbering of the eigenvalues $\omega_n(k)$ at each $k$ according to the magnitude.

\begin{remark}
	Assumption (A3) seems relatively restrictive as it does not allow for $\{k^{(1)}, \dots, k^{(N)}\}$ to be a proper subset of the level set $W_{\omega_*}$. We need this assumption to estimate the correction term, which is supported (in the wave-number $k$) away from small neighbourhoods of the points $\{k^{(1)}, \dots, k^{(N)}\}$, i.e. away from the support of the main contribution of the solution. The support of the correction term must not intersect $W_{\omega_*}$ because otherwise $(\omega_{n_*}(k)-\omega_*)^{-1}$ blows up on this support. Note that this can be contrasted against the case of the bifurcation of nonlinear Bloch waves in \cite{DU16}, where $\{k^{(1)}, \dots, k^{(N)}\}$ can be a proper subset of $W_{\omega_*}$ provided the $k-$points generated by (iterations of) the nonlinearity, i.e. the points 
	$$k\in S_3^n(\{k^{(1)}, \dots, k^{(N)}\}) \text{ for some }n\in \N,$$ 
	where 
	$$S_3(\{k^{(1)}, \dots, k^{(N)}\}):=\{k^{(\alpha)}+k^{(\beta)}-k^{(\gamma)}: \alpha,\beta,\gamma \in\{1,\dots,N\}\},$$ 
	lie outside the level set. Unlike in \cite{DU16} the $k-$support of the leading order term of the gap solitons contains whole neighbourhoods of the points $\{k^{(1)}, \dots, k^{(N)}\}$ (and not isolated points) such that iterations of $S_3$ applied to the union of these neighbourhoods generate all $k\in \B$. 
	
	It is however possible for some components of the CME-solutions to be zero, i.e. $A_{m_1}= \dots =A_{m_M}=0$ for some $1\leq m_1,\dots,m_M\leq N$ (with $M<N)$. This can happen only if the CMES are consistent with the reduction to the components $A_k, k\in \{1,\dots,N\}\setminus \{m_1,\dots,m_M\}$ or equivalently if 
	$$S_3(\{k^{(m_1)},\dots, k^{(m_M)}\})\cap \{k^{(1)}, \dots, k^{(N)}\} = \{k^{(m_1)},\dots, k^{(m_M)}\}.$$ 
	In that sense assumption (A3) is effectively the same as assuming that $\{k^{(1)}, \dots, k^{(N)}\}$ is a consistent subset (in the above sense) of the level set $W_{\omega_*}$ and that $W_{\omega_*}$ is finite.
\end{remark}

To rigorously justify the formal approximation via \eqref{ansatz_phys_var}, we need to assume the following additional conditions: 
\begin{enumerate}
	\item[(A5)] the material functions $\epsilon$ and $\chi^{(3)}$ satisfy $\epsilon\in W^{2,\infty}(\R^2)$, $\epsilon^{-1}\in L^\infty(\R^2)$, $\chi^{(3)}\in H^2_\text{loc}(\R^2)$;
	\item[(A6)] symmetry of the material: $\epsilon(x)=\epsilon(-x)$, $\chi^{(3)}(x)=\chi^{(3)}(-x)$ for all $x\in \R^2$;
\item[(A7)] the eigenvalue $\omega_{n_*}(k)$ is geometrically simple for almost all (w.r.t. the Lebesgue measure) $k\in \B$;
\end{enumerate}
Note that assumption (A7) allows for the touching of eigenvalue graphs $(k,\omega_{n_*}(k))$ and $(k,\omega_{m}(k))$ with $m\neq n_*$ as long as they touch along a curve; which is the canonical situation. This curve may include the points $k^{(1)}, \dots, k^{(N)}$, see assumption (A3).

Under the above assumptions Theorem \ref{Thm_main} justifies the use of the effective amplitude equations \eqref{CMEs_phys} to determine the envelopes $A_j$ in the ansatz \eqref{ansatz_phys_var} and constitutes the main result of the paper. It uses the following definition.
\bdefi\label{D:nondegen}
A solution $A_*\in L^2(\R^2)^N$ of \eqref{CMEs_phys}, denoted by $\mathcal{G}(A)=0$, is called \textit{non-degenerate} if the kernel of the Jacobian of $\mathcal{G}$ evaluated at $A_*$ is only three dimensional as generated by the two spatial shift invariances and the complex phase invariance of the CMEs, i.e.
$$\mbox{ker}\begin{pmatrix} \pa_{A_R}\mathcal{G}_R(A_*) & \pa_{A_I}\mathcal{G}_R(A_*)\\
\pa_{A_R}\mathcal{G}_I(A_*) & \pa_{A_I}\mathcal{G}_I(A_*)
\end{pmatrix} = \mbox{span}\left\{\pa_{y_1} \begin{pmatrix}A_{*,R}\\A_{*,I}\end{pmatrix}, \pa_{y_2} \begin{pmatrix}A_{*,R}\\ A_{*,I}\end{pmatrix}, \begin{pmatrix}-A_{*,I}\\ A_{*,R}\end{pmatrix}\right\},$$
where $A_{*,R}:=\mbox{Re}(A_*), A_{*,I}:=\mbox{Im}(A_*)$ and analogously for the other variables and functions.
\edefi

Assumptions (A6), (A7) are used to remove invariances (and thus eliminate non-trivial elements of the kernel) in a perturbed CME-problem by restricting to a symmetric subspace. This perturbed system is obtained in the justification analysis. The symmetric subspace is defined by the $\PT$-symmetry, i.e.
\begin{equation*}\label{PTsymm}
A(x)=\overline{A(-x)}\qquad\mbox{for all}\,\,x\in\R^2.
\end{equation*}
In this subspace the CMEs no longer possess the invariances wrt. the spatial shift and the complex phase. Hence, under the non-degeneracy condition, the linearized operator of the perturbed CME system is invertible. Note that other symmetric subspaces can be used to eliminate the kernel, see \cite{DU}.

Moreover, the evenness of $\epsilon$ and $\chi^{(3)}$ implies that the coefficients $I_{\alpha,\beta,\gamma}^{\,j}$ are real as explained at the end of Sec. \ref{S:perturbedCME}. 

%-----------------------------------------------
%-----------------------------------------------
\section{Proof of Theorem \ref{Thm_main}} \label{S:NL}
From now on, the bifurcation parameter $\omega$ is chosen to lie in the spectral gap in a $\cO(\varepsilon^2)$-vicinity of the edge $\omega_*$, i.e. 
\begin{equation}\label{omega}
\omega=\omega_*+\varepsilon^2\Omega,
\end{equation}
where $\Omega=\pm1$, the sign being determined by the condition that $\omega$ shall lie in the gap. Hence, $\Omega=\pm 1$ if $\omega_*$ is the bottom/top edge of a spectral gap, respectively. 

\subsection{Lyapunov-Schmidt decomposition}
We study the problem in Bloch variables in the space $\cX_2$: applying the Bloch transform $\cT$ to \eqref{eq}, we get
\begin{equation}\label{eq_Bloch}
L_k\tu(x,k):=\nabla'_k\times\nabla'_k\times\tu(x,k)-\omega^2\epsilon(x)\tu(x,k)=\omega^2\tF(\tu)(x,k),
\end{equation}
where
\begin{equation}\label{F_Bloch}
\tF_d(\tu)=\sum_{a,b,c=1}^3\underline{\chi}^{(3)}_{a,b,c,d}\tu_a\astB\tu_b \astB \tbu_c.
\end{equation}
Here properties \eqref{convolution} and \eqref{Bloch_periodic} have been used. Recall that $\omega$ is fixed as in \eqref{omega}. Note that in \eqref{F_Bloch} the double convolution equals $(\tu_a\astB\tu_b)\astB\tbu_c=\tu_a\astB(\tu_b\astB\tbu_c)$.

Note also that below $f\astB g$ is understood componentwise for scalar $f$ and vector valued $g$.

\vskip0.2truecm

Theorem \ref{Thm_main} claims that the solution can be approximated by a modulated sum of the Bloch eigenfunction $p_{n_*}(\cdot,k)$ at the chosen points $k^{(1)}, \dots, k^{(N)}$. Therefore, we decompose $\tu$ into the part $\tv$ corresponding to the eigenfunction $p_{n_*}$ and the rest $\tw$. Next, $\tv$ is once more split into a first term which incorporates the behaviour in the vicinity of the points $k^{(1)}, \dots, k^{(N)}$ and a rest. To this end, we first introduce some projections on $L^2(\cQ)$ which take into account the presence of the potential $\epsilon(\cdot)$.

\paragraph{Projections}
Let $P_k$ denote the standard $L^2(\cQ)$-projection onto the mode $p_{n_*}(\cdot,k),$ i.e. for $f\in L^2(\cQ)$
\begin{equation*}
(P_kf)(\cdot,k):=\langle f,p_{n_*}(\cdot,k)\rangle\,p_{n_*}(\cdot,k),
\end{equation*}
and let $Q_k:=I-P_k$ be its $L^2$-orthogonal projection. As the normalization of the mode $p_{n_*}(\cdot,k)$ holds in the $L^2$-norm weighted by the periodic potential $\epsilon(\cdot)$ (see \eqref{orthog_pn}), we also introduce
\begin{equation*}
(P_k^\epsilon f)(\cdot,k):=\langle f,\epsilon(\cdot)p_{n_*}(\cdot,k)\rangle\,p_{n_*}(\cdot,k)\quad\,\,\mbox{and}\quad\,\, Q_k^\epsilon:=I-P_k^\epsilon
\end{equation*}
as well as
\begin{equation*}
\ePk:=\epsilon(\cdot)P_k\quad\,\,\mbox{and}\quad\,\,\eQk:=I-\ePk=I-\epsilon(\cdot)P_k.
\end{equation*}
\begin{lem}\label{Lemma_proj_orthog}
	$\Pke,\Qke,\ePk,\eQk$ are projections in $L^2(\cQ)$ for which the following orthogonality conditions hold:
	\begin{enumerate}[i)]
		\item $\Pke L^2\perp_{L^2}\eQk L^2$,
		\item $\ePk L^2\perp_{L^2}\Qke L^2$,
		\item $\Pke L^2\perp_{L^2_\epsilon}\Qke L^2$,
		\item $\ePk L^2\perp_{L^2_{\epsilon^{-1}}}\eQk L^2$,
	\end{enumerate}
	where $L^2$ stands for $L^2(\cQ)$ and $L^2_w$ is the weighted $L^2(\cQ)$ by the weight $w(\cdot)$, i.e. $f\perp_{L^2_w}g$ means $\int_\cQ wf\cdot\overline g\dd x=0$.
\end{lem}
\begin{proof}
	We prove just (\textit{i}), the proof of the claims (\textit{ii})-(\textit{iv}) being similar.
	
	Let $f\in\eQk L^2(\cQ)$, that is $f\in L^2(\cQ)$ such that $\ePk v=0$. Hence $\langle f, p_{n_*}(\cdot,k)\rangle=0$. Let moreover $g\in\Pke L^2(\cQ)$, i.e. $g=\Pke g$. Then,
	\begin{equation*}
	\langle f,g\rangle=\big\langle f,\langle g,\epsilon(\cdot)p_{n_*}(\cdot,k)\rangle\,p_{n_*}(\cdot,k)\big\rangle=\langle g,\epsilon(\cdot)p_{n_*}(\cdot,k)\rangle\langle f,p_{n_*}(\cdot,k)\rangle=0.
	\end{equation*}
\end{proof}

In the sequel the operator $L_k$ in \eqref{eq_Bloch} needs to be inverted with the inverse bounded independently of $\eps$. Recall that for $\omega=\omega_*$ the kernel of $L_k$ is non-trivial at $k\in W_{\omega_*}$, i.e. at $k\in\{k^{(1)},\dots,k^{(N)}\}$, cf. (A3). For $\omega=\omega_*+\eps^2\Omega$ the bound on the inverse explodes as $\eps\to 0$. The Lyapunov-Schmidt reduction based on the projections introduced above decomposes the problem into a critical and a regular part.
In particular, the projections $\ePk$ and $\Pke$ are onto the set of modes, the eigenfunctions of which attain in $\B$ the ``critical'' value $\omega_*$, cf. (A3). This means, we expect that the complementary projections produce an operator with the inverse bounded independently of $\eps$. This is what we prove in the following result.

\begin{lem}\label{invertibility}
	There exists $\eps_0>0$ such that for all $\eps\in (-\eps_0,\eps_0)$ and $\Omega=\pm 1$ the linear operator $\cL_k:=\eQk L_k\Qke\,\,:\,\,H_\#(\text{curl}^2,\cQ)\to\eQk L^2_\#(\cQ)$ with $\omega=\omega_*+\eps^2\Omega$ is invertible on $\Qke H_\#(\text{curl}^2,\cQ)$ and
	\begin{equation}\label{invert_regularity}
	\|\cL_k^{-1}\|_{\eQk L^2_\#\to\Qke H_\#(\text{curl}^2)}\leq C_\cL,
	\end{equation}
	where the constant $C_\cL$ is independent of $\varepsilon, \Omega,$ and $k$.
\end{lem}
\begin{remark}
	We point out that the operator $\cL_k$ depends on $\varepsilon$ via the factor $\omega^2$ in $L_k$.
\end{remark}
\begin{proof}
	\textbf{Step 1.} \textit{$\cL_k$ is injective}. By linearity of $\cL_k$ it is equivalent to show that if $v\in\Qke H_\#(\text{curl}^2,\cQ)$ such that $\cL_k v=0$, then $v=0$. Let thus $v\in H_\#(\text{curl}^2,\cQ)$ be such that $\Qke v=v$. $\cL_k v=0$ means that
	$$L_k\Qke v=L_k v\in\ker\eQk=\spann\{\epsilon p_{n_*}(\cdot,k)\}.$$
	In addition, the assumption $\Qke v=v$ implies $\Pke v =0,$ i.e. $\langle v,\epsilon p_{n_*}(\cdot,k)\rangle =0$. Hence
	\begin{equation*}
	\begin{split}
	L_k v&=\epsilon\langle L_kv,p_{n_*}(\cdot,k)\rangle p_{n_*}(\cdot,k)=\epsilon\langle v,L_kp_{n_*}(\cdot,k)\rangle p_{n_*}(\cdot,k)\\
	&=\epsilon (\omega_{n_*}^2(k)-\omega_*^2)\langle v,\epsilon p_{n_*}(\cdot,k)\rangle p_{n_*}(\cdot,k)=0.
	\end{split}
	\end{equation*}
	This implies $v=0$ because $\omega\not\in\sigma(L)=\bigcup_{k\in\B}\sigma(L_k)$.
	\vskip0.2truecm
	\textbf{Step 2.} \textit{$\cL_k$ is surjective}. The aim is to show that for any  $f\in\eQk L^2_\#(\cQ)$ there exists $u\in \Qke H_\#(\text{curl}^2,\cQ)$ such that $\cL_k u=f$.
	
	First, notice that $f\in \eQk L^2_\#(\cQ)$ if and only if $f\in L^2_\#(\cQ)$ and $\langle f,p_{n_*}(\cdot,k)\rangle=0$.
	
	By the closed range theorem the equation $\cL_k u=f$ is solvable in the domain of $\cL_k$, i.e. in $H_\#(\text{curl}^2,\cQ)$, if and only if $f\perp\ker\big(\cL_k\big)$. Here we are using that the operator $\cL_k$ is self-adjoint and we postpone the proof to the subsequent Lemma \ref{LemmaSA}. Let thus $v\in\ker\big(\!\eQk L_k\Qke\big)$, i.e. $v=\ePk L_k\Qke v$, then there holds
	\begin{equation*}
	\langle f,v\rangle=\langle \eQk f,\ePk L_k\Qke v\rangle=\langle\ePk L_k\eQk f,\Qke v\rangle,
	\end{equation*}
	where the symmetry of $\ePk L_k$ is shown in Lemma \ref{LemmaSA}. Noticing that by Lemma \ref{Lemma_proj_orthog}(ii) the spaces $\ePk L^2$ and $\Qke L^2$ are $L^2$-orthogonal, we deduce $\langle f,v\rangle=0$, i.e. $f\perp\ker\big(\cL_k\big)$. The sought function is then $\Qke u$, with $u$ given by the closed range theorem.
	
	\vskip0.2truecm
	\textbf{Step 3.} \textit{The estimate \eqref{invert_regularity}}. Recall first that for a linear self-adjoint operator $\cA$ acting on a Hilbert space, the well-known estimate
	$$\|\cA^{-1}\|\leq\dist(0,\sigma(\cA))^{-1}$$
	holds, where $\sigma(\cA)$ denotes its spectrum. The self-adjointness of $\eQk L_k\Qke$ is shown in Lemma \ref{LemmaSA}(ii), so we need to bound $\dist(0,\sigma(\eQk L_k\Qke))^{-1}$.
	
 For a fixed $\omega$ the spectrum of $L_k$ is given by $\cup_{j\in \N}\{\omega_j^2(k)-\omega^2\}$ as shown in Sec. \ref{S:spec-H}-\ref{S:spec-E}. The application of the projections $\eQk$ and $\Qke$ yields $\sigma(\cL_k)=\bigcup_{j\in\N\setminus\{n_*\}}\{\omega_j^2(k)-\omega^2\}$. By our assumptions on the band structure (A3) we infer that each of the remaining eigenvalues has some positive distance to $\omega_*$, hence for all $j\not=n_*$ we have $\inf_{k\in\B}|\omega_j^2(k)-\omega_*^2|=:m_j>0$. Since the map $j\mapsto\omega_j(k)$ is increasing for every $k\in\B$ fixed, then $m:=\min_{j\not=n_1,\dots,n_N}m_j>0$ is well-defined and we infer 
	\begin{equation*}
	\begin{split}
	\|(\eQk L_k\Qke)^{-1}\|_{\eQk L^2_\#\to\Qke H_\#(\text{curl}^2)}&\leq\dist(0,\sigma(\eQk L_k\Qke))^{-1}\\
	&\leq2\big(\min_{j\not=n_1,\dots,n_N}	\inf_{k\in\B} |\omega_j^2(k)-\omega^2|\big)^{-1}=\frac2m=:C_\cL
	\end{split}
	\end{equation*}
	for all $\varepsilon$ small enough.
\end{proof}

\begin{remark}
	Notice that without the projections estimate \eqref{invert_regularity} does not hold even if $\Omega \in \{-1,1\}$ is chosen such that $\omega=\omega_*+\eps^2\Omega$ lies outside the spectrum. The operator $L_k$ would be invertible but by \eqref{omega} we would only get $\dist(0,\sigma(L_k))=O(\varepsilon^2)$, whence $\|L_k^{-1}\|$ would not be bounded uniformly in $\varepsilon$.
\end{remark}

\begin{lem}\label{LemmaSA}
	\begin{enumerate}
		\item[(i)] The operator $\ePk L_k:H_\#({\rm curl}^2,\cQ)\to L^2_\#(\cQ)$ is symmetric.
		\item[(ii)]The operator $\cL_k:H_\#({\rm curl}^2,\cQ)\to \eQk L^2_\#(\cQ)$ is self-adjoint.
	\end{enumerate}
\end{lem}
\begin{proof}
	\textit{(i)} Let $v,w\in H_\#(\text{curl}^2,\cQ)$, then
	\begin{equation*}
	\begin{split}
	\langle\ePk L_k v,w\rangle&=\langle P_kL_kv,\epsilon w\rangle=\langle L_kv,p_{n_*}(\cdot,k)\rangle\langle p_{n_*}(\cdot,k),\epsilon w\rangle=\langle v,L_kp_{n_*}(\cdot,k)\rangle\langle\epsilon p_{n_*}(\cdot,k),w\rangle\\
	&=\langle v,\epsilon p_{n_*}(\cdot,k)\rangle\langle \omega_{n_j}(k_*)^2\epsilon p_{n_*}(\cdot,k),w\rangle=\langle v,\epsilon p_{n_*}(\cdot,k)\rangle\langle L_k p_{n_*}(\cdot,k),w\rangle\\
	&
	=\langle v,\epsilon p_{n_j}(\cdot,k)\rangle\langle p_{n_*}(\cdot,k),L_k w\rangle
	=\langle v,\epsilon p_{n_*}(\cdot,k)\rangle\overline{\langle L_k w,p_{n_*}(\cdot,k)\rangle}=\langle v,\ePk L_kw\rangle.
	\end{split}
	\end{equation*}
	\textit{(ii)} First we claim that for $v\in H_\#(\text{curl}^2,\cQ)$ one can rewrite $\cL_kv$ as
	\begin{equation}\label{cL_k_rewritten}
	\cL_kv=L_kv-(\omega_*^2-\omega^2)\langle v,\epsilon p_{n_*}(\cdot,k)\rangle\epsilon p_{n_*}(\cdot,k).
	\end{equation}
	Indeed, recalling the expression of the projections $\eQk$ and $\Qke$, one finds
	\begin{equation*}
	\begin{split}
	\cL_kv&=\eQk L_k\Qke v=L_k\Qke v-\langle L_k\Qke v,p_{n_*}(\cdot,k)\rangle\epsilon p_{n_*}(\cdot,k)\\
	&=L_k\left(v-\langle v,\epsilon p_{n_*}(\cdot,k)\rangle p_{n_*}(\cdot,k)\right)-\langle\left(v-\langle v,\epsilon p_{n_*}(\cdot,k)\rangle p_{n_*}(\cdot,k)\right),L_kp_{n_*}(\cdot,k)\rangle\epsilon p_{n_*}(\cdot,k)\\
	&=L_kv-\langle v,\epsilon p_{n_*}(\cdot,k)\rangle L_kp_{n_*}(\cdot,k)-\langle v,L_kp_{n_*}(\cdot,k)\rangle\epsilon p_{n_*}(\cdot,k)\\
	&\quad+\langle v,\epsilon p_{n_*}(\cdot,k)\rangle\langle p_{n_*}(\cdot,k),L_kp_{n_*}(\cdot,k)\rangle \epsilon p_{n_*}(\cdot,k),
	\end{split}
	\end{equation*}
	having used the symmetry of $L_k$ on $H_\#(\text{curl}^2,\cQ)$. Identity \eqref{cL_k_rewritten} is then proven recalling that $p_{n_*}(\cdot,k)$ are normalised as in \eqref{orthog_pn} and $L_kp_{n_*}(\cdot,k)=(\omega_*^2-\omega^2)\epsilon p_{n_*}(\cdot,k)$.
	
	Next, we show that $D(\cL_k^*)\subset H_\#(\text{curl}^2,\cQ)$, where $\cL_k^*$ denotes the adjoint operator of $\cL_k$. Let $\phi\in D(\cL_k^*)$, then there exists $\eta\in L^2_\#(\cQ)$ so that $\langle\cL_kv,\phi\rangle=\langle v,\eta\rangle$ for all $v\in D(\cL_k)=H_\#(\text{curl}^2,\cQ)$. Using \eqref{cL_k_rewritten} and reorganising the terms, one obtains
	\begin{equation}\label{Lk_adj}
	\langle L_kv,\phi\rangle=\langle v,\eta+(\omega_*^2-\omega^2)\langle\phi,\epsilon p_{n_*}(\cdot,k)\rangle\epsilon p_{n_*}(\cdot,k)\rangle=:\langle v,\tilde\eta\rangle
	\end{equation}
	with $\tilde\eta\in L^2_\#(\cQ)$. This means that $\phi\in D(L_k^*)=D(L_k)=H_\#(\text{curl}^2,\cQ)$ since $L_k$ is self-adjoint \cite{Dauge_et_al}.
	
	Finally, for $\phi\in H_\#(\text{curl}^2,\cQ)$ one has $\langle L_kv,\phi\rangle=\langle v,L_k\phi\rangle$, therefore from \eqref{cL_k_rewritten} and \eqref{Lk_adj} one infers
	\begin{equation*}
	\begin{split}
	\langle\cL_kv,\phi\rangle&=\langle v,L_k\phi\rangle-(\omega_*^2-\omega^2)\langle v,\epsilon p_{n_*}(\cdot,k)\rangle\langle\epsilon p_{n_*}(\cdot,k),\phi\rangle\\
	&=\langle v,L_k\phi-(\omega_*^2-\omega^2)\langle\phi,\epsilon p_{n_*}(\cdot,k)\rangle\epsilon p_{n_*}(\cdot,k)\rangle=\langle v,\cL_k\phi\rangle.
	\end{split}
	\end{equation*}
	This implies $\phi\in D(\cL_k^*)$ and that $\cL_k$ is symmetric. Having proved that $D(\cL_k^*)=H_\#(\text{curl}^2,\cQ)=D(\cL_k)$ and the symmetry, we conclude that $\cL_k$ is self-adjoint.
\end{proof}

\paragraph{Decomposition of the solution}

We decompose the solution $\tu$ of \eqref{eq_Bloch} using the above projections as 
\begin{equation}\label{u}
\tu(x,k)=\tv(x,k)+\tw(x,k),
\end{equation}
where
\begin{equation*}
\tv(x,k):=\Pke\tu(x,k)=\langle \tu(\cdot,k),\epsilon(\cdot)p_{n_*}(\cdot,k)\rangle\,p_{n_*}(x,k)=:U(k)p_{n_*}(x,k)
\end{equation*}
and
\begin{equation*}
\tw(x,k):=\Qke\tu(x,k),
\end{equation*}
and where we have defined
\begin{equation*}\label{Uj}
U(k):=\langle\tu(\cdot,k),\epsilon(\cdot)p_{n_*}(\cdot,k)\rangle.
\end{equation*}
We note that $U$ is $\Lambda$-periodic because $\tu(x,\cdot)$ and $p_{n_*}(x,\cdot)$ are quasiperiodic. 

Now we project suitably equation \eqref{eq_Bloch} and find an equivalent system of two equations, the linear part of which is decoupled. On the one hand, applying $P_k$ to \eqref{eq_Bloch}, we find
\begin{equation*}
\begin{split}
\omega^2\langle\tF(\tu)(\cdot,k),p_{n_*}(\cdot,k)\rangle&=\langle\tu(\cdot,k),\big(\nabla'_k\times\nabla'_k\times-\omega^2\epsilon(\cdot)\big)p_{n_*}(\cdot,k)\rangle\\
&=\langle\tu(\cdot,k),\big(\omega_{n_*}^2(k)-\omega^2\big)\epsilon(\cdot)p_{n_*}(\cdot,k)\rangle.
\end{split}
\end{equation*}
By the definition of $U$ we have
\begin{equation}\label{v_Eq}
\big(\omega_{n_*}^2(k)-\omega^2\big)U(k)=\omega^2\langle\tF(\tu)(\cdot,k),p_{n_*}(\cdot,k)\rangle.
\end{equation}
On the other hand we get
\begin{equation}\label{w_Eq}
\eQk L_k\Qke\tw(x,k)=\omega^2\eQk\tF(\tu)(x,k).
\end{equation}
Indeed, $L_k\tu(x,k)=L_k\left(\tv(x,k)+\tw(x,k)\right)$, where $\tw(x,k)=\Qke\tw(x,k)$ and
\begin{equation*}
\begin{split}
\eQk L_k\tv(x,k)&=\eQk\left(U(k)\left(\omega_{n_*}^2(k)-\omega^2\right)\epsilon(\cdot)p_{n_*}(\cdot,k)\right)\\
&=U(k)\left(\omega_{n_*}^2(k)-\omega^2\right)\left(I-\epsilon(\cdot)P_k\right)\left(\epsilon(\cdot)p_{n_*}(\cdot,k)\right)\\
&=U(k)\left(\omega_{n_*}^2(k)-\omega^2\right)\left(\epsilon(\cdot)p_{n_*}(\cdot,k)-\epsilon(\cdot)\langle\epsilon(\cdot)p_{n_*}(\cdot,k),p_{n_*}(\cdot,k)\rangle p_{n_*}(\cdot,k)\right)\\
&=0,
\end{split}
\end{equation*}
by the normalization \eqref{orthog_pn} of the Bloch eigenfunctions.

Next, we decompose further $\tw$ into
\begin{equation*}
\tw(x,k)=\tw_0(x,k)+\tw_R(x,k),
\end{equation*}
where $\tw_0$ and $\tw_R$ solve the equations
\begin{equation}\label{w0_eq}
\eQk L_k\Qke\tw_0(x,k)=\omega^2\eQk\tF(\tv)(x,k),
\end{equation}
\begin{equation}\label{wR_eq}
\eQk L_k\Qke\tw_R(x,k)=\omega^2\eQk(\tF(\tu)-\tF(\tv))(x,k).
\end{equation}

The system \eqref{v_Eq}, \eqref{w0_eq}, and \eqref{wR_eq} is an equivalent reformulation of equation \eqref{eq_Bloch}. We search for a solution $\tu$ (represented by the variables $U,\tw_0,\tw_R$) which is close to the Bloch transformation of the formal ansatz $u_\text{ans}$. In detail, for the sought solution the components $\tw_0$ and $\tw_R$ are small and $U$ is concentrated at the points $\kj, j=1,\dots,N$ and near the concentration points $\kj$ it approximates $\hat{A}_j$, where $A$ is a solution of the CMEs.
The Bloch transformation of the formal ansatz is
\begin{equation}\label{ansatz_Bloch_var}
\tu_\text{ans}(x,k)=\frac1\varepsilon\sum_{j=1}^N\sum_{K\in\Lambda^*}\hA_j\left(\frac{k-k^{(j)}+K}\varepsilon\right)p_{n_*}(x,k^{(j)})e^{\ri K\cdot x}
\end{equation}
using \eqref{Bloch_periodic} and the fact that $\big(A(\varepsilon\cdot)e^{\ri k^{(j)}\cdot}\big)^\land(k)=\varepsilon^{-2}\hA(\varepsilon^{-1}(k-k^{(j)}))$. Since $\hA_j(\varepsilon^{-1}(k-k^{(j)}))$ is concentrated near $k=k^{(j)}$, we decompose $U$ on $\B$ into $N+1$ parts with the first $N$ being compactly supported in the vicinity of one of the points $\kj$ and the last one supported away from all $\kj$. This is then extended $\Lambda^*$-periodically onto $\R^2$. We write
\begin{equation*}
U(k)=\sum_{k=1}^N\frac1\varepsilon\hB_j\left(\frac{k-\kj}\varepsilon\right) + \hC(k), \quad k \in \R^2,
\end{equation*}
where
\beq\label{E:BjC-supp}
\supp\left(\hB_j\left(\frac{\cdot-\kj}\varepsilon\right)\right)\cap \B = B_{\varepsilon^{r}}(\kj)\quad \text{ and }\quad \supp (\hC) \cap \B = \B\setminus \bigcup_{j=1}^N B_{\varepsilon^{r}}(\kj)
\eeq
and where $\hC$ is $\Lambda^*$-periodic and $\hB_j$ is $\eps^{-1}\Lambda^*$-periodic on $\R^2$. That means
\begin{equation*}
\hC(k+K)=\hC(k), \quad \hB_j(k+\varepsilon^{-1}K)=\hB_j(k)\qquad\mbox{for all}\; K\in\Lambda^*\;\mbox{and}\; k\in\R^2.
\end{equation*}
Here $r\in(0,1)$ is a parameter to be specified to suit the nonlinear estimates. Moreover we define $\hB_j^*$ and $\hC^*$ as the restrictions of such functions to the periodicity cell, i.e.
\begin{equation*}
\hB_j^*:=\chi_{\varepsilon^{-1}\B}\hB_j\qquad\mbox{and}\qquad\hC^*:=\chi_{\B}\hC.
\end{equation*}

We point out that in $\hB_j$ and $\hC$  the $\,\,\widehat{\cdot}\,\,$ notation does not refer to the Fourier transform of given functions $B_j,C$; it just stresses out the connection between $\hB_j$ and $\hA_j$, the latter of course being the Fourier transform of $A_j$. With this further decomposition the sought solution has the components $\hB_j^*$ close to $\hA_j$ and the component $\hC^*$ small. Note that $\hA_j, j=1,\dots,N$, satisfies equation \eqref{CMEs_phys} transformed in Fourier variables, i.e.
\begin{equation}\label{CMEs_FT}
\Omega \hA_j-\frac12\big(k_1^2\partial_{k_1}^2\omega_{n_*}(\kj)+k_2^2\partial_{k_2}^2\omega_{n_*}(\kj)+2k_1k_2\partial^2_{k_1k_2}\omega_{n_*}(\kj)\big)\hA_j+\hcN_j=0, \quad k \in \R^2.
\end{equation}

The aim now is to apply a fixed point argument to solve system \eqref{v_Eq}, \eqref{w0_eq}, and \eqref{wR_eq}, which is of course coupled in the components $(\hB_j)_{j=1}^N, \hC, \tw_0, \tw_R$. 
The equations for the components $\tw_0$ and $\tw_R$ both involve the linear operator $\eQk L_k\Qke$, see \eqref{w0_eq},\eqref{wR_eq}. This operator is boundedly invertible on its image, by Lemma \ref{invertibility}, and the bound on the inverse is independent of $\varepsilon$. This is thanks to the fact that the projection $\Qke$ projects out the Bloch eigenfunction $p_{n_*}$.

We will also make use of the notation 
\begin{equation}\label{vB}
\tvB(x,k):=\sum_{j=1}^N\tvBj(x,k)\quad\mbox{where}\quad\tvBj(x,k):=\frac1\varepsilon\hB_j\left(\frac{k-\kj}\varepsilon\right)p_{n_*}(x,k)
\end{equation}
and
\begin{equation*}
\tvC(x,k):=\hC(k)p_{n_*}(x,k).
\end{equation*}
We can thus write
\begin{equation*}
\tv(x,k)=\tvB(x,k)+\tvC(x,k).
\end{equation*}

\vskip0.4truecm
Inspired by the strategy of \cite{DP,DU,DS}, our algorithm to construct a solution $\tu\in\cX_2$ of our problem \eqref{eq_Bloch} is the following nested fixed point argument.
\begin{enumerate}[(1)]
	\item For any given $\tv$ bounded, determine the unique small solution $\tw_0$ of the linear problem \eqref{w0_eq} by means of Lemma \ref{invertibility};
	\item For any given $\tv$ bounded and $\tw_0$ from Step 1, apply the Banach fixed point theorem to \eqref{wR_eq} in a neighbourhood of zero to find a small solution $\tw_R$;
	\item For any given $N$-tuple $(\hB_j)_{j=1}^N$ with $(\hBp_j)_{j=1}^N$ decaying sufficiently fast, find a small $\hC$ with support as in \eqref{E:BjC-supp} applying the Banach fixed point theorem to \eqref{v_Eq} on this support;
	\item For $\hC$ given by Step 3 prove the existence of such solutions $(\hBp_j)_{j=1}^N$ to \eqref{v_Eq} which are close to $(\hA_j)_{j=1}^N$ and have the support as in \eqref{E:BjC-supp}. It is here where
	the restriction to a $\PT$-symmetric subspace is used.
\end{enumerate}
The rest of the section carries this algorithm out.

%-----------------------------------------------
\subsection{Preliminary Estimates}

We define for convenience $L^2:=L^2(\R^2)$ and $L^2_{s_B}:=L^2_{s_B}(\R^2)$ (the weighted space $L^2_s(\R^2)$ defined in \eqref{Ls}). 
\begin{lem}\label{Lemma1}
It holds
	\begin{equation*}
	\|\tv_B\|_{\cX_2}\leq c\sum_{j=1}^N\|\hB_j\|_{L^2(\varepsilon^{-1}\B)}, \quad \|\tv_C\|_{\cX_2}\leq c\,\|\hC\|_{L^2(\B)}.
	\end{equation*}
\end{lem}
\begin{proof}
	Using the regularity result in Lemma \ref{Lemma_Reg_supH^2}, we have $\|p_{n_*}(\cdot,k)\|_{H^2(\cQ)}$ uniformly bounded in $k\in\B$. Therefore,
	\begin{equation*}
	\begin{split}
	\|\tv_C\|_{\cX_2} & \leq \esssup_{k\in\B}\|p_{n_*}(\cdot,k)\|_{H^2(\cQ)}\|\hC\|_{L^2(\B)}\leq c\,\|\hC\|_{L^2(\B)},\\
	\|\tv_B\|_{\cX_2}&\leq\esssup_{k\in\B}\|p_{n_*}(\cdot,k)\|_{H^2(\cQ)}\frac 1\varepsilon\sumjN\left\|\hB_j\left(\frac{\cdot-k^{(j)}}{\varepsilon}\right)\right\|_{L^2(\B+\kj)}\\
	&\leq c\sumjN\left(\int_{\varepsilon^{-1}\B}|\hB_j(t)|^2\dd t\right)^{\frac12}=c\sum_{j=1}^N\|\hB_j\|_{L^2(\varepsilon^{-1}\B)}.
\end{split}
	\end{equation*}
Note that in the first inequality for $\|\tv_B\|_{\cX_2}$ the $k^{(j)}$-shift of the $k$-integral is allowed due to the periodicity of $\hB_j$.
	\end{proof}
Clearly, the estimate in Lemma \ref{Lemma1} is $O(1)$ in $\eps$. In estimating the residual below, it will be necessary to show the smallness of the nonlinearity in the $\cX_2$-norm for $\varepsilon$ small. Inspired by (33) in \cite{DU-err}, the next Lemma produces this smallness for the components $\tvB$ if $s_B \geq 2$.
%DOHNAL-UECKER, p.26, TRICK!!
\begin{lem}\label{LemmaDU} Let $s_B\geq 2$. For $a,b\in\{1,2,3\}$ and $\tvB$ defined in \eqref{vB} there holds
	\begin{equation*}
	\|\tv_{B,a}\astB\tv_{B,b}\|_{\cX_2}\leq c\varepsilon\sum_{i,j=1}^N\|\hBp_i\|_{L^2_{s_B}}\|\hBp_j\|_{L^2_{s_B}}.
	\end{equation*}
\end{lem}
\begin{proof} For $a\in\{1,2,3\}$ we define $v_{B,a}:=\cT^{-1}\tv_{B,a}$. First,
	\begin{equation}\label{Lemma_first}
	\|\tv_{B,a}\astB\tv_{B,b}\|^2_{\cX_2}=\|\widetilde{v_{B,a}v_{B,b}}\|^2_{\cX_2}\leq c\|v_{B,a}v_{B,b}\|^2_{H^2(\R^2)},
	\end{equation}
	by the isomorphism \eqref{Bloch_isomorphismum}. We now claim that
	\begin{equation}\label{Lemma_claim}
	\|v_{B,a}v_{B,b}\|^2_{H^2(\R^2)}\leq c\varepsilon^4\sum_{i,j=1}^N\|\Bp_i(\varepsilon\cdot)\Bp_j(\varepsilon\cdot)\|_{H^2(\R^2)}^2,
	\end{equation}
	where $\Bp_i:=(\hBp_i)^\vee$ for $i\in\{1,\dots,N\}$. Indeed, for $i,j\in\{1,\dots,N\}$ one has
	\begin{equation*}
	\begin{split}
	\|v_{B_i,a}v_{B_j,b}\|^2_{L^2(\R^2)}&\leq c\varepsilon^4\esssup_{k\in\B}\|p_{n_*}(\cdot,k)\|^2_\infty\int_{\R^2}\prod\limits_{l=i,j}\bigg|\int_{\varepsilon^{-1}\B}\hB_l(\kappa)e^{\ri\eps \kappa\cdot x}\dd \kappa\bigg|^2\dd x\\
	&\leq c\varepsilon^4\|\Bp_i(\varepsilon\cdot)\Bp_j(\varepsilon\cdot)\|_{L^2(\R^2)}^2,
	\end{split}
	\end{equation*}
	\begin{equation*}	
	\begin{split}
	\|\nabla\left(v_{B_i,a}v_{B_j,b}\right)\|^2_{L^2(\R^2)}&\leq c\varepsilon^4\esssup_{k\in\B}\|p_{n_*}(\cdot,k)\|^2_{W^{1,\infty}}\int_{\R^2}\prod\limits_{l=i,j}\bigg|\int_{\varepsilon^{-1}\B}(1+\ri \varepsilon\kappa)\hB_l(\kappa)e^{\ri\eps \kappa\cdot x}\dd \kappa\bigg|^2\dd x\\
	&\leq\varepsilon^4\|\Bp_i(\varepsilon\cdot)\Bp_j(\varepsilon\cdot)\|_{H^1(\R^2)}^2
	\end{split}
	\end{equation*}
	and analogously one infers 
	\begin{equation*}
	\|\nabla^2\left(v_{B_i,a}v_{B_j,b}\right)\|^2_{L^2(\R^2)}\leq\varepsilon^4\|\Bp_i(\varepsilon\cdot)\Bp_j(\varepsilon\cdot)\|_{H^2(\R^2)}^2.
	\end{equation*}
	Notice that here we used the regularity estimate for the eigenfunctions given by Lemma \ref{Lemma_Reg_supW^2infty}.
	
	Therefore, combining \eqref{Lemma_first} and \eqref{Lemma_claim} one finds
	\begin{equation*}
	\begin{split}
		\|\tv_{B,a}\astB\tv_{B,b}\|^2_{\cX_2}&\leq c\varepsilon^4\sum_{i,j=1}^N\|\Bp_i(\varepsilon\cdot)\Bp_j(\varepsilon\cdot)\|_{H^2(\R^2)}^2\\
		&=c\varepsilon^4\sum_{i,j=1}^N\int_{\R^2}(1+|k|^2)^2\left|\left(\Bp_i(\varepsilon\cdot)\Bp_j(\varepsilon\cdot)\right)^\land(k)\right|^2\dd k\\
		&\leq c\varepsilon^4\sum_{i,j=1}^N\int_{\R^2}\left|\left(\Bp_i(\varepsilon\cdot)\Bp_j(\varepsilon\cdot)\right)^\land(k)\right|^2\dd k+
		c\varepsilon^4\sum_{i,j=1}^N\int_{\R^2}|k|^4\left|\left(\Bp_i(\varepsilon\cdot)\Bp_j(\varepsilon\cdot)\right)^\land(k)\right|^2\dd k\\
		&= c\varepsilon^4\sum_{i,j=1}^N\|\Bp_i(\varepsilon\cdot)\Bp_j(\varepsilon\cdot)\|_{L^2(\R^2)}^2+c\varepsilon^4\sum_{i,j=1}^N\int_{\R^2}\left|\eps^{-1}k\right|^4\left|\left(\Bp_i\Bp_j\right)^\land\left(\eps^{-1}k\right)\right|^2\dd k\\
		&\leq c\varepsilon^2\sum_{i,j=1}^N\|\Bp_i\|_{L^\infty}^2\|\Bp_j\|^2_{L^2(\R^2)}+c\varepsilon^6\sum_{i,j=1}^N\|\left(\Bp_i\Bp_j\right)^\land\|_{L^2_2(\R^2)}^2\\
		&\leq c\varepsilon^2\sum_{i,j=1}^N\|\Bp_i\|^2_{H^2(\R^2)}\|\Bp_j\|^2_{H^2(\R^2)}\\
		&\leq c\varepsilon^2\sum_{i,j=1}^N\|\hBp_i\|^2_{L^2_2(\R^2)}\|\hBp_j\|^2_{L^2_2(\R^2)}.
	\end{split}
	\end{equation*}
	Notice that we have used the embedding $H^2(\R^2)\hookrightarrow L^\infty(\R^2)$ and, to conclude, the isomorphism property of the Fourier transform between $H^2(\R^2)$ and $L^2_2(\R^2)$.
\end{proof}
For the analysis of the nonlinearity we need to calculate the double convolutions $\tv_a*_\B \tv_b *_\B \tbv_c$ (for $a,b,c\in\{1,2,3\}$) appearing in $\tF(\tv)$. We have
\beq\label{E:double-convol}
\begin{aligned}
(\tv_a*_\B \tv_b *_\B \tbv_c)(x,k) &= \int_\B\bigg(\int_\B \tv_a(k-s-t,x)\tv_b(t,x)\dd t\bigg)\tbv_c(s,x)\dd s\\
&=\int_{2\B}\int_\B \tv_a(k-l,x)\tv_b(l-s,x)\tbv_c(s,x)\dd s\dd l
\end{aligned}
\eeq
using the transformation $l=s+t$. Also note that the integration domains $2\B$ and $\B$ can be both shifted by an arbitrary $k_*\in \R^2$ due to the quasi-periodicity of $\tv_a,\tv_b$, and $\tv_c$ with respect to the variable $k$.

\begin{lem}\label{Lemma2} 
Let $\hB_j\in L^2_{s_B}(\eps^{-1}\B)$ with $s_B>1$ and $\hC\in L^2(\B)$ have the supports as in \eqref{E:BjC-supp}. Then 
\beq
\begin{aligned}\label{E:Fest-2}
	\|\tF(\tv)\|_{\cX_2}\leq\,\, &c\varepsilon^2\left(\sumabgN\|\hBp_\alpha\|_{L^2_{s_B}}\|\hBp_\beta\|_{L^2_{s_B}}\|\hBp_\gamma\|_{L^2_{s_B}}+\sum_{\alpha,\beta=1}^N\|\hBp_\alpha\|_{L^2_{s_B}}\|\hBp_\beta\|_{L^2_{s_B}}\|\hCp\|_{L^2}\right) \\
	&+ c\eps \sum_{\alpha=1}^N \|\hBp_\alpha\|_{L^2_{s_B}}\|\hCp\|_{L^2}^2 + c\,\|\hCp\|_{L^2}^3.
\end{aligned}
\eeq
For $\tF(\tv_B)$ we have for each $d=1,2,3$
\beq\label{E:FvB}
\begin{aligned}
	\tF_d(\tv_B)(x,k)=&\,\, \eps\sum_{a,b,c=1}^3\underline{\chi}_{a,b,c,d}^{(3)}(x)\sumabgN \int_{2\eps^{-1}\B}\int_{\eps^{-1}\B} \hB_\alpha\left(\frac{k-(k^{(\alpha)}+k^{(\beta)}-k^{(\gamma)})}{\eps}-l\right)\\
	&\cdot \hB_\beta(l-t)\hbB_\gamma(t)p_{n_*,a}(x,k-k^{(\beta)}+k^{(\gamma)}-\eps l)p_{n_*,b}(x,k^{(\beta)}+\eps (l-t))\\
	&\cdot p_{n_*,c}(x,-k^{(\gamma)}+\eps t)\dd t \dd l
\end{aligned}
\eeq
and the estimate
\beq\label{E:FvB-est}
	\|\tF(\tv_B)\|_{\cX_2}\leq c\varepsilon^2\sumabgN\|\hBp_\alpha\|_{L^2_{s_B}}\|\hBp_\beta\|_{L^2_{s_B}}\|\hBp_\gamma\|_{L^2_{s_B}}.
\eeq
\end{lem}
\begin{proof}
	First, notice that
	\begin{equation*}
	\tbv_B(x,k)=\frac1\varepsilon\sumjN\int_\B\hbB_j\left(\frac{k+\kj}\varepsilon\right)p_{n_*}(x,k)e^{\ri k\cdot x}\dd k.
	\end{equation*}
	Indeed,
	\begin{equation*}
	\begin{split}
	\tbv_B(x,k)&=\frac1\varepsilon\sumjN\int_\B\overline{\hB_j}\left(\frac{k-\kj}\varepsilon\right)\overline{p_{n_*}(x,k)}e^{-\ri k\cdot x}\dd k=\frac1\varepsilon\sumjN\int_\B\hbB_j\left(\frac{-k+\kj}\varepsilon\right)\overline{p_{n_*}(x,k)}e^{-\ri k\cdot x}\dd k\\
	&=\frac1\varepsilon\sumjN\int_\B\hbB_j\left(\frac{k+\kj}\varepsilon\right)\overline{p_{n_*}(x,-k)}e^{\ri k\cdot x}\dd k
	\end{split}
	\end{equation*}
	and finally we use the symmetry $\overline{p_{n_*}(x,-k)}=p_{n_*}(x,k)$, see \eqref{E:sym-minus-k}. 
	
Similarly,
	$$\tbv_C(x,k)=\hbC(k)p_{n_*}(x,k).$$
We start with formula \eqref{E:FvB}. Using \eqref{E:double-convol}, we have
$$
\begin{aligned}
	\tF_d(\tv_B)(x,k)=&\,\eps\sum_{a,b,c=1}^3\underline{\chi}^{(3)}_{a,b,c,d}(x)\sumabgN \int_{2\B+k^{(\beta)}-k^{(\gamma)}}\int_{\B-k^{(\gamma)}} \hB_\alpha\left(\frac{k-l-k^{(\alpha)}}{\eps}\right)\\
	&\cdot\hB_\beta\left(\frac{l-t-k^{(\beta)}}{\eps}\right)\hbB_\gamma\left(\frac{t+k^{(\gamma)}}{\eps}\right)p_{n_*,a}(x,k-l)p_{n_*,b}(x,l-t)p_{n_*,c}(x,t)\dd t \dd l,
\end{aligned}
$$
where we have used \eqref{E:double-convol} and the fact that due to the quasi-periodicity in $k$ the convolution domains can be shifted by arbitrary $k_*\in \R^2$. The transformations $t'=\eps^{-1}(t+k^{(\gamma)}),$ $l'=\eps^{-1}(l-k^{(\beta)}+k^{(\gamma)})$ produce \eqref{E:FvB}.

The estimates \eqref{E:Fest-2} and \eqref{E:FvB-est} are proved next. Let us denote $\|\chi^{(3)}\|_{H^2(\cQ)}:=\!\displaystyle{\max_{a,b,c,d\in\{1,2,3\}}}\|\chi^{(3)}_{a,b,c,d}\|_{H^2(\cQ)}$. Using the assumption on $\chi^{(3)}$ in (A5),
$$
\begin{aligned}
	\|\tF(\tv)\|_{\cX_2}&\leq c\,\|\chi^{(3)}\|_{H^2(\cQ)}\max_{a\in \{1,2,3\}} \|p_{n_*,a}\|_{L^\infty(\B,H^2(\cQ))}^3\\
	&
	\quad\cdot\left[ \eps^{-6} \sumabgN \int_\B\left|\int_{2\B}\int_{\B}\hB_\alpha\left(\frac{k-l-k^{(\alpha)}}{\eps}\right)\hB_\beta\left(\frac{l-t-k^{(\beta)}}{\eps}\right)\hbB_\gamma\left(\frac{t+k^{(\gamma)}}{\eps}\right)  \dd t \dd l\right|^2\dd k\right. \\
	&\quad\quad\left. +\eps^{-4} \sum_{\alpha,\beta=1}^N \int_\B\left|\int_{2\B}\int_{\B}\hB_\alpha\left(\frac{k-l-k^{(\alpha)}}{\eps}\right)\hB_\beta\left(\frac{l-t-k^{(\beta)}}{\eps}\right)\hbC(t)\dd t \dd l\right|^2\dd k \right.\\
	& \quad\quad\left. + \eps^{-2}\sum_{\alpha=1}^N \int_\B\left|\int_{2\B}\int_{\B}\hB_\alpha\left(\frac{k-l-k^{(\alpha)}}{\eps}\right)\hC(l-t)\hbC(t)\dd t \dd l\right|^2\dd k \right.\\
	& \quad\quad\left. + \int_\B\left|\int_{2\B}\int_{\B}\hC(k-l)\hC(l-t)\hbC(t)\dd t \dd l\right|^2\dd k \right]^\frac12\\
	& \leq  c\left[\eps^4\sumabgN \|\hB_\alpha *_{\eps^{-1}\B}\hB_\beta *_{\eps^{-1}\B} \hbB_\gamma\|_{L^2(\eps^{-1}\B)}^2 \right.\\
	& \quad\quad\left. + \eps^4 \sum_{\alpha,\beta=1}^N \int_\B\left|\int_{2\eps^{-1}\B}\int_{\eps^{-1}\B}\hB_\alpha\left(\frac{k-k^{(\alpha)}-k^{(\beta)}}{\eps}-l'\right)\hB_\beta\left(l'-t'\right)\hbC(\eps t')\dd t' \dd l'\right|^2\dd k \right. \\
	& \quad\quad\left. + \eps^{6}\sum_{\alpha=1}^N \int_\B\left|\int_{2\eps^{-1}\B}\int_{\eps^{-1}\B}\hB_\alpha\left(\frac{k-k^{(\alpha)}}{\eps}-l'\right)\hC(\eps(l'-t'))\hbC(\eps t')\dd t' \dd l'\right|^2\dd k \right.\\
	& \quad\quad\left. + \int_\B\left|\int_{2\B}\int_{\B}\hC(k-l)\hC(l-t)\hbC(t)\dd t \dd l\right|^2\dd k \right]^\frac12,
\end{aligned}
$$
	where  we have used again transformations of the type $t'=\eps^{-1}(t+k^{(\gamma)})$ and the fact that the convolution domains can be shifted by arbitrary $k_*\in \R^2$.
	
Next, for $j\in\{1,\dots,N\}$ and $m\in\N$ we introduce notation for $\hB$ and $\hC$ restricted to $m^2$ periodicity cells. In detail, let
	$$\hB^{(m)}_j:=\chi_{m\varepsilon^{-1}\B}\hB_j, \quad \hC^{(m)}:=\chi_{m\B}\hC.$$
	Note that $\hB^{(1)}_j=\hBp_j$ and $\hC^{(1)}=\hCp$. With this notation we have
	\begin{equation*}
	\begin{aligned}
	&\|\tF(\tv)\|_{\cX_2}\leq c \left[\eps^4\sumabgN \|\hB_\alpha^{(3)} *\hB_\beta^{(2)} *\hbB_\gamma^*\|_{L^2(\R^2)}^2  + \eps^6 \sum_{\alpha,\beta=1}^N \|\hB_\alpha^{(3)} *\hB_\beta^{(2)} *\hbC^*(\eps \cdot) \|_{L^2(\R^2)}^2\right. \\
	& \ \left. + \eps^{8}\sum_{\alpha=1}^N \|\hB_\alpha^{(3)} *\hC^{(2)}(\eps \cdot) * \hbC^*(\eps \cdot)\|_{L^2(\R^2)}^2 +  \|\hC^{(3)} *\hC^{(2)}* \hbC^* \|_{L^2(\R^2)}^2 \right]^\frac12,
	\end{aligned}
	\end{equation*}
	where $*$ denotes the convolution over the full $\R^2$.
	Using Young's inequality for convolutions and the fact that (due to the periodicity) $\|\hB_j^{(m)}\|_{L^p(\R^2)}\leq c\,\|\hB_j^*\|_{L^p(\R^2)}, \|\hC^{(m)}\|_{L^p(\R^2)}\leq c\,\|\hC^*\|_{L^p(\B)}$ for all $m\in \N$, $p\geq 1$, and $j=1,\dots,N$, we estimate
	\begin{equation*}
	\begin{aligned}
	&\|\hB_\alpha^{(3)} *\hB_\beta^{(2)} *\hbB_\gamma^*\|_{L^2(\R^2)} \leq c\,\|\hB_\alpha^*\|_{L^2(\R^2)}\|\hB_\beta^*\|_{L^1(\R^2)}\|\hB_\gamma^*\|_{L^1(\R^2)},\\
	& \|\hB_\alpha^{(3)} *\hB_\beta^{(2)} *\hbC^*(\eps \cdot)\|_{L^2} \leq c\,\|\hB_\alpha^*\|_{L^1(\R^2)}\|\hB_\beta^*\|_{L^1(\R^2)}\|\hC^*(\eps \cdot)\|_{L^2(\R^2)},\\
	&\|\hB_\alpha^{(3)} *\hC^{(2)}(\eps \cdot) * \hbC^*(\eps \cdot)\|_{L^2(\R^2)} \leq c\,\|\hB_\alpha^*\|_{L^1(\R^2)}\|\hC^*(\eps \cdot)\|_{L^1(\R^2)}\|\hC^*(\eps \cdot)\|_{L^2(\R^2)},\\
	&\|\hC^{(3)} *\hC^{(2)} * \hbC^*\|_{L^2(\R^2)} \leq c\,\|\hC^*\|_{L^2(\B)}\|\hC^*\|_{L^1(\B)}^2.
	\end{aligned}
	\end{equation*}
	Finally, we arrive at \eqref{E:Fest-2} by using $\|\hbC^*\|_{L^1(\B)}\leq c\,\|\hbC^*\|_{L^2(\B)},$
	$$\|\hC^*(\eps \cdot)\|_{L^2(\R^2)}=\eps^{-1}\|\hC\|_{L^2(\B)}, \quad \|\hC^*(\eps \cdot)\|_{L^1(\R^2)}=\eps^{-2}\|\hC\|_{L^1(\B)},
	$$
and
	\beq\label{E:L1-L2s}
	\|\hB_j^*\|_{L^1(\R^2)} \leq c\,\|\hB_j^*\|_{L^2_{s_B}(\R^2)} \quad\text{if} \; s_B>1.
	\eeq
	The last inequality follows from 
	\begin{equation*}
	\begin{split}
	\|\hBp_j\|_{L^1(\R^2)}&=\int_{\R^2}(1+|t|)^{-s_B}(1+|t|)^{s_B}|\hBp_j(t)|\dd t\\
	&\leq\bigg(\int_{\R^2}(1+|t|)^{-2s_B}\dd t\bigg)^{\frac12}\bigg(\int_{\R^2}(1+|t|)^{2s_B}|\hBp_j(t)|^2\dd t\bigg)^{\frac12},
	\end{split}
	\end{equation*}
	where the first factor is bounded provided $s_B>1$.
	
Estimate \eqref{E:FvB-est} is clearly the first part of \eqref{E:Fest-2}.
\end{proof}

For the whole analysis of the components $\tw_0$ and $\tw_R$ we assume that for all sufficiently small $\varepsilon>0$ and some $s_B> 1$ we have
\begin{equation}\label{BjC_smallness}
\|\hCp\|_{L^2}=\|\hC\|_{L^2(\B)}\leq c_0\eps^{2-2r}\quad\!\mbox{and}\quad \|\hBp_j\|_{L^2_{s_B}}=\|\hB_j\|_{L^2_{s_B}(\varepsilon^{-1}\B)}\leq c\quad \mbox{for all}\,\,j\in\{1,\dots,N\}
\end{equation}
for some constants $c_0,c>0$. The regularity parameter $s_B$ is chosen below in order for the required estimates to work. Notice that under such an assumption the terms involving $\hC$ in the estimate \eqref{E:Fest-2} are $o(\varepsilon^2)$, provided $r\in\left(0,\tfrac12\right]$.
%-----------------------------------------------
\subsection{Component \texorpdfstring{$\boldsymbol{\tw_0}$}{w0tilde}}\label{Section_w0}

Under assumption \eqref{BjC_smallness} with $s_B\geq 2$ we solve the (linear) equation \eqref{w0_eq} for $\tw_0$ and derive an $\cX_2$-estimate on $\tw_0$.

\begin{lem}
	Let $s_B\geq 2$ and assume \eqref{BjC_smallness}, where $\hB_j$ and $\hC$ have the supports as in \eqref{E:BjC-supp} with $r\in (0,\tfrac{2}{3}]$. Then \eqref{w0_eq} has a unique solution $\tw_0\in\cX_2$ with
	\begin{equation}\label{w0_estimate}
	\|\tw_0\|_{\cX_2}\leq c_1\|\tF(\tv)\|_{\cX_2}\leq c_2\varepsilon^2,
	\end{equation}
	where $c_1,c_2>0$ and $c_2$ depends polynomially on $\|\hBp_j\|_{L^2_{s_B}}$, $\|\hCp\|_{L^2}$ for all $j\in\{1,\dots,N\}$.
\end{lem}

\begin{proof}
	By Lemma \ref{Lemma2} and assumption \eqref{BjC_smallness}, we get for the right-hand side of \eqref{w0_eq}
	\begin{equation*}
	\|\!\eQk\tF(\tv)\|_{\cX_2}\leq c\,\|\tF(\tv)\|_{\cX_2}\leq c'\varepsilon^2
	\end{equation*}
	if $5-4r\geq 2$ and $6-6r\geq 2$, which hold for $r\in (0,\tfrac{2}{3}]$. 
	Here $c'$ depends polynomially on $\|\hBp_j\|_{L^2_{s_B}}$, $j\in\{1,\dots,N\}$, and $\|\hCp\|_{L^2}$. Lemma \ref{invertibility} produces a solution $\tw_0(\cdot,k)\in\Qke H_\#(\text{curl}^2,\cQ)$ with
	\begin{equation}\label{consequence_Lemma_invertibility}
	\|\tw_0\|_{L^2(\B,H_\#(\text{curl}^2,\cQ))}\leq c\,\|\tF(\tv)\|_{L^2(\B,L^2(\cQ))}\leq c\,\|\tF(\tv)\|_{\cX_2}.
	\end{equation}
	For an $\cX_1$-estimate we need to control also the divergence $\nabla_k'\cdot\tw_0$. First note that
	$$\eQk\left(\epsilon\tw_0\right)(\cdot,k)=\epsilon\tw_0(\cdot,k)-\langle\epsilon\tw_0(\cdot,k),p_{n_*}(\cdot,k)\rangle\epsilon p_{n_*}(\cdot,k).$$
	Taking the divergence $\nabla_k'\cdot$ of equation \eqref{w0_eq} produces
	\begin{equation*}
	\nabla'_k\cdot\tw_0=\frac1\epsilon\bigg(-\nabla'_k\cdot\eQk\tF(\tv)-\left\langle\frac1{\omega^2}\nabla'_k\times\nabla'_k\times\tw_0(\cdot,k)-\epsilon\tw_0(\cdot,k),p_{n_*}(\cdot,k)\right\rangle\nabla'_k\cdot\left(\epsilon p_{n_*}(\cdot,k)\right)-\nabla'_k\epsilon\cdot\tw_0\bigg).
	\end{equation*}
	Due to (A5) and the regularity $p_{n_*}(\cdot,k)\in H^2(\cQ)$ for all $k\in\B$ given by Lemma \ref{Lemma_Reg_supH^2}, we have
	\begin{equation*}
	\nabla'_k\cdot\tw_0(\cdot,k)\in L^2(\cQ)\quad\mbox{such that}\quad\tw_0(\cdot,k)\in H^1(\cQ).
	\end{equation*}
This allows us to estimate the $\cX_1$-norm
\begin{equation*}
\begin{split}
\|\tw_0\|_{\cX_1}&\leq c\left(\|\tw_0\|_{L^2(\B,H_\#(\text{curl},\cQ))}+\|\nabla'_k\cdot\tw_0\|_{\cX_0}\right)\\
&\leq c\,\bigg(\|\tF(\tv)\|_{\cX_2}+\|\nabla'_k\cdot\eQk\tF(\tv)\|_{\cX_0}\\
& \qquad + \bigg\|\left\langle\frac1{\omega^2}\nabla'_k\times\nabla'_k\times\tw_0(\cdot,k)-\epsilon\tw_0(\cdot,k),p_{n_*}(\cdot,k)\right\rangle\bigg\|_{L^2(\B)}\!+ \|\tw_0\|_{\cX_0}\bigg)\\
&\leq c\,\big(\|\tF(\tv)\|_{\cX_2}+\|\nabla_k'\cdot\tF(\tv)\|_{\cX_0}+\|\tF(\tv)\|_{\cX_0}+\|\tw_0\|_{L^2(\B,H_\#(\text{curl}^2,\cQ))}+\|\tw_0\|_{\cX_0}\big)\\
&\leq c\,\|\tF(\tv)\|_{\cX_2},
 \end{split}
\end{equation*}
where in the second and the last step we used \eqref{consequence_Lemma_invertibility}. The $\cX_2$-estimate is analogous:
\begin{equation*}
\begin{split}
\|\tw_0\|_{\cX_2}&\leq c\left(\|\tw_0\|_{\cX_1}+\|\nabla'_k\times\tw_0\|_{\cX_1}+\|\nabla'_k\cdot\tw_0\|_{\cX_1}\right)\\
&\leq c\,\big(\|\tF(\tv)\|_{\cX_2}+\|\nabla_k'\times\tw_0\|_{L^2(\B,L^2(\cQ))}+\|\nabla_k'\times\nabla_k'\times\tw_0\|_{L^2(\B,L^2(\cQ))}+\|\nabla_k'\cdot\tw_0\|_{\cX_1}\big)\\
&\leq c\,\|\tF(\tv)\|_{\cX_2}.
\end{split}
\end{equation*}
Note that in the second inequality the divergence $\nabla_k'\cdot$ of $\nabla_k'\times\tw_0$ vanishes and in the last inequality the estimate of $\|\nabla_k'\cdot\tw_0\|_{\cX_1}$ is analogous to the $\cX_1$ estimate above. 
\end{proof}

%-----------------------------------------------
\subsection{Component \texorpdfstring{$\boldsymbol{\tw_R}$}{wR~}}\label{Section_wR}
Next, we keep assumption \eqref{BjC_smallness} and solve equation \eqref{wR_eq} for $\tw_R$ via a Banach fixed point argument with $\tv$ satisfying \eqref{v_Eq} and $\tw_0$ as just obtained in Sec. \ref{Section_w0}. We show that for $r\in\left(0,\tfrac12\right]$ a solution of $O(\eps^3)$ (in the $\cX_2$-norm) exists.
We write
\begin{equation}\label{wR_eq_fixpt}
\tw_R=\left(\eQk L_k\Qke\right)^{-1}\omega^2\eQk\left(\tF(\tv+\tw_0+\tw_R)-\tF(\tv)\right)=:G(\tw_R)
\end{equation}
and we aim to show the contraction property of the map $G$ in the ball
\begin{equation*}
\cB^{\cX_2}_{K\varepsilon^\eta}:=\{f\in\cX_2\,|\,\|f\|_{\cX_2}<K\varepsilon^\eta\}
\end{equation*}
for suitable values of $K,\eta>0$. Applying the algebra property \eqref{algebra:X_2} of $\cX_2$, we get
\begin{equation*}
\begin{split}
\|G(\tw_R)\|_{\cX_2}&\leq c\,\big[\!\!\sum_{a,b\in\{1,2,3\}}\big(\|\tv_a\astB\tv_b\|_{\cX_2}+\|\tv_a\astB\tbv_b\|_{\cX_2}\big)(\|\tw_0\|_{\cX_2}+\|\tw_R\|_{\cX_2})+\|\tw_0\|^2_{\cX_2}(\|\tv\|_{\cX_2}+\|\tw_R\|_{\cX_2})\\
&\qquad+\|\tw_R\|^2_{\cX_2}(\|\tv\|_{\cX_2}+\|\tw_0\|_{\cX_2})+\|\tv\|_{\cX_2}\|\tw_0\|_{\cX_2}\|\tw_R\|_{\cX_2}+\|\tw_0\|^3_{\cX_2}+\|\tw_R\|^3_{\cX_2}\big].
\end{split}
\end{equation*}
Recalling Lemmas \ref{Lemma1}-\ref{LemmaDU} we first have for any $a,b=1,2,3$
\begin{equation*}
\begin{split}
\|\tv_a\astB\tv_b\|_{\cX_2}&\leq\|\tv_{B,a}\astB\tv_{B,b}\|_{\cX_2}+\|\tv_{B,a}\|_{\cX_2}\|\tv_{C,b}\|_{\cX_2}+\|\tv_{C,a}\|_{\cX_2}\|\tv_{B,b}\|_{\cX_2}+\|\tv_{C,a}\|_{\cX_2}\|\tv_{C,b}\|_{\cX_2}\\
&\leq c\varepsilon\sum_{i,j=1}^N\|\hBp_i\|_{L^2_{s_B}}\|\hBp_j\|_{L^2_{s_B}}+c\sum_{i=1}^N\|\hBp_i\|_{L^2_{s_B}}\|\hC\|_{L^2(\B)}+c\|\hC\|_{L^2(\B)}^2\\
&\leq c\,\left(\varepsilon+\varepsilon^{2-2r}+\varepsilon^{4-4r}\right).
\end{split}
\end{equation*}
Therefore, recalling \eqref{w0_estimate}, we obtain for $\tw_R\in\cB^{\cX_2}_{K\varepsilon^\eta}$
\begin{equation*}
\begin{split}
\|G(\tw_R)\|_{\cX_2}&\leq P\big[(\varepsilon+\varepsilon^{2-2r}+\varepsilon^{4-4r})(\varepsilon^2+K\varepsilon^\eta)+\varepsilon^4(1+K\varepsilon^\eta)+K^2\eps^{2\eta}(1+\varepsilon^2)+K\varepsilon^{2+\eta}+\varepsilon^6+K^3\varepsilon^{3\eta}\big]\\
&\leq (P+1)\big(\varepsilon^3+K\varepsilon^{1+\eta}+K^2\eps^{2\eta}\big),
\end{split}
\end{equation*}
where the second inequality holds for all $r\in\left(0,\tfrac12\right]$ and $P$ is a constant depending just on the norms $\|\hBp_j\|_{L^2_{s_B}}$, $j\in\{1,\dots,N\}$, and $c_0$. Choosing $\eta=3$ and $K=P+2$, then $G:\cB^{\cX_2}_{K\varepsilon^\eta}\to\cB^{\cX_2}_{K\varepsilon^\eta}$, i.e.
\begin{equation}\label{G_estimate}
\|G(\tw_R)\|_{\cX_2}\leq K\varepsilon^3.
\end{equation}
Next, we address the contraction property. For  $\tu^{(j)}:=\tv+\tw_0+\tw_R^{(j)}, j\in\{1,2\}$ we have
\begin{equation}\label{G_Lip}
\begin{split}
\|G(\tw_R^{(1)})-G(\tw_R^{(2)})\|_{\cX_2}&\leq c\big[\max_{a,b\in\{1,2,3\}}\big(\|\tv_a\astB\tv_b\|_{\cX_2}+\|\tv_a\astB\tbv_b\|_{\cX_2}\big)+\|w_0\|^2_{\cX_2}+\|\tw_R^{(1)}\|^2_{\cX_2}\\
&\qquad+\|\tw_R^{(2)}\|^2_{\cX_2}+\|\tv\|_{\cX_2}\|\tw_R^{(1)}\|_{\cX_2}+\|\tv\|_{\cX_2}\|\tw_R^{(2)}\|_{\cX_2}\big]\|\tw_R^{(1)}-\tw_R^{(2)}\|_{\cX_2}\\
&\leq c\varepsilon\|\tw_R^{(1)}-\tw_R^{(1)}\|_{\cX_2}.
\end{split}
\end{equation}
The contraction thus follows provided $\varepsilon$ is small enough. By the Banach fixed point theorem there exists a unique solution to equation \eqref{wR_eq_fixpt} for $\tw_R$ which satisfies the estimate
\begin{equation}\label{wR_estimate}
\|\tw_R\|_{\cX_2}\leq K\varepsilon^3.
\end{equation}

For later use, we need also to show the Lipschitz dependence of $\tw$ on $\tv_B$ and $\tv_C$.
\begin{lem}\label{Lemma_VR_Lip}
	The map $\cB^{\cX_2}_\rho(0)\times\cB^{\cX_2}_{\rho\varepsilon^{2-2r}}(0)\ni(\tvB,\tvC)\mapsto\tw(\tvB,\tvC)\in\cB^{\cX_2}_{\rho\varepsilon^2}(0)$ is Lipschitz-continuous for any $\rho>0$ and $\eps>0$ small enough. The Lipschitz constant $C_L$ satisfies $C_L=\cO(\varepsilon)$ as $\varepsilon\to0$.
\end{lem}
\begin{proof}
	Let $\tv^{(1)}=(\tvB^{(1)},\tvC^{(1)})$ and $\tv^{(2)}=(\tvB^{(2)},\tvC^{(2)})\in\cB^{\cX_2}_\rho(0)\times\cB^{\cX_2}_{\rho\varepsilon^{2-2r}}(0)$ and for $i\in\{1,2\}$ define $\tw_0^{(i)}=\tw_0(\tv^{(i)})$ as solutions of \eqref{w0_eq} and $\tw_R^{(i)}=\tw_R(\tv^{(i)})$ as solutions of \eqref{wR_eq} with $\tw_0$ replaced by $\tw_0^{(i)}$. Such functions are well-defined since $\tv^{(i)}\in\cB^{\cX_2}_\rho(0)\times\cB^{\cX_2}_{\rho\varepsilon^{2-2r}}(0)$ implies that the respective coefficients $\hB_j^{(i)}$, $\hC^{(i)}$ fulfil assumption \eqref{BjC_smallness}. Since $\|\tw_0^{(1)}-\tw_0^{(2)}\|_{\cX_2}\leq c_1\|\tF(\tv^{(1)})-\tF(\tv^{(2)})\|_{\cX_2}$, we obtain similarly to \eqref{G_Lip}
	$$\|\tw_0^{(1)}-\tw_0^{(2)}\|_{\cX_2}\leq c\varepsilon\|\tv^{(1)}-\tv^{(2)}\|_{\cX_2}.$$
	For $\tw_R^{(i)}$ we have
	$$\eQk L_k\Qke\tw_R^{(i)}=\omega^2\eQk\big(\tF(\tv^{(i)}+\tw_0^{(i)}+\tw_R^{(i)})-\tF(\tv^{(i)})\big)$$
	and again with analogous computations as in \eqref{G_Lip} we then get
	$$\|\tw_R^{(1)}-\tw_R^{(2)}\|_{\cX_2}\leq c\varepsilon\big(\|\tv^{(1)}-\tv^{(2)}\|_{\cX_2}+\|\tw_R^{(1)}-\tw_R^{(2)}\|_{\cX_2}\big),$$
	which leads to the desired estimate.
\end{proof}

%-----------------------------------------------
\subsection{Component \texorpdfstring{$\boldsymbol{\hC}$}{C-hat}}\label{Section_C}
So far, we have completed the first two steps of the initial program: under assumption \eqref{BjC_smallness} we inferred the existence of a small solution $\tw$ of \eqref{w_Eq}. Now we have to deal with the component $\tv$, i.e. the projection of $\tu$ onto the mode $p_{n_*}$. Recall that our aim is to find solutions $\tv$ with $\hBp_j$ close to the coefficients $\hA_j$ from the ansatz \eqref{ansatz_Bloch_var} and with $\hC$ small.
\vskip0.2truecm
In this section we assume $r\in\left(0,\tfrac12\right]$, choose an arbitrary $(\hB_j)_{j=1}^N$ with $\|\hBp_j\|_{L^2_{s_B}(\R^2)}\leq c$ for all $j\in\{1,\dots,N\}$, and seek a small $\hC$. Recall that $\tu=\tv+\tw=\tvB+\tvC+\tw$, where $\tvB$ is now fixed and $\tw$ is determined by Sec. \ref{Section_w0}-\ref{Section_wR}. Hence, we write $\tw=\tw(\tvC)$ suppressing the dependence on $\tvB$. 
\vskip0.2truecm
Since the support of $\hCp$ within the Brillouin zone $\B$ is in $\B\setminus \cup_{j=1}^NB_{\eps^r}(k^{(j)})$, we  introduce the characteristic function
$$\chi^C(k):=1-\sum_{j=1}^N\chi_{B_{\eps^r}(k^{(j)})}(k).$$
The equation for $\hC$ then reads
\begin{equation}\label{C_eq}
\begin{split}
\hC(k)&=(\omega_{n_*}(k)^2-\omega^2)^{-1}\omega^2 \chi^C(k)\langle \tF(\tvB)(\cdot,k),p_{n_*}(\cdot,k)\rangle\\
&\quad+(\omega_{n_*}(k)^2-\omega^2)^{-1}\omega^2\chi^C(k)\langle \big(\tF(\tvB+\tvC+\tw(\tvC))-\tF(\tvB)\big)(\cdot,k),p_{n_*}(\cdot,k)\rangle\\
&=:T_{1}(k)+T_{2}(k).
\end{split}
\end{equation}
In order to enjoy the algebra property of $\cX_s, s >1$, we multiply both sides of \eqref{C_eq} by $p_{n_*}(\cdot,k)$ and apply a fixed point approach to the resulting equation for $\tv_C$ in a small ball in $\cX_2$.
\begin{equation}\label{vC_eq}
\begin{split}
\tv_C(x,k)=T_{1}(k)p_{n_*}(x,k)+T_{2}(k)p_{n_*}(x,k)=:\cH(\tv_C)(x,k),
\end{split}
\end{equation}
which we aim to solve in the ball 
$$\cB_{K\varepsilon^{2-2r}}^{\cX_2}:=\{f\in\cX_2\,|\,\|f\|_{\cX_2}\leq K\varepsilon^{2-2r}\}$$
for some $K>0$ and all $\eps>0$ small enough.

We have isolated the leading order part $\tF(\tvB)$ of the nonlinearity in the term $T_1$. Note that $\tF(\tvB)$ is concentrated only near a finite number of $k-$points, namely
$$\supp (\tF(\tvB)(x,\cdot))\cap \B \subset B_{3\eps^r}(S),$$
where
$$S:=\{k\in \B: k = k^{(\alpha)}+k^{(\beta)}-k^{(\gamma)}+K \text{ for some } \alpha,\beta,\gamma\in \{1,\dots,N\}, K\in\Lambda^*\}.$$
We write 
$$ S=S_\text{c}\cup S_{\text{nc}},\; \text{ where }\; S_\text{c}:=\{k^{(1)},\dots, k^{(N)}\}.$$
$S_c$ is the critical set as it lies in the level set $W_{\omega_*}$.
Note that $\chi^C(k) =0$ for each $k\in B_{\eps^r}(S_\text{c})$. We split $T_1$ accordingly
$$T_1= T_{1,\text{c}}+T_{1,\text{nc}},\; \text{ where }\;\; \supp (T_{1,\text{c}})\subset B_{3\eps^r}(S_\text{c})\setminus B_{\eps^r}(S_\text{c})\;\;\ \text{and}\;\; \ \supp (T_{1,\text{nc}}) \subset B_{3\eps^r}(S_\text{nc}).$$
We estimate these components separately.

First, we estimate the factor $(\omega^2-\omega_{n_*}(k)^2)^{-1}$. Due to assumption (A3) we have
\beq\label{E:est-om-nc}
|(\omega^2-\omega_{n_*}(k)^2)^{-1}|\leq c \qquad \forall k\in \supp (T_{1,\text{nc}}).
\eeq
On $\supp (T_{1,\text{c}})$ we use the locally quadratic nature of $\omega_{n_*}(k)$ near $k=k^{(j)}, j=1,\dots,N$. Indeed, as $\omega_{n_*}$ has an extremum at each $k^{(j)}$, we have
\begin{equation*}
\nabla\big(\omega_{n_*}(k)^2\big)|_{k=\kj}=2\omega_*\nabla\omega_{n_*}(\kj)=0.
\end{equation*}
Moreover,
\begin{equation*}
\nabla^2\big(\omega_{n_*}(k)^2\big)|_{k=\kj}=2\omega_*\nabla^2\omega_{n_*}(\kj)+\cK(\omega_{n_*})(\kj),
\end{equation*}
where
\begin{equation*}
\cK(\omega_{n_*}):=\begin{pmatrix}
\big(\partial_1\omega_{n_*}\big)^2 & \big(\partial_1\omega_{n_*}\big)\big(\partial_2\omega_{n_*}\big) \\ 
\big(\partial_1\omega_{n_*}\big)\big(\partial_2\omega_{n_*}\big) & \big(\partial_2\omega_{n_*}\big)^2
\end{pmatrix},
\end{equation*}
the determinant of which evidently vanishes. Using (A4), we deduce then that the Hessian of 
$$k\mapsto\big(\omega_{n_*}(k)\big)^2$$ 
is definite at $\kj.$ This in turns implies that
\begin{equation}\label{sup_inverse_Hessian}
|\omega_{n_*}(k)^2-\omega^2|^{-1}\leq c\varepsilon^{-2r} \qquad \forall k\in \supp (T_{1,\text{c}})
\end{equation}
for $\varepsilon$ small enough. It is mainly here where assumption (A3) is used. If $\{k^{(1)},\dots,k^{(N)}\}$ was a proper subset of the level set $W_{\omega_*}$, then $\supp (T_{1,\text{c}})$ would intersect $W_{\omega_*}$ and $|\omega_{n_*}(k)^2-\omega^2|^{-1}$ would blow up on $\supp (T_{1,\text{c}})$.

The estimate of $T_{1,\text{nc}}$ follows directly from \eqref{E:est-om-nc} and \eqref{E:FvB-est}. We get
\beq\label{E:T1nc-est}
\|T_{1,\text{nc}}p_{n_*}\|_{\cX_2}\leq c \eps^2 \sumabgN\|\hBp_\alpha\|_{L^2_{s_B}}\|\hBp_\beta\|_{L^2_{s_B}}\|\hBp_\gamma\|_{L^2_{s_B}}\leq c\eps^2
\eeq
using \eqref{BjC_smallness}.

For $T_{1,\text{c}}$ the estimate \eqref{sup_inverse_Hessian} causes loss of powers of $\eps$ but we gain some powers by assuming a fast decay of  $\hB_j$. In detail, define the weight $h_j(k):=\big(1+\tfrac{|k-\kj|}\varepsilon\big)^{s_B}$. Then analogously to the proof of \eqref{E:FvB-est} in Lemma \ref{Lemma2} we have, due to \eqref{sup_inverse_Hessian},
\beq\label{E:T1c-est}
\begin{aligned}
\|T_{1,\text{c}}p_{n_*}\|_{\cX_2} & \leq c \eps^{2-2r}\sum_{j=1}^N\sumabgN \sup_{k\in B_{\eps^r}(k^{(j)})}h_j(k)^{-1}\|h_j(k^{(j)}+\eps\cdot)\hBp_\alpha * \hBp_\beta * \hbBp_\gamma\|_{L^2}\\
& \leq c\eps^{2-2r+(1-r)s_B}\sumabgN\|\hBp_\alpha\|_{L^2_{s_B}}\|\hBp_\beta\|_{L^2_{s_B}}\|\hBp_\gamma\|_{L^2_{s_B}}\leq c\eps^{2-2r+(1-r)s_B}.
\end{aligned}
\eeq

The estimate of the term $T_2$ is more delicate as $T_2$ is nonlinear in $\tv_C$. Since $\|\tvB\|_{\cX_2}\leq c$ by assumption, we get $\tw\in\cB_{K\varepsilon^2}^{\cX_2}$ as a solution of \eqref{w_Eq} with $K$ dependent on $\big(\|\hBp_j\|_{L^2_{s_B}}\big)_{j=1}^N$, provided $\|\tilde{v}_C\|_{\cX_2}\leq K\eps^{2-2r}$. We show now that such $\tilde{v}_C$ exists. First, similarly to the map $G$ in Sec. \ref{Section_wR} and taking into account \eqref{sup_inverse_Hessian},
\begin{equation}\label{vCj_NLest}
\begin{split}
\|T_2p_{n_*}\|_{\cX_2}&\leq c\varepsilon^{-2r}\|\tF(\tu)-\tF(\tvB)\|_{\cX_2}\\
&\leq c\varepsilon^{-2r}\bigg[\|\tvB\|_{\cX_2}\|\tvC\|_{\cX_2}\|\tw\|_{\cX_2}+\|\tw\|^3_{\cX_2}+\|\tw\|^2_{\cX_2}(\|\tvB\|_{\cX_2}+\|\tvC\|_{\cX_2})\\
&\quad+\|\tvC\|^2_{\cX_2}\|\tw\|_{\cX_2}+\!\!\!\sum_{a,b\in\{1,2,3\}}\!\!\big(\|\tv_{B,a}\astB\tv_{B,b}\|_{\cX_2}+\|\tv_{B,a}\astB\tbv_{B,b}\|_{\cX_2}\big)\|\tw\|_{\cX_2}\\
&\quad+\sum_{a,b,c\,\in\{1,2,3\}}\sum_{\substack{\mu_1,\mu_2,\mu_3\in\{B,C\}\\(\mu_1,\mu_2,\mu_3)\not=(B,B,B)}}\|\tv_{\mu_1,a}\astB\tv_{\mu_2,b}\astB\tbv_{\mu_3,c}\|_{\cX_2}\bigg].
\end{split}
\end{equation}
We may apply the algebra property of $\cX_2$ and Lemma \ref{LemmaDU} to treat the convolution terms. For $\tv_C\in \cB_{K\varepsilon^{2-2r}}^{\cX_2}$ we obtain
\begin{equation*}\label{convolutions_BBC_CCB_CCC}
\begin{split}
&\|\tv_{B_i,a}\astB\tv_{B_j,b}\astB\tbv_{C_m,c}\|_{\cX_2}\leq\|\tv_{B_i,a}\astB\tv_{B_j,b}\|_{\cX_2}\|\tv_{C_m,c}\|_{\cX_2}\leq c\varepsilon^{3-2r},\\
&\|\tv_{B_i,a}\astB\tv_{C_j,b}\astB\tbv_{C_m,c}\|_{\cX_2}\leq\|\tv_{B_i,a}\|_{\cX_2}\|\tv_{C_j,b}\|_{\cX_2}\|\tv_{C_m,c}\|_{\cX_2}\leq c\varepsilon^{4-4r},\\
&\|\tv_{C_i,a}\astB\tv_{C_j,b}\astB\tbv_{C_m,c}\|_{\cX_2}\leq\|\tv_{C_i,a}\|_{\cX_2}\|\tv_{C_j,b}\|_{\cX_2}\|\tv_{C_m,c}\|_{\cX_2}\leq c\varepsilon^{6-6r}.
\end{split}
\end{equation*}
Analogous estimates hold when the complex conjugation is moved onto another term. From \eqref{vCj_NLest} we then get
\begin{equation}\label{vCj_NLest_pt2}
\begin{split}
\|T_2p_{n_*}\|_{\cX_2}&\leq c\eps^{-2r}\big(\varepsilon^{4-2r}+\varepsilon^6+\varepsilon^4(1+\varepsilon^{2-2r})+\varepsilon^{6-4r}+\varepsilon^3+\varepsilon^{3-2r}+\varepsilon^{4-4r}+\varepsilon^{6-6r}\big)\\
&\leq c\varepsilon^{3-4r},
\end{split}
\end{equation}
since $r\in\left(0,\tfrac12\right]$.

Combining then \eqref{E:T1nc-est}, \eqref{E:T1c-est} and \eqref{vCj_NLest_pt2}, we obtain
\begin{equation*}
\|\cH(\tv_{C})\|_{\cX_2}\leq c\big(\eps^2+\varepsilon^{(1-r)s_B+2-2r}+\varepsilon^{3-4r}\big).
\end{equation*}
Since for $r\in\left(0,\tfrac12\right]$ all exponents are greater than or equal to $2-2r$, we get 
$$\cH:\cB^{\cX_2}_{K\varepsilon^{2-2r}}\to\cB^{\cX_2}_{K\varepsilon^{2-2r}}$$ 
for $K=K\big(\|\hBp_1\|_{L^2_{s_B}},\dots,\|\hBp_N\|_{L^2_{s_B}}\big)$.
\vskip0.2truecm

We address now the contraction property of the map $\cH$. 
Take $\tv_C^{(1)},\tv_C^{(2)}\in\cB_{K\varepsilon^{2-2r}}^{\cX_2}$ and consider $\tw^{(i)}:=\tw(\tv_C^{(i)})$ for $i\in\{1,2\}$ as given by Sec. \ref{Section_w0}-\ref{Section_wR}. We aim to estimate $\|\cH(\tv_C^{(1)})-\cH(\tv_C^{(2)})\|_{\cX_2}$. Clearly, $T_1$ is independent of $\tv_C$ and similarly to \eqref{vCj_NLest} we infer
\begin{equation*}
\begin{split}
\|\big(T_{2}&(\tv_C^{(1)})-T_{2}(\tv_C^{(2)})\big)p_{n_*}\|_{\cX_2}\leq c\eps^{-2r}\bigg[\|\tvB\|_{\cX_2}\|\tw^{(1)}\astB\tvC^{(1)}-\tw^{(2)}\astB\tvC^{(2)}\|_{\cX_2}\\
&+\|\tw^{(1)}\astB\tw^{(1)}\astB\tbv_C^{(1)}-\tw^{(2)}\astB\tw^{(2)}\astB\tbv_C^{(2)}\|_{\cX_2}+\|\tw^{(1)}\astB\tvC^{(1)}\astB\tbv_C^{(1)}-\tw^{(2)}\astB\tvC^{(2)}\astB\tbv_C^{(2)}\|_{\cX_2}\\
&+\!\!\!\sum_{a,b\in\{1,2,3\}}\!\!\big(\|\tv_{B,a}\astB\tv_{B,b}\|_{\cX_2}+\|\tv_{B,a}\astB\tbv_{B,b}\|_{\cX_2}\big)\|\tv_C^{(1)}-\tv_C^{(2)}\|_{\cX_2}\\
&+\big\|\tvB\|_{\cX_2}\|\tvC^{(1)}\astB\tvC^{(1)}-\tvC^{(2)}\astB\tvC^{(2)}\|_{\cX_2}+\|\tvC^{(1)}\astB\tvC^{(1)}\astB\tbv_C^{(1)}-\tvC^{(2)}\astB\tvC^{(2)}\astB\tbv_C^{(2)}\|_{\cX_2}\bigg].
\end{split}
\end{equation*}
All the terms are then estimated in a similar way, e.g.
\begin{equation*}
\begin{split}\|\tw^{(1)}\astB\tvC^{(1)}-\tw^{(2)}\astB\tvC^{(2)}\|_{\cX_2}&\leq \|\tw^{(1)}\|_{\cX_2}\|\tvC^{(1)}-\tvC^{(2)}\|_{\cX_2}+\|\tvC^{(2)}\|_{\cX_2}\|\tw^{(1)}-\tw^{(2)}\|_{\cX_2}\\
&\leq c\varepsilon^2\|\tvC^{(1)}-\tvC^{(2)}\|_{\cX_2}+K\varepsilon^{2-2r}\|\tw^{(1)}-\tw^{(2)}\|_{\cX_2}\\
&\leq c\big(\varepsilon^2+\varepsilon^{3-2r}\big)\|\tvC^{(1)}-\tvC^{(2)}\|_{\cX_2},
\end{split}
\end{equation*}
where we applied Lemma \ref{Lemma_VR_Lip} in the last step. Hence
\begin{equation*}\label{contraction_H:2}
\|\cH(\tv_{C}^{\,(1)})-\cH(\tv_{C}^{\,(2)})\|_{\cX_2}\leq c(\varepsilon^{2-2r}+\varepsilon^{3-4r})\|\tv_C^{(1)}-\tv_C^{(2)}\|_{\cX_2},
\end{equation*}
i.e. a contraction due to $r\in (0,\tfrac{1}{2}].$

As a result, we obtain a solution $\tv_C\in\cB^{\cX_2}_{K\varepsilon^{2-2r}}$ of \eqref{vC_eq}, where $K=K\big(
\|\hB_1\|_{L^2_{s_B}},\dots,\|\hB_N\|_{L^2_{s_B}}\big)$, and in turn the estimates
\begin{equation}\label{vC_estimate}
\|\tvC\|_{\cX_2}\leq c\varepsilon^{2-2r}
\end{equation} 
and
\begin{equation}\label{C_estimate}
\|\hC\|_{L^2(\B)}\leq c\varepsilon^{2-2r}.
\end{equation}
Indeed, \eqref{C_estimate} follows from \eqref{vC_estimate}:
\begin{equation*}
\begin{split}
\|\hC\|_{L^2(\B)}&\leq\big(\min_{k\in\B}\|p_{n_*}(\cdot,k)\|_{H^2(\cQ)}\big)^{-1}\left(\int_\B\big|\hC(k)\big|^2\|p_{n_*}(\cdot,k)\|^2_{H^2(\cQ)}\dd k\right)^{\frac12}\\
&\leq c\|\tvC\|_{\cX_2}\leq c\varepsilon^{2-2r}.
\end{split}
\end{equation*}

%-----------------------------------------------
\subsection{Components \texorpdfstring{$\boldsymbol{\hB_j}$}{Bj-hat}}\label{Section_B}
We finally address the component $\tvB$ of the solution $\tu$ and with $\tw$ and $\tvC$ found above we solve for such $\tvB$, for which $\big(\hB_j\big)_{j=1}^N$ is close to the solutions $\big(\hA_j\big)_{j=1}^N$ of the CMEs \eqref{CMEs_FT}. As a result the component $\tvB$ is the dominant part of the solution $\tu$.

Equation \eqref{v_Eq} on the compact support of $\hBp_j$ can be rewritten as
\begin{equation*}
(\omega_{n_*}^2(k)-\omega^2)\frac1\varepsilon\hBp_j\left(\frac{k-\kj}\varepsilon\right)=\omega^2\langle\tF(\tu)(\cdot,k),p_{n_*}(\cdot,k)\rangle, \quad k \in B_{\eps^r}(k^{(j)})
\end{equation*}
or equivalently
\begin{equation}\label{vB_eq}
(\omega_{n_*}^2(k)-\omega^2)\frac1\varepsilon\hBp_j\left(\frac{k-\kj}\varepsilon\right)=\omega^2\chi_{\varepsilon,r}(k-\kj)\langle\tF(\tu)(\cdot,k),p_{n_*}(\cdot,k)\rangle, \quad k \in\R^2,
\end{equation}
where we define $\chi_{\varepsilon,r}:=\chi_{B_{\eps^r}(0)}$.

Expanding the eigenvalue $\omega_{n_*}$ near $\kj$ by assumptions (A3) and (A4) as
\begin{equation*}
\omega_{n_*}(k)=\omega_*+\frac12(k-\kj)^{\mathsf T}\nabla^2\omega_{n_*}(\kj)(k-\kj)+\omega_r(k),
\end{equation*}
where $|\omega_r(k)|\leq C|k-\kj|^3$, and then recalling \eqref{omega}, we obtain
\begin{equation*}
\omega_{n_*}^2(k)-\omega^2=\varepsilon^2\omega_*\left(\left(\tfrac{k-\kj}\varepsilon\right)^{\mathsf T} \nabla^2\omega_{n_*}(\kj)\left(\tfrac{k-\kj}\varepsilon\right)-2\Omega\right)-\varepsilon^4\Omega^2+\omega_R(k),
\end{equation*}
with
\begin{equation}\label{omega_R}
|\omega_R(k)|\leq C|k-\kj|^3.
\end{equation}
Inserting this into equation \eqref{vB_eq} and defining $k':=\tfrac{k-\kj}\varepsilon$, we obtain for $k'\in B_{\eps^{r-1}}(0)$
\begin{equation}\label{vB_eq:2}
\begin{split}
\omega_*\big((k')^{\mathsf T}\nabla^2\omega_{n_*}(\kj)k'-2\Omega\big)\hBp_j(k')&=\,\frac{\omega^2}\varepsilon\chi_{\varepsilon,r-1}(k')\langle\tF(\tv)(\cdot,\kj+\varepsilon k'),p_{n_*}(\cdot,\kj+\varepsilon k')\rangle\\
&\quad+\frac{\omega^2}\varepsilon\chi_{\varepsilon,r-1}(k')\langle\big(\tF(\tu)-\tF(\tv)\big)(\cdot,\kj+\varepsilon k'),p_{n_*}(\cdot,\kj+\varepsilon k')\rangle\\
&\quad-\frac1\varepsilon\big(\omega_R(\kj+\varepsilon k')-\varepsilon^4\Omega^2\big)\hBp_j(k').
\end{split}
\end{equation}
Now we estimate separately the terms on the right. We will see that the second and third terms are small, while the first one recovers the right-hand side of the CMEs \eqref{CMEs_FT}, so that \eqref{vB_eq:2} may be interpreted as a perturbed CME system. 

First we deal with the third term of \eqref{vB_eq:2}. By \eqref{omega_R},
\begin{equation}\label{omegaR_estimate}
\begin{split}
\left\|\frac{\omega_R(\kj+\varepsilon\cdot)}\varepsilon\hB_j\right\|_{L^2(\varepsilon^{-1}\B)}&=\frac1\varepsilon\left(\int_{\varepsilon^{-1}\B}|\omega_R(\kj+\varepsilon k')|^2|\hB_j(k')|^2\dd k'\right)^\frac12\\
&\leq c\varepsilon^2\left(\int_{B_{\varepsilon^{r-1}}(0)}\frac{|k'|^6}{(1+|k'|)^{2s_B}}(1+|k'|)^{2s_B}|\hB_j(k')|^2\dd k'\right)^\frac12\\
&\leq c\varepsilon^2\sup_{|k'|<\varepsilon^{r-1}}\frac{|k'|^3}{(1+|k'|)^{s_B}}\|\hBp_j\|_{L^2_{s_B}}\\
&\leq c\varepsilon^{2-\max\{0,(1-r)(3-s_B)\}}\|\hBp_j\|_{L^2_{s_B}}.
\end{split}
\end{equation}
To make this term $o(1)$, we need that $2>\max\{0,(1-r)(3-s_B)\}$. This is ensured for all $r\in\big(0,\frac12\big]$ as long as we take $s_B>1$.

The second term in \eqref{vB_eq:2} is estimated similarly as in Sec. \ref{Section_wR}. Indeed,
\begin{equation}\label{FuFv_estimate}
\begin{split}
\bigg\|\frac{\omega^2}\varepsilon\langle\big(&\tF(\tu)-\tF(\tv)\big)(\cdot,\kj+\varepsilon k'),p_{n_*}(\cdot,\kj+\varepsilon k')\rangle\bigg\|_{L^2(\varepsilon^{-1}\B)}\\
&\leq\omega^2\eps^{-1}\esssup_{k\in\B}\|p_{n_*}(\cdot,k)\|_{L^2(\cQ)}\left(\int_{\varepsilon^{-1}\B}\|\big(\tF(\tu)-\tF(\tv)\big)(\cdot,\kj+\varepsilon k')\|^2_{L^2(\cQ)}\dd k'\right)^\frac12\\
&\leq c\eps^{-1}\left(\int_{\B+\kj}\|\big(\tF(\tu)-\tF(\tv)\big)(\cdot, k)\|^2_{H^2(\cQ)}\varepsilon^{-2}\dd k\right)^\frac12\\
&=c\varepsilon^{-2}\|\tF(\tu)-\tF(\tv)\|_{\cX_2}\leq c\varepsilon,
\end{split}
\end{equation}
where the constant $c$ depends just on $\big(\|\hBp_j\|_{L^2_{s_B}}\big)_{j=1}^N$. The last inequality is given by \eqref{G_estimate}.

Let us now address the first term in \eqref{vB_eq:2} on its support $k'\in B_{\eps^{r-1}}(0)$. Equivalently we consider $k\in B_{\eps^r}(k^{(j)})$ and split the term as follows:
\begin{equation}\label{W1W2W3}
\begin{split}
\langle\tF&(\tv)(\cdot,k),p_{n_*}(\cdot,k)\rangle= \langle(\tF(\tv)-\tF(\tv_B))(\cdot,k),p_{n_*}(\cdot,k)\rangle\\
&+\eps^{-3}\sumabgN\int_{2\B}\int_\B \hB_\alpha\left(\frac{k-l-k^{(\alpha)}}{\eps}\right)\hB_\beta\left(\frac{l-t-k^{(\beta)}}{\eps}\right)\hbB_\gamma\left(\frac{t+k^{(\gamma)}}{\eps}\right)\\
&\quad\cdot\sumabcd\big(\CyrB^{a,b,c,d}(k,k-l,l-t,t)-\Theta_{\alpha,\beta,\gamma,j}^{a,b,c,d}\big)\dd t\dd l\\
&+\eps^{-3}\sumabcd\sumabgN\Theta_{\alpha,\beta,\gamma,j}^{a,b,c,d}\int_{2\B}\int_\B \hB_\alpha\left(\frac{k-l-k^{(\alpha)}}{\eps}\right)\hB_\beta\left(\frac{l-t-k^{(\beta)}}{\eps}\right)\\
&\quad\cdot\hbB_\gamma\left(\frac{t+k^{(\gamma)}}{\eps}\right)\dd t\dd l\\
&:=W_1(k)+W_2(k)+W_3(k),
\end{split}
\end{equation}
where
\begin{equation*}\label{CYRB}
\CyrB^{a,b,c,d}(k,k-l,l-t,t):=\langle \underline{\chi}_{a,b,c,d}^{(3)}p_{n_*,a}(\cdot,k-l)p_{n_*,b}(\cdot,l-t)\overline{p_{n_*,c}(\cdot,-t)},p_{n_*,d}(\cdot,k)\rangle
\end{equation*}
and
\begin{equation}\label{Theta}
\begin{split}
\Theta_{\alpha,\beta,\gamma,j}^{a,b,c,d}:&=\langle \underline{\chi}_{a,b,c,d}^{(3)} u_{n_*,a}(\cdot,\kalpha)u_{n_*,b}(\cdot,\kbeta)\overline{u_{n_*,c}(\cdot,\kgamma)},u_{n_*,d}(\cdot,\kj)\rangle\\
&=\langle \underline{\chi}_{a,b,c,d}^{(3)}p_{n_*,a}(\cdot,\kalpha)p_{n_*,b}(\cdot,\kbeta)\overline{p_{n_*,c}(\cdot,\kgamma)}e^{\ri K_{\alpha,\beta,\gamma,j}\,\,\bigcdot},p_{n_*,d}(\cdot,\kj)\rangle\\
&=\CyrB_{\alpha,\beta,\gamma,j}^{a,b,c,d}(\kj,\kalpha-K_{\alpha,\beta,\gamma,j},\kbeta,-\kgamma)
\end{split}
\end{equation}
with $K_{\alpha,\beta,\gamma,j}:=\kalpha+\kbeta-\kgamma-\kj$.

The aim is to show that $\eps^{-1}W_1$ and $\eps^{-1}W_2$ are small and that $\eps^{-1}\tfrac{\omega_*}{2}\chi_{\eps,r}(\cdot-\kj)W_3$ is the Fourier transform of the nonlinear term $\cN_j$ in the CMEs applied to $(\hBp_1,\dots,\hBp_N)$ and evaluated at $\eps^{-1}(\cdot-\kj)$. 

First, $W_1$ is estimated analogously to the term $T_2$ in Sec. \ref{Section_C} producing
\beq\label{E:W1-est}
\|W_1\|_{L^2(\B)}\leq c\eps^{3-2r}.
\eeq
For $W_2$ we take advantage of the Lipschitz continuity of $\CyrB$ and of the asymptotically small support of the double convolution of the $\hB$'s. We rewrite
\begin{equation*}
\begin{split}
W_2(\kj\!+\varepsilon k') =&\,\varepsilon\!\sumabcd\,\sumabgN\int_{\eps^{-1}(2\B+\kgamma-\kbeta)}\int_{\varepsilon^{-1}(\B+\kgamma)}\!\hB_\alpha(k'-l')\hB_\beta(l'-t')\hbB_\gamma(t')\\
&\quad\cdot\big(\CyrB^{a,b,c,d}(\kj+\varepsilon k',\kalpha\!-K_{\alpha,\beta,\gamma,j}\!+\varepsilon(k'-l'),\kbeta\!+\varepsilon(l'-t'),-\kgamma\!+\varepsilon t')\\
&\qquad-\CyrB^{a,b,c,d}(\kj,\kalpha-K_{\alpha,\beta,\gamma,j},\kbeta,-\kgamma)\big)\dd t'\dd l'
\end{split}
\end{equation*}
using the obvious changes of variables and the $\eps^{-1}\Lambda^*$-periodicity of $\hB_\alpha, \alpha\in \{1,\dots,N\}.$
Next, we exploit the fact that the map $(k_1,k_2,k_3,k_4)\mapsto\CyrB^{a,b,c,d}(k_1,k_2,k_3,k_4)$ is Lipschitz continuous with respect to all variables, i.e. there is $C_{\cyrB}>0$ such that for all $k',l',t'\in\eps^{-1}\B$
\begin{equation*}
\begin{split}
\big|\CyrB(\kj\!+\varepsilon &k',\kalpha\!-K_{\alpha,\beta,\gamma,j}\!+\varepsilon(k'-l'),\kbeta\!+\varepsilon(l'-t'),-\kgamma\!+\varepsilon t')-\CyrB(\kj,\kalpha\!-K_{\alpha,\beta,\gamma,j},\kbeta,-\kgamma)\big|\\
&\quad\leq C_{\cyrB}\varepsilon\big(|k'|+|k'-l'|+|l'-t'|+|t'|\big)\\
&\quad\leq 2C_{\cyrB}\varepsilon\big(|k'-l'|+|l'-t'|+|t'|\big),
\end{split}
\end{equation*}
where we have omitted the indices of $\CyrB$ for brevity.
Therefore
\begin{equation}\label{W2_est}
\begin{split}
\|W_2(\kj\!+\varepsilon\cdot)\|_{L^2(\varepsilon^{-1}\B)}&\leq c \varepsilon^2\!\sumabgN\big(2\big\||\tau\hBp_\alpha|\ast|\hBp_\beta|\ast|\hbBp_\gamma|\big\|_{L^2}+\big\||\hBp_\alpha|\ast|\hBp_\beta|\ast|\tau\hbBp_\gamma\big|\|_{L^2}\big)\\
&\leq c\eps^2\!\sumabgN\|\hBp_\alpha\|_{L^2_1}\|\hBp_\beta\|_{L^1}\|\hBp_\gamma\|_{L^1}\leq c\eps^2\!\sumabgN\|\hBp_\alpha\|_{L^2_s}\|\hBp_\beta\|_{L^2_s}\|\hBp_\gamma\|_{L^2_s}
\end{split}
\end{equation}
for any $s>1$, where $\tau(k):=k$. Because $\|W_2\|_{L^2(\B)}=\eps \|W_2(\kj+\eps\cdot)\|_{L^2(\eps^{-1}\B)}$, we  get
from \eqref{E:W1-est} and \eqref{W2_est}
\beq\label{E:W1W2-est}
\eps^{-1}\|W_1+W_2\|_{L^2(\B)}\leq c(\eps^{2-2r}+\eps)\leq c\eps
\eeq
as $r\in \left(0,\frac12\right]$.

%----------------------------------------
\subsubsection{Perturbed CMEs}\label{S:perturbedCME}
We return to equation \eqref{vB_eq:2}. By \eqref{omegaR_estimate}, \eqref{FuFv_estimate}, \eqref{W1W2W3} and \eqref{E:W1W2-est} we get for each $j\in  \{1,\dots,N\}$ and $k'\in B_{\eps^{r-1}}(0)$ 
\begin{equation}\label{vB_eq:3}
\begin{split}
\omega_*\big((k')^{\mathsf T}&\nabla^2\omega_{n_*}(\kj)k'-2\Omega\big)\hBp_j(k')=\frac{\omega^2}\varepsilon\chi_{\varepsilon,r-1}(k')\!\sumabcd\,\sumabgN\Theta_{\alpha,\beta,\gamma,j}^{a,b,c,d}\\
&\cdot\frac1{\varepsilon^3}\bigg(\hB_\alpha\left(\tfrac{\cdot-\kalpha}\varepsilon\right)\astB\hB_\beta\left(\tfrac{\cdot-\kbeta}\varepsilon\right)\astB\hbB_\gamma\left(\tfrac{\cdot-\kgamma}\varepsilon\right)\bigg)(\kj+\varepsilon k')+\Zhe_j(\kj+\varepsilon k'),
\end{split}
\end{equation}
where $\Zhe_j$ collects all the perturbations. Since $r\in\big(0,\tfrac12\big]$, it is
\begin{equation}\label{Zhe_error_from_CME_beforeconditions}
\begin{split}
\|\Zhe_j(\kj+\varepsilon\cdot)\|_{L^2(\varepsilon^{-1}\B)}&\leq c\big(\varepsilon^{2-\max\{0,(1-r)(3-s_B)\}}+\varepsilon+\varepsilon^{2-2r}\big)\\
&\leq c\big(\varepsilon^{2-\max\{0,(1-r)(3-s_B)\}}+2\varepsilon\big).
\end{split}
\end{equation}
Prescribing now
\begin{equation}\label{r_sC_sB_conditions}
s_B\geq3-\frac1{1-r},
\end{equation}
we see that the first exponent in \eqref{Zhe_error_from_CME_beforeconditions} is greater or equal than $1$. Therefore, under condition \eqref{r_sC_sB_conditions} we get
\begin{equation}\label{Zhe_error_from_CME}
\|\Zhe_j(\kj+\varepsilon\cdot)\|_{L^2(\varepsilon^{-1}\B)}\leq c\varepsilon.
\end{equation}
Note that \eqref{r_sC_sB_conditions} is satisfied e.g. by $s_B=2$ since $r>0$.

Next, a direct calculation shows that 
\begin{equation}\label{BBB}
\begin{split}
\chi_{\varepsilon,r-1}&(k')\frac1{\varepsilon^3}\left(\hB_\alpha\left(\tfrac{\cdot-\kalpha}\varepsilon\right)\astB\hB_\beta\left(\tfrac{\cdot-\kbeta}\varepsilon\right)\astB\hbB_\gamma\left(\tfrac{\cdot-\kgamma}\varepsilon\right)\right)(k)\\
&=\begin{cases}
\eps(\hBp_\alpha *_{B_{\eps^{r-1}}(0)}\hBp_\beta*_{B_{\eps^{r-1}}(0)} \hbB^*_\gamma)(k')=\eps(\hBp_\alpha *\hBp_\beta * \hbB^*_\gamma)(k') \quad &\text{if } \; (\alpha,\beta,\gamma)\in \sigma_j\\
0 & \text{if } \; (\alpha,\beta,\gamma)\notin \sigma_j.
\end{cases}
\end{split}
\end{equation}
In detail: due to the periodicity of $\hB_\alpha$ the convolution $\astB$ can be replaced by $\ast_{\B+k_*}$ for $k_*\in\R^2$ arbitrary. This implies by the obvious change of variables that the left hand side equals
\begin{equation*}
\begin{split}
&\frac1{\varepsilon^3}\int_{\B-\kgamma}\left(\hB_\alpha\left(\frac{\cdot-\kalpha}\varepsilon\right)\ast_{\B+\kbeta}\hB_\beta\left(\frac{\cdot-\kbeta}\varepsilon\right)\right)(k-l)\,\hbB_\gamma\left(\frac{l+\kgamma}\varepsilon\right)\dd l\\
&=\varepsilon^{-1}\int_{B_{\varepsilon^{r-1}}(0)}\int_{\B+\kbeta}\hB_\alpha\left(\frac{k-\varepsilon l'-\kalpha+\kgamma-s}\varepsilon\right)\hB_\beta\left(\frac{s-\kbeta}\varepsilon\right)\hbB_\gamma(l')\dd s\dd l'\\
&=\varepsilon\int_{B_{\varepsilon^{r-1}}(0)}\int_{B_{\varepsilon^{r-1}}(0)}\hB_\alpha\left(\frac{k-(\kalpha+\kbeta-\kgamma)}\varepsilon-l'-s'\right)\hB_\beta(s')\hbB_\gamma(l')\dd s'\dd l'.
\end{split}
\end{equation*}
With the further transformation $l'+s'=:t'\in B_{2\varepsilon^{r-1}}(0)$, we infer
\begin{equation*}
=\varepsilon\int_{B_{2\varepsilon^{r-1}}(0)}\int_{B_{\varepsilon^{r-1}}(0)}\hB_\alpha\left(\frac{k-\kj-K_{\alpha,\beta,\gamma,j}}\varepsilon-t'\right)\hBp_\beta\left(t'-l'\right)\hbBp_\gamma\left(l'\right)\dd l'\dd t'.
\end{equation*}

Now recall that $k\in B_{\eps^r}(\kj)$ and that $\supp (\hBp_\alpha) \subset B_{\eps^{r-1}}(0)$. Due to $t'\in B_{2\eps^{r-1}}(0)$ and the $\eps^{-1}\Lambda^*$-periodicity of $\hB_\alpha$, the function $\hB_\alpha\left(\frac{k-\kj-K_{\alpha,\beta,\gamma,j}}\varepsilon-t'\right)$ is nonzero if and only if $K_{\alpha,\beta,\gamma,j}\in \Lambda^*$, i.e. if $(\alpha,\beta,\gamma)\in \sigma_j$, with $\sigma_j$ defined in \eqref{sigma_j}. The periodicity allows then for dropping the shift $\eps^{-1}K_{\alpha,\beta,\gamma,j}$ in the argument of $\hB_\alpha$. Moreover, for $k\in B_{\eps^r}(\kj)$ and $t'\in B_{2\eps^{r-1}}(0)$ it is $\hB_\alpha\left(\frac{k-\kj}\varepsilon-t'\right)=\hBp_\alpha\left(\frac{k-\kj}\varepsilon-t'\right)$. We get
\begin{equation*}
\begin{aligned}
&\chi_{\eps,r-1}(k')\frac1{\varepsilon^3}\left(\hB_\alpha\left(\tfrac{\cdot-\kalpha}\varepsilon\right)\astB\hB_\beta\left(\tfrac{\cdot-\kbeta}\varepsilon\right)\astB\hbB_\gamma\left(\tfrac{\cdot-\kgamma}\varepsilon\right)\right)(\kj +\eps k')\\
&=\begin{cases}
\eps\chi_{\eps,r-1}(k')\int_{B_{2\varepsilon^{r-1}}(0)}\int_{B_{\varepsilon^{r-1}}(0)}\hB_\alpha\left(k'-t'\right)\hBp_\beta\left(t'-l'\right)\hbBp_\gamma\left(l'\right)\dd l'\dd t' \quad &\text{if } \; (\alpha,\beta,\gamma)\in \sigma_j\\
0 & \text{if } \; (\alpha,\beta,\gamma)\notin \sigma_j,
\end{cases}
\end{aligned}
\end{equation*}
such that \eqref{BBB} follows.

Hence, by \eqref{vB_eq:3},\eqref{Zhe_error_from_CME}, and \eqref{BBB} we deduce that $\big(\hBp_j\big)_{j=1}^N$ satisfy the perturbed CME system
\begin{equation}\label{vB_eq:final}
\cG_j(\hBp)(k'):=\!\left(\frac12(k')^{\mathsf T}\nabla^2\omega_{n_*}(\kj)k'-\Omega\right)\!\hBp_j(k')-\hcN_j(\hBp)(k')=\hR_j(\hBp)(k'), \;\,\,\, k'\in B_{\eps^{r-1}}(0)
\end{equation}
for $j\in\{1,\dots,N\}$ and $k'\in B_{\varepsilon^{r-1}}(0)$, where we recall that (cf. \eqref{Nj} and \eqref{Theta})
\begin{equation*}\label{hNj}
\begin{split}
\hcN_j(\hBp)(k')&=\,\frac{\omega_*}2\!\!\sumabcd\,\sumabgN\!\!\Theta_{\alpha,\beta,\gamma,j}^{a,b,c,d}\left(\hBp_\alpha *_{B_{\eps^{r-1}}(0)}\hBp_\beta*_{B_{\eps^{r-1}}(0)} \hbB^*_\gamma\right)(k')\\
&=\,\frac{\omega_*}2\!\!\sum_{(\alpha,\beta,\gamma)\in \sigma_j}\!\!I_{\alpha,\beta,\gamma}^{\,j}\left(\hBp_\alpha*\hBp_\beta*\hbB^*_\gamma\right)(k'),
\end{split}
\end{equation*}
the coefficients $I_{\alpha,\beta,\gamma}^{\,j}$ being defined in \eqref{I_abg^j}. The remainder term $\hR_j(\hBp)=\hR_j\big(\hBp_1,\dots,\hBp_N\big)$ is defined via
\begin{equation}\label{R}
\hR_j(\hBp)(k'):=\frac{1}{2\omega_*}\Zhe_j(\kj+\varepsilon k'), \qquad k'\in B_{\eps^{r-1}}(0)
\end{equation}
and satisfies
\begin{equation}\label{Remainder_error}
\|\hR_j(\hBp)\|_{L^2(\varepsilon^{-1}\B)}<c\varepsilon.
\end{equation}
Notice that \eqref{vB_eq:final} is therefore an $\varepsilon$-perturbation of the CMEs in Fourier variables on the compact support $k'\in B_{\eps^{r-1}}(0)$. In what follows, we prove the existence of solutions $\hBp:=\big(\hBp_j\big)_{j=1}^N$ of \eqref{vB_eq:final} close to $\chi_{\varepsilon,r-1}\hA$, where $\hA:=\big(\hA_j\big)_{j=1}^N$ is the Fourier transform of the solution of the CMEs \eqref{CMEs_phys}. We follow the approach of \cite{DP,DW}.

To this aim, for $j\in\{1,\dots,N\}$ we define $\Aej:=\chi_{\varepsilon,r-1}\hA_j$ and write 
$$\hBp_j=\Aej+\hbj$$ 
with $\supp(\hbj)\subset B_{\varepsilon^{r-1}}(0)$. In order to expand $\cG=\big(\cG_1,\dots,\cG_N\big)$ around the vector $\Ae$ and use the Jacobian of the CMEs, we write $\cG_j$ in the real variables. Indices $R$ and $I$ denote hereafter the real and the imaginary part respectively, e.g. $A_j=A_{j,R}+\ri A_{j,I}$. We define (cf. \eqref{CMEs_phys})
\begin{equation*}
\phi_j(A):=-\frac12\nabla^\trans(\nabla^2\omega_{n_*}(\kj)\nabla A_j)-\Omega A_j-\cN_j(A),\qquad j\in\{1,\dots,N\}
\end{equation*}
so that $\big(\phi(A)\big)^\land=\cG(\hA)$, which in real variables becomes
\begin{equation*}
\Phi_j(A_R,A_I):=\begin{pmatrix}
\Real(\phi_j(A_R+\ri A_I))\\\Imag(\phi_j(A_R+\ri A_I))
\end{pmatrix},\qquad j\in\{1,\dots,N\}.
\end{equation*}
We denote its $2N\times2N$ Jacobian by $D\Phi(A_R,A_I)$, its Fourier counterpart by
\begin{equation*}
D_{\hA}\cG(\hA):=\big(D\Phi(A_R,A_I)\big)^\land,
\end{equation*}
as well as its Fourier-truncation
\begin{equation*}
\chi_{\varepsilon,r-1}D_{\hA}\cG(\Ae)=\chi_{\varepsilon,r-1}\big(D\Phi(A_R^\varepsilon,A_I^\varepsilon)\big)^\land,
\end{equation*}
with $A_R^\varepsilon:=(\Ae_R)^\lor$ and $A_I^\varepsilon:=(\Ae_I)^\lor$. Thus here $D_{\hA}\cG_j(\hA)$ is just a symbolic notation. Recalling the definition of $\cN_j$ in \eqref{Nj}, we have
\begin{equation*}
\Real(\cN_j(A))=\sum_{(\alpha,\beta,\gamma)\in \sigma_j} I_{\alpha,\beta,\gamma}^{\,j}\big(A_{\alpha,R}A_{\beta,R}A_{\gamma,R}+A_{\alpha,R}A_{\beta,I}A_{\gamma,I}+A_{\alpha,I}A_{\beta,R}A_{\gamma,I}-A_{\alpha,I}A_{\beta,I}A_{\gamma,R}\big),
\end{equation*}
\begin{equation*}
\Imag(\cN_j(A))=\sum_{(\alpha,\beta,\gamma)\in \sigma_j} I_{\alpha,\beta,\gamma}^{\,j}\big(-A_{\alpha,R}A_{\beta,R}A_{\gamma,I}+A_{\alpha,R}A_{\beta,I}A_{\gamma,R}+A_{\alpha,I}A_{\beta,R}A_{\gamma,R}+A_{\alpha,I}A_{\beta,I}A_{\gamma,I}\big).
\end{equation*}
Therefore, for $m\in \{1,\dots,N\}$
\begin{equation*}
\begin{split}
\partial_{A_{m,R}}\Real(\cN_j(A))&=\sumabN\nu_{\alpha \beta jm}^{11,R} A_{\alpha,R}A_{\beta,R}+\sumabN\nu_{\alpha \beta jm}^{11,I}A_{\alpha,I}A_{\beta,I},\\
\partial_{A_{m,I}}\Real(\cN_j(A))&=\partial_{A_{m,R}}\Imag(\cN_j(A))
=\sumabN\nu_{\alpha \beta jm}^{12}A_{\alpha,R}A_{\beta,I},\\
\partial_{A_{m,I}}\Imag(\cN_j(A))&=\sumabN\nu_{\alpha \beta jm}^{22,R}A_{\alpha,R}A_{\beta,R}+\sumabN\nu_{\alpha \beta jm}^{22,I}A_{\alpha,I}A_{\beta,I},
\end{split}
\end{equation*}
where the coefficients $\nu_{\alpha \beta jm}^{11,R},\dots,\nu_{\alpha \beta jm}^{22,I}$ are linear combinations of $I_{\alpha,\beta,m}^{\,j}$ for all $\alpha,\beta\in\{1,\dots,N\}$.
Hence we may write
\begin{equation*}
D\Phi(A_R,A_I)=L-D_A\cN(A),
\end{equation*}
where
\begin{equation*}
L:=-\frac12\begin{pmatrix}
\nabla^\trans\big(\nabla^2\omega_{n_*}(k^{(1)})\nabla\big)\Id_{2\times2} &  &  \\ 
& \ddots &  \\ 
&  & \nabla^\trans\big(\nabla^2\omega_{n_*}(k^{(N)})\nabla\big)\Id_{2\times2}
\end{pmatrix} - \Omega\Id_{2N\times2N}
\end{equation*}
and $D_A\cN(A)$ is a block matrix with the $(j,k)$-th block ($j,k\in \{1,\dots,N\}$) being
\begin{equation*}
\begin{split}
&M^{\cN}(A)_{j,m}:=\left(\begin{array}{cc}%{c|c}
\partial_{A_{m,R}}\Real(\cN_j(A))& \partial_{A_{m,I}}\Real(\cN_j(A))\\ %\hline
\partial_{A_{m,R}}\Imag(\cN_j(A))& \partial_{A_{m,I}}\Imag(\cN_j(A))
\end{array}\right)\\
&=\left(\begin{array}{cc}%{c|c}
\sum_{\alpha,\beta=1}^N\big(\nu_{\alpha \beta jm}^{11,R}\,A_{\alpha,R}A_{\beta,R}+\nu_{\alpha \beta jm}^{11,I}\,A_{\alpha,I}A_{\beta,I}\big)& \sum_{\alpha,\beta=1}^N\nu_{\alpha \beta jm}^{12}\,A_{\alpha,R}A_{\beta,I}\\ %\hline
\sum_{\alpha,\beta=1}^N\nu_{\alpha \beta jm}^{21}\,A_{\alpha,R}A_{\beta,I}& \sum_{\alpha,\beta=1}^N\big(\nu_{\alpha \beta jm}^{22,R}\,A_{\alpha,R}A_{\beta,R}+\nu_{\alpha \beta jm}^{22,I}\,A_{\alpha,I}A_{\beta,I}\big)
\end{array}\right).
\end{split}
\end{equation*}
In Fourier variables this rewrites as
\begin{equation}\label{DAFA_Bloch}
D_{\hA}\cG(\hA)=\hL-D_{\hA}\hcN(\hA),
\end{equation}
where $\hL$ is a block-diagonal matrix with $N$ blocks of size 2x2, where the $j$-th block is $(\tfrac12(k')^\trans\nabla^2\omega_{n_*}(\kj)k'-\Omega)\Id_{2\times2}$ and $D_{\hA}\hcN(\hA)$ is a block matrix with the $(j,m)$-th block ($j,m\in\{1,\dots,N\}$) being
\begin{equation*}
\begin{split}
&\widehat{M}^{\cN}(\hA)_{j,m}=\\
&\left(\begin{array}{cc}%{c|c}
\sum_{\alpha,\beta=1}^N\big(\nu_{\alpha \beta jm}^{11,R}\,\hA_{\alpha,R}*\hA_{\beta,R}+\nu_{\alpha \beta jm}^{11,I}\,\hA_{\alpha,I}*\hA_{\beta,I}\big)& \sum_{\alpha,\beta=1}^N\nu_{\alpha \beta jm}^{12}\,\hA_{\alpha,R}*\hA_{\beta,I}\\ %\hline
\sum_{\alpha,\beta=1}^N\nu_{\alpha \beta jm}^{21}\,\hA_{\alpha,R}*\hA_{\beta,I}& \sum_{\alpha,\beta=1}^N\big(\nu_{\alpha \beta jm}^{22,R}\,\hA_{\alpha,R}*\hA_{\beta,R}+\nu_{\alpha \beta jm}^{22,I}\,\hA_{\alpha,I}*\hA_{\beta,I}\big)
\end{array}\right).
\end{split}
\end{equation*}
The action of $\hL$ is multiplicative but $\widehat{M}^{\cN}(\hA)$ acts as a convolution operator, e.g. $(\hA_{\alpha,R}*\hA_{\beta,R})(\hb_{\alpha,R})=\hA_{\alpha,R}*\hA_{\beta,R}*\hb_{\alpha,R}$. If $\hA\in L^2_{s_A}(\R^2)$ with $s_A>1$, then 
$$D_{\hA}\cG(\hA):L^2_2(\R^2)\to L^2(\R^2).$$
For $\hL$ this follows from the second order property of $L$. For $\widehat{M}^{\cN}(\hA)$ we have, e.g. 
 $$\|\hA_{\alpha,R}*\hA_{\beta,R}*\hb_{\alpha,R}\|_{L^2}\leq \|\hA_{\alpha,R}\|_{L^1}\|\hA_{\beta,R}\|_{L^1}\|\hb_{\alpha,R}\|_{L^2}\leq \|\hA_{\alpha,R}\|_{L^2_{s_A}}\|\hA_{\beta,R}\|_{L^2_{s_A}}\|\hb_{\alpha,R}\|_{L^2_2}$$
using Young's inequality for convolutions and \eqref{E:L1-L2s}.

From \eqref{vB_eq:final} and using a Taylor expansion of $\cG(\Ae+\hb)$, we deduce then the following system of equations for the error term $\hb$,
\begin{equation}\label{b_eq}
\begin{split}
\chi_{\varepsilon,r-1}D_{\hA}\cG(\hA)\,\hb&=\chi_{\varepsilon,r-1}\hR(\Ae+\hb)-\chi_{\varepsilon,r-1}\left(\cG(\Ae+\hb)-D_{\hA}\cG(\hA)\,\hb\right)\\
&=\chi_{\varepsilon,r-1}\hR(\Ae+\hb)-\chi_{\varepsilon,r-1}\left(\cG(\Ae)+\left(D_{\hA}\cG(\Ae)-D_{\hA}\cG(\hA)\right)\hb+g(\hb)\right)\\
&=:\cW(\hb),
\end{split}
\end{equation}
where $g$ is quadratic in $\hb$. Once more, we want to apply a fixed point argument to \eqref{b_eq} on a small ball around the origin in $\big(L^2_{s_B}(\R^2)\big)^N$. Hence we need to estimate the terms in $\cW$. First, using the assumption that $\big(\hA_j\big)_{j=1}^N$ solves the CMEs \eqref{CMEs_FT}, for $k'\in B_{\varepsilon^{r-1}}(0)$ we have
\begin{equation*}
\begin{split}%\cG_{\text{CME},j}
\cG_j(\Ae)(k')&=\left(\frac12(k')^{\mathsf T}\nabla^2\omega_{n_*}(\kj)k'-\Omega\right)\hA_j(k')-\hcN_j(\Ae)(k')\\
&=\hcN_j(\hA)(k')-\hcN_j(\Ae)(k')\\
&=\frac{\omega_*}{2}\sum_{(\alpha,\beta,\gamma)\in \sigma_j} I_{\alpha,\beta,\gamma}^{\,j}\left(\hA_\alpha\ast\hA_\beta\ast\hbA_\gamma-\Ae_\alpha\ast\Ae_\beta\ast\Ae_\gamma\right)(k').
\end{split}
\end{equation*}
Notice that the right-hand side includes terms which are double convolutions between $\Aej$ and $\haej:=\hA_j-\Aej=\big(1-\chi_{\varepsilon,r-1}\big)\hA_j$ with at least one occurrence of $\haej$. 
Since for $k'\in\R^2\setminus B_{\varepsilon^{r-1}(0)}$ there holds
\begin{equation*}\label{estimate_ae}
\begin{split}
|\haej(k')|&\leq (1+|k'|)^{s_A}|\haej(k')|\sup_{|k'|>\varepsilon^{r-1}}(1+|k'|)^{-s_A}\leq c\varepsilon^{s_A(1-r)}(1+|k'|)^{s_A}|\hA_j(k')|,
\end{split}
\end{equation*}
we have by Young's inequality for convolutions and \eqref{E:L1-L2s}
\begin{equation*}
\begin{split}
\|\hae_\alpha\ast\Ae_\beta\ast\bAegamma\|_{L^2(\R^2)}&\leq \|\hae_\alpha\|_{L^2(\R^2)}\|\Ae_\beta\|_{L^1(\R^2)}\|\Ae_\gamma\|_{L^1(\R^2)}\\
&\leq c\varepsilon^{s_A(1-r)}\|\hA_\alpha\|_{L^2_{s_A}(\R^2)}\|\hA_\beta\|_{L^2_{s_A}(\R^2)}\|\hA_\gamma\|_{L^2_{s_A}(\R^2)}
\end{split}
\end{equation*}
for $s_A>1$, and similarly one may handle all other terms, because again by \eqref{E:L1-L2s} one has $\|\haej\|_{L^1}\leq\|\haej\|_{L^2_{s_A}}\leq\|\hA_j\|_{L^2_{s_A}}$. 
Hence
\begin{equation}\label{FjAe}
\|\cG_j(\Ae)\|_{L^2(\R^2)}\leq c\!\sumabgN\varepsilon^{s_A(1-r)}\|\hA_\alpha\|_{L^2_{s_A}(\R^2)}\|\hA_\beta\|_{L^2_{s_A}(\R^2)}\|\hA_\gamma\|_{L^2_{s_A}(\R^2)}.
\end{equation}

Next, we estimate the difference of the Jacobians in \eqref{b_eq}. Since the linear part of them (cf.\eqref{DAFA_Bloch}) is the same for $\Ae$ and $\hA$, we get
\begin{equation*}
\begin{split}
&\big\|\chi_{\varepsilon,r-1}\left(D_{\hA}\cG(\Ae)-D_{\hA}\cG(\hA)\right)\hb\big\|_{L^2(\R^2)}\leq\sum_{j,m=1}^N\big\|\chi_{\varepsilon,r-1}\left((\widehat{M}^{\cN}(\Ae))_{j,m}-(\widehat{M}^{\cN}(\hA))_{j,m}\right)\hb_m\big\|_{L^2(\R^2)}\\
&\leq\sum_{\alpha,\beta,j,m=1}^N\bigg[\big\|\nu_{\alpha \beta jm}^{11,R}\left(\Ae_{\alpha,R}\ast\Ae_{\beta,R}-\hA_{\alpha,R}\ast\hA_{\beta,R}\right)\ast\hb_{m,R}\big\|_2+\big\|\nu_{\alpha \beta jm}^{11,I} \left(\Ae_{\alpha,I}\ast\Ae_{\beta,I}-\hA_{\alpha,I}\ast\hA_{\beta,I}\right)\ast\hb_{m,R}\big\|_2\\
&\quad+\big\|\nu_{\alpha \beta jm}^{12} \left(\Ae_{\alpha,R}\ast\Ae_{\beta,I}-\hA_{\alpha,R}\ast\hA_{\beta,I}\right)\ast\hb_{m,I}\big\|_2+\big\|\nu_{\alpha \beta jm}^{12} \left(\Ae_{\alpha,R}\ast\Ae_{\beta,I}-\hA_{\alpha,R}\ast\hA_{\beta,I}\right)\ast\hb_{m,R}\big\|_2\\
&\quad+\big\|\nu_{\alpha \beta jm}^{22,R} \left(\Ae_{\alpha,R}\ast\Ae_{\beta,R}-\hA_{\alpha,R}\ast\hA_{\beta,R}\right)\ast\hb_{m,I}\big\|_2+\big\|\nu_{\alpha \beta jm}^{22,I} \left(\Ae_{\alpha,I}\ast\Ae_{\beta,I}-\hA_{\alpha,I}\ast\hA_{\beta,I}\right)\ast\hb_{m,I}\big\|_2\bigg].
\end{split}
\end{equation*}
We see that all terms are of same kind and moreover are linear in $\hb$ and either linear or quadratic in $\hae:=\Ae-\hA$. Applying then estimates similar to the ones used to deduce \eqref{FjAe}, we infer
\begin{equation}\label{Jacobians:estimate}
\big\|\chi_{\varepsilon,r-1}\left(D_{\hA}\cG(\Ae)-D_{\hA}\cG(\hA)\right)\hb\big\|_{L^2(\R^2)}\leq c\varepsilon^{s_A(1-r)}\!\!\sumabgN\!\|\hA_\alpha\|_{L^2_{s_A}(\R^2)}\|\hA_\beta\|_{L^2_{s_A}(\R^2)}\|\hb\|_{L^2_{s_B}(\R^2)}
\end{equation}
if $s_A>1$.
Combining \eqref{Remainder_error} (where note that the dependence of $c$ on $\|\hB\|_{L^2_{s_B}}$ - and in turn on $\|\hb\|_{L^2_{s_B}}$ - is polynomial), \eqref{FjAe} and \eqref{Jacobians:estimate}, we can thus conclude from \eqref{b_eq} that 
\begin{equation}\label{W}
\|\cW(\hb)\|_{L^2(\R^2)}\leq c_A\left(\varepsilon+\varepsilon^{s_A(1-r)}+(\varepsilon+\varepsilon^{s_A(1-r)})\|\hb\|_{L^2_{s_B}}+\|\hb\|_{L^2_{s_B}}^2+\|\hb\|_{L^2_{s_B}}^3\right).
\end{equation}
In order to solve \eqref{b_eq} for $\hb$ by a fixed point argument, we would need the invertibility of the Jacobian $D_{\hA}\cG(\hA):L^2_2(\R^2)\to L^2(\R^2)$. Indeed, from this it would follow that $\chi_{\varepsilon,r-1}D_{\hA}\cG(\hA)\chi_{\varepsilon,r-1}:L^2_2(\R^2)\to L^2(\R^2)$  is uniformly invertible, see \cite[Theorem IV.3.17]{Kato}. However, this is not the case, because of the presence of the three zero eigenvalues of $D\Phi(A_R,A_I)$ produced by the two spacial shift invariances and  the complex phase invariance of the CMEs \eqref{CMEs_phys}. To eliminate the zero eigenvalues, we assume the non-degeneracy of $A$, see Definition \ref{D:nondegen}, and work (in Fourier variables) in a subspace of $L^2_2(\R^2)$ in which the invariances do not hold. A natural subspace is the one generated by the $\PT$-symmetry, i.e. we work with $\hA$ and $\hb$ such that
\begin{equation*}
A(-x)=\overline{A(x)},\qquad\mbox{and}\qquad b(-x)=\overline{b(x)}
\end{equation*}
or equivalently, 
$$\hA:\R^2\to \R^N, \quad \hb:\R^2\to \R^N.$$

Under the non-degeneracy condition, the Jacobian $D_{\hA}\cG(\hA)$ is invertible in such a subspace and we can apply a fixed point argument to equation \eqref{b_eq}. In detail, assuming $\hA:\R^2\to \R^N$, we look for a solution of 
\begin{equation}\label{b:eq^-1}
\hb=\big(\chi_{\varepsilon,r-1}D_{\hA}\cG(\hA)\big)^{-1}\cW(\hb)
\end{equation}
in the space 
\begin{equation*}
\begin{split}
\boldsymbol{L}^2_{2,\text{sym}}:
&=\big\{\hb\in L^2_{2}(\R^2)^{N}\,|\,\supp(\hb)\subset B_{\varepsilon^{r-1}},\,\hb(\cdot)\,\,\mbox{is real}\big\}.
\end{split}
\end{equation*}
However, we need to make sure that the $\PT$-symmetry is preserved by the maps $\cW$ and $\big(\chi_{\varepsilon,r-1}D_{\hA}\cG(\hA)\big)^{-1}$. This is proved at the end of the section. We address now the application of the fixed point argument to \eqref{b:eq^-1} in the ball
\begin{equation*}
\cB^{2,\text{sym}}_{c\varepsilon^\rho}:=\big\{\hb\in\boldsymbol{L}^2_{2,\text{sym}}\,|\,\|\hb\|_{L^2_2}\leq c\varepsilon^\rho\big\},
\end{equation*}
where $c,\rho>0$ have to be found. For $\hb\in\cB^{2,\text{sym}}_{c\varepsilon^\rho}$ we deduce from \eqref{W} that
\begin{equation*}
\|\cW(\hb)\|_{L^2(\R^2)}\leq c_A\left((\varepsilon+\varepsilon^{s_A(1-r)})(1+c\varepsilon^\rho)+c^2\varepsilon^{2\rho}+c^3\varepsilon^{3\rho}\right).
\end{equation*}
Choosing
\begin{equation}\label{rho_c}
\rho=\min\{1,s_A(1-r)\}\qquad\mbox{and}\qquad c=2c_A,
\end{equation}
we infer $\cW(\hb)\in\cB^{2,\text{sym}}_{c\varepsilon^\rho}$. Moreover, the map $\hb\mapsto\cW(\hb)$ is contractive in such a ball. Indeed, for $\hb_1,\hb_2\in\cB^{2,\text{sym}}_{c\varepsilon^\rho}$,
\begin{equation*}
\begin{split}
\|\cW(\hb^{(1)})-\cW(\hb^{(2)})\|_{L^2(\R^2)}&\leq\|g(\hb^{(1)})-g(\hb^{(2)})\|_2+\|\chi_{\varepsilon,r-1}\big(\hR(\Ae+\hb^{(1)})-\hR(\Ae+\hb^{(2)})\big)\|_2\\
&\quad+\left\|\chi_{\varepsilon,r-1}\big(D_{\hA}\cG(\Ae)-D_{\hA}\cG(\hA)\big)(\hb^{(1)}-\hb^{(2)})\right\|_2\\
&\leq c(\varepsilon+\varepsilon^{s_A(1-r)})\|\hb^{(1)}-\hb^{(2)}\|_2,
\end{split}
\end{equation*}
because of \eqref{Remainder_error}, \eqref{Jacobians:estimate} and of the quadratic nature of $g$. Since $\chi_{\varepsilon,r-1}D_{\hA}\cG(\hA):\boldsymbol{L}^2_{2,\text{sym}}\to \boldsymbol{L}^2(\R^2)$ is boundedly invertible, the existence of a $\PT$-symmetric solution $\hb\in \boldsymbol{L}^2_{2,\text{sym}}$ of equation \eqref{b_eq} so that
\begin{equation}\label{b:estimate}
\|\hb\|_{L^2_2}\leq 2c_A\varepsilon^{\min\{1,s_A(1-r)\}}
\end{equation}
follows from the Banach fixed point theorem. Notice that the optimal estimate $\|\hb\|_{L^2_2}\leq 2c_A\varepsilon$ can be obtained for any $s_A>1$ as $r\in (0,\tfrac12\big]$ can be chosen arbitrarily small.

To conclude the argument, it remains to be proved that the $\PT$-symmetry is preserved by the maps $\cW$ and $\big(\chi_{\varepsilon,r-1}D_{\hA}\cG(\hA)\big)^{-1}$, i.e. that 
they map real valued functions $\hb$ to real valued functions. First, note that $u$ is $\PT$-symmetric if and only if $\tu(\cdot,k)$ is so for almost all $k\in\B$. Hence, we can check the inheritance of the property in the Bloch setting. We now need to make sure that all the components in which we decomposed our solution, and which now depend just on $\hb$, inherit the $\PT$-symmetry. If so, then the residual term $\hR$ in \eqref{R} is real. To complete this step, analyzing the equations that $\tw_0$, $\tw_R$ and $\hC$ have to fulfill, namely \eqref{w0_eq}, \eqref{wR_eq} and \eqref{C_eq}, we see that we just need that our operator $L_k$, the projections $\eQk$ and $\Qke$, and the nonlinear map $\tF$ commute with $\PT$. In detail:
\begin{itemize}
	\item $L_k$ is $\PT$-symmetric since it involves only derivatives of order $2$ and $0$ and by assumption (A6).
	\item By the simpleness assumption (A7), the Bloch eigenfunctions $p_{n_*}(\cdot,k)$ are $\PT$-symmetric for almost all $k\in\B$, see \eqref{E:sym-minus-x}. This, together with (A6), implies that the projections $\ePk,\eQk,\Pke,\Qke$ commute with $\PT$. E.g.,
	\begin{equation*}
	\begin{split}
	\Pke\big(\PT(\tu)\big)(x,k)&=\sum_{j=1}^N\langle\PT(\tu)(\cdot,k),\epsilon p_{n_*}(\cdot,k)\rangle p_{n_*}(x,k)\\
	&=\sum_{j=1}^N\langle\PT(\tu)(\cdot,k),\epsilon \PT(p_{n_*})(\cdot,k)\rangle\PT(p_{n_*})(x,k)\\
	&=\PT\bigg(\sum_{j=1}^N\overline{\langle\PT(\tu)(\cdot,k),\epsilon \PT(p_{n_*})(\cdot,k)\rangle}\,p_{n_*}(x,k)\bigg)\\
	&=\PT\bigg(\sum_{j=1}^N\langle\tu(\cdot,k),\epsilon p_{n_*}(\cdot,k)\rangle p_{n_*}(x,k)\bigg)=\PT(\Pke\tu)(x,k),
	\end{split}
	\end{equation*}
	since
	\begin{equation*}
	\begin{split}
	\int_{\R^2}\overline{\tu}(-x,k)\cdot\epsilon(x)p_{n_*}(-x,k)\dd x&=\int_{\R^2}\overline{\tu}(y,k)\cdot\epsilon(-y)p_{n_*}(y,k)\dd y\\
	&=\overline{\int_{\R^2}\tu(y,k)\cdot\epsilon(y)\overline{p_{n_*}(y,k)}\dd y}=\overline{\langle\tu(\cdot,k),\epsilon p_{n_*}(\cdot,k)\rangle}.
	\end{split}
	\end{equation*}
	\item $\tF$ only involves convolutions in $\B$ (cf. \eqref{F_Bloch}), hence the $\PT$-symmetry is trivially preserved using the evenness of $\chi^{(3)}$, see assumption (A6).
\end{itemize}

Consequently, if we start with a $\PT$-symmetric solution $\hA$ of the CMEs \eqref{CMEs_phys} and consider $\hb\in\cB^{2,\text{sym}}_{c\varepsilon^\rho}$, then all components $\hB,\hC,\tw_R,\tw_0$ inherit the same symmetry. This implies that the term $\hR(\hB)$ is real. Moreover, exploiting the $\PT$-symmetry of the mode $p_{n_*}(\cdot,k)$, see \eqref{E:sym-minus-x}, it is easy to show that the coefficients $\Theta_{\alpha,\beta,\gamma,j}^{a,b,c,d}$ defined in \eqref{Theta} (or equivalently the coefficients $I_{\alpha,\beta,\gamma}^{\,j}$ defined in \eqref{I_abg^j}) are real. Hence also $\cG(\hB)$ and $D_{\hA}\cG(\hA)$ are real. We are able to conclude that $\big(\chi_{\varepsilon,r-1}D_{\hA}\cG(\hA)\big)^{-1}\cW:\cB^{2,\text{sym}}_{c\varepsilon^\rho}\to\cB^{2,\text{sym}}_{c\varepsilon^\rho}$ with the former choices of $\rho$ and $c$ in \eqref{rho_c}, and therefore we find a real solution $\hb$ to \eqref{b:eq^-1} satisfying \eqref{b:estimate}.
\vskip0.2truecm
This shows that the function $\tu$ in \eqref{u} constructed along Sec. \ref{Section_w0}-\ref{Section_B} is an $H^2$ $\PT$-symmetric solution of \eqref{eq_Bloch}.

%-----------------------------------------------
\subsection{Approximation Error of \texorpdfstring{$\boldsymbol{u}_\text{ans}$}{u_ans}}\label{Section_approx}
In order to complete the proof of Theorem \ref{Thm_main}, we need to show that the initial ansatz $u_\text{ans}$ defined in \eqref{ansatz_phys_var} is actually a good approximation of the solution $u$ of \eqref{eq} which we constructed in Sec. \ref{Section_w0}-\ref{Section_B}. Recalling that
\begin{equation*}
\tu(x,k)=\tv(x,k)+\tw(x,k)=\tvB(x,k)+\tvC(x,k)+\tw_0(x,k)+\tw_R(x,k),
\end{equation*}
and in virtue of the estimates \eqref{w0_estimate}, \eqref{wR_estimate} and \eqref{vC_estimate}, we have
\begin{equation}\label{final_1}
\|\tu_\text{ans}-\tu\|_{\cX_2}\leq \|\tu_\text{ans}-\tvB\|_{\cX_2}+c_A(\varepsilon^{2-2r}+\varepsilon^2+\varepsilon^3)\leq \|\tu_\text{ans}-\tvB\|_{\cX_2}+c\varepsilon^{2-2r}.
\end{equation}
We split now the first term as follows (cf. \eqref{ansatz_Bloch_var}):
\begin{equation*}
\begin{split}
\tu_\text{ans}&(x,k)-\tvB(x,k)=\frac1\varepsilon\sum_{j=1}^N\left\{\sum_{K\in\Lambda^*}\hA_j\left(\frac{k-\kj+K}\varepsilon\right)p_{n_*}(x,\kj)e^{\ri K\cdot x}-\hB_j\left(\frac{k-\kj}\varepsilon\right)p_{n_*}(x,k)\right\}\\
&=\frac1\varepsilon\sum_{j=1}^N\bigg\{-\hb_j\left(\frac{k-\kj}\varepsilon\right)p_{n_*}(x,k)+(\chi_{\varepsilon,r-1}\hA_j)\left(\frac{k-\kj}\varepsilon\right)\left(p_{n_*}(x,\kj)-p_{n_*}(x,k)\right)\\
&\quad+(\left(1-\chi_{\varepsilon,r-1}\right)\hA_j)\left(\frac{k-\kj}\varepsilon\right)p_{n_*}(x,\kj)+\sum_{0\not=K\in\Lambda^*}\hA_j\left(\frac{k-\kj+K}\varepsilon\right)p_{n_*}(x,\kj)e^{\ri K\cdot x}\bigg\}
\end{split}
\end{equation*}
and we estimate term by term. First, by \eqref{b:estimate} one gets
\begin{equation}\label{final_bj}
\left\|\hb_j\left(\tfrac{\cdot\,-\kj}\varepsilon\right)p_{n_*}\right\|_{\cX_2}^2\leq c\int_\B\left|\hbj\left(\tfrac{k-\kj}\varepsilon\right)\right|^2\dd k=c\varepsilon^2\|\hbj\|_{L^2(\varepsilon^{-1}\B)}^2\leq c\,\varepsilon^{2\left(1+\min\{1,s_A(1-r)\}\right)}.
\end{equation}
Second, using the Lipschitz continuity of the map $k\mapsto p_{n_*}(\cdot,k)\in H^2(\cQ)$ for $k$ in a vicinity of $\kj$ given by Lemma \ref{eigf_Lip_H2p}, we get
\begin{equation}\label{final_Aj_1}
\begin{split}
\int_\B\bigg|(\chi_{\varepsilon,r-1}\hA_j)&\left(\tfrac{k-\kj}\varepsilon\right)\bigg|^2\|p_{n_*}(\cdot,k)-p_{n_*}(\cdot,\kj)\|_{H^2(\cQ)}^2\dd k\\
&\leq c\varepsilon^2 \int_{B_{\varepsilon^r}(\kj)}\left|\hA_j\left(\tfrac{k-\kj}\varepsilon\right)\right|^2\left|\tfrac{k-\kj}\varepsilon\right|^2\dd k\\
&\leq c\varepsilon^4\int_{B_{\varepsilon^{r-1}}(0)}|z|^2|\hA_j(z)|^2\dd z = c\varepsilon^4\|\hA_j\|_{L^2_1}^2\leq c\varepsilon^4\|\hA_j\|_{L^2_{s_A}}^2.
\end{split}
\end{equation}
Next,
\begin{equation}\label{final_Aj_2}
\begin{split}
\int_\B\big|(1-\chi_{\varepsilon,r-1}\hA_j)&\big(\tfrac{k-\kj}\varepsilon\big)\big|^2\|p_{n_*}(\cdot,\kj)\|_{H^2(\cQ)}^2\dd k\\
&\leq c\varepsilon^2\sup_{z\not\in B_{\varepsilon^{r-1}}(0)}(1+|z|)^{-2s_A}\int_{\R^2\setminus B_{\varepsilon^{r-1}}(0)}(1+|z|)^{2s_A}|\hA_j(z)|^2\dd z\\
&\leq c\varepsilon^2(1+\varepsilon^{r-1})^{-2s_A}\|\hA_j\|_{L^2_{s_A}}^2\leq c\varepsilon^{2(1+s_A(1-r))}\|\hA_j\|_{L^2_{s_A}}^2.
\end{split}
\end{equation}
Finally we consider the term involving the translated Brillouin zones:
\begin{equation}\label{final_Aj_3}
\begin{split}
\sum_{0\not=K\in\Lambda^*}&\int_\B\big|\hA_j\left(\tfrac{k-\kj+K}\varepsilon\right)\big|^2\big\|p_{n_*}(\cdot,\kj)e^{\ri K\cdot}\big\|_{H^2(\cQ)}^2\dd k\\
&\leq c\varepsilon^2\!\sup_{0\not=\tilde{K}\in\Lambda^*}\,\sup_{z\in\varepsilon^{-1}\left(\B-\kj+\tilde{K}\right)}\!(1+|z|)^{-2s_A}\!\!\sum_{0\not=K\in\Lambda^*}\int_{\varepsilon^{-1}\left(\B-\kj+K\right)}\!\!(1+|z|)^{2s_A}|\hA_j(z)|^2\dd z\\
&\leq c\varepsilon^2(1+\varepsilon^{-1})^{-2s_A}\int_{\R^2\setminus\varepsilon^{-1}(\B-\kj)}(1+|z|)^{2s_A}|\hA_j(z)|^2\dd z\\
&\leq c\varepsilon^{2(1+s_A)}\|\hA_j\|_{L^2_{s_A}}^2.
\end{split}
\end{equation} 
By combining estimates \eqref{final_1}-\eqref{final_Aj_3} we arrive at
\begin{equation*}
\begin{split}
\|\tu_\text{ans}-\tu\|_{\cX_2}&\leq\frac{c}\varepsilon \big(\varepsilon^{1+\min\{1,s_A(1-r)\}}+\big(\varepsilon^2+\varepsilon^{1+s_A(1-r)}+\varepsilon^{1+s_A}\big)\|\hA_j\|_{L^2_{s_A}}+\varepsilon^{3-2r}\big)\\
&\leq c\,\varepsilon^{\min\{1,s_A(1-r)\}}
\end{split}
\end{equation*}
since $r\in(0,\tfrac12]$. Because $s_A>1$, if we take $r\in \Big(0,1-\tfrac1{s_A}\Big]$ then $s_A(1-r)\geq 1$ and hence
$$\|\tu_\text{ans}-\tu\|_{\cX_2}\leq c\varepsilon.$$
The proof is thus complete recalling that the Bloch transform is an isomorphism between $\cX_2$ and $H^2(\R^2)$, see \eqref{Bloch_isomorphismum}.

%------------
\appendix
\section{Appendix}

In this last section we collect some auxiliary results needed throughout the paper. First, retracing the strategy of its standard proof (see e.g. \cite[Theorem 23.17]{Schw}), we prove a Helmholtz decomposition adapted to our ``shifted'' operator $\nabla'_k$. This is employed in Sec. \ref{S:spec-H} for the well-posedness of the eigenvalue problem in $L^2_\#(\cQ)$. Second, we address the regularity of the eigenfunctions $p_j(\cdot,k)$ and the Lipschitz continuity of the maps $k\mapsto\omega_j(k)$, $k\mapsto q_j(\cdot,k)$, and $k\mapsto p_j(\cdot,k)$, i.e. of the eigenvalue and eigenfunctions of the Bloch eigenvalue problems \eqref{E:vj-eq-L2} and \eqref{E:ev-eq-weak-uj}. The Lipschitz continuity is exploited in the nonlinear estimates of Sec. \ref{Section_approx}.

\subsection{Helmholtz Decomposition.}

We note first that $H^1(\Omega,\C)=\{f\in L^2(\Omega,\C)\,|\,\nabla'_kf\in L^2(\Omega,\C^3)\}$ for any $k\in \R^2$ and any measurable $\Omega \subset \R^2$. 
We also define 
$$H({\rm curl},\Omega):=\{v\in L^2(\Omega,\C^3): \nabla'\times v \in L^2(\Omega,\C^3)\}.$$

\begin{lem}\label{L:helmh}
	Let $\Omega\subset \R^2$ be a bounded domain, $\kappa\in\R\setminus\{0\}$ and $k\in \R^2$. Then
	$$H({\rm curl},\Omega)=W_k\oplus Z_k,$$
	where
	$$W_k:=\left\{w\in H({\rm curl},\Omega)\,\big|\,\int_\Omega w\cdot\overline{\nabla'_kf}=0,\;\forall f\in H^1(\Omega,\C)\right\}$$
	and
	$$Z_k:=\{g\in L^2(\Omega,\C^3)\,|\,\exists\,\psi\in H^1(\Omega,\C)\,:\,g=\nabla'_k\psi\}$$
	are closed subspaces.
\end{lem}
\begin{proof}
	Notice that $W_k$ and $Z_k$ are by definition orthogonal in $L^2(\Omega,\C^3)$ and $Z_k\subset H({\rm curl},\Omega)$ since $\nabla'_k\times(\nabla'_k\psi)=0$ for all $\psi\in H^1(\Omega,\C)$.
	
	\textbf{Step 1}: \textit{$W_k$ and $Z_k$ are closed}.
	
	Let $(w_j)_j\subset W_k$, i.e. $\int_\Omega w_j\cdot\overline{\nabla'_kf}=0$  for any $j\in\N$, and assume $w_j\to w$ in $H({\rm curl},\Omega)$. Then
	\begin{equation*}
	\left|\int_\Omega w_j\cdot\overline{\nabla'_kf}-\int_\Omega w\cdot\overline{\nabla'_kf}\,\right|\leq \int_\Omega|w_j-w||\nabla'_kf|\leq\|w_j-w\|_2\|\nabla'_kf\|_2\to 0,
	\end{equation*}
	therefore $w\in W_k$.
	
	Let now $(g_j)_j\subset Z_k$ be such that $g_j\to g$ in $H({\rm curl},\Omega)$. Then $g_j=\nabla'_k\psi_j$ with $\psi_j\in H^1(\Omega,\C)$ for any $j\in\N$ and the sequence $(\nabla'_k\psi_j)_j$ is Cauchy in the $L^2$-norm. Noticing that
	$$\|\nabla'_k\psi_j\|^2_2=\|\pa_1\psi_j+\ri k_1\psi_j\|^2_2+\|\pa_2\psi_j+\ri k_2\psi_j\|^2_2+|\kappa|^2\|\psi_j\|^2_2,$$
	one immediately infers that also $(\psi_j)_j$ is Cauchy in $L^2(\Omega,\C)$ and since $\|\psi_j\|_{H^1}^2\leq \|\nabla'_k \psi_j\|_{L^2}^2+c(k,\kappa)\|\psi_j\|_{L^2}^2$, the sequence $(\psi_j)_j$ is Cauchy also in $H^1(\Omega,\C)$. Hence there exists $\psi\in H^1(\Omega,\C)$ such that $\psi_j\to\psi$ in $H^1(\Omega,\C)$, such that $g_j=\nabla'_k\psi_j\to\nabla'_k\psi$ in $L^2(\Omega,\C^3)$. By the uniqueness of the limit we deduce $g=\nabla'_k\psi$, hence $g\in Z_k$.
	
	\textbf{Step 2}: \textit{Decomposition}.
	
	Let $v\in H({\rm curl},\Omega)$ and introduce $\mu_v\in(H^1(\Omega,\C))'$ via 
	\begin{equation*}
	\mu_v(\varphi):=\int_\Omega v\cdot\overline{\nabla'_k\varphi}\quad\ \mbox{for all}\,\,\varphi\in H^1(\Omega,\C),
	\end{equation*}
	as well as the sesquilinear form $S_k: H^1(\Omega,\C)\times H^1(\Omega,\C)\to\C$ via
	\begin{equation*}
	\begin{split}
	S_k(\psi,\varphi):&=\int_\Omega\nabla'_k\psi\cdot\overline{\nabla'_k\varphi}\\
	&=\int_\Omega\nabla\psi\cdot\overline{\nabla\varphi}+(\kappa^2+k_1^2+k_2^2)\int_\Omega\psi\overline{\varphi}-\ri\int_\Omega\nabla\psi\cdot k\overline{\varphi}+\ri\int_\Omega\overline{\nabla\varphi}\cdot k\psi\,,
	\end{split}
	\end{equation*}
	which is clearly continuous in $H^1(\Omega,\C)$. We now prove that $S_k$ is also coercive in $H^1(\Omega,\C)$:
	\begin{equation}\label{coercivity'k}
	\begin{split}
	S_k(\psi,\psi)&=\int_\Omega|\nabla\psi|^2+(\kappa^2+k_1^2+k_2^2)\int_\Omega|\psi|^2+2\,\text{Im}\bigg[\int_\Omega\nabla\psi\cdot k\overline{\psi}\bigg]\\
	&\geq(1-\varepsilon)\int_\Omega|\nabla\psi|^2+\big(\kappa^2+(1-\tfrac1{\varepsilon})(k_1^2+k_2^2)\big)\int_\Omega|\psi|^2.
	\end{split}
	\end{equation}
	If we choose $\varepsilon\in\left(\frac{|k|^2}{|\kappa|^2+|k|^2},1\right)$, which is nonempty since $\kappa\not=0$, both constants in \eqref{coercivity'k} are positive and the sesquilinear form is coercive in $H^1(\Omega,\C)$. By the theorem of Lax-Milgram we then find $\psi\in H^1(\Omega,\C)$ such that
	\begin{equation*}
	\int_\Omega\nabla'_k\psi\cdot\overline{\nabla'_k\varphi}=\int_\Omega v\cdot\overline{\nabla'_k\varphi}\quad\ \mbox{for all}\,\,\varphi\in H^1(\Omega,\C).
	\end{equation*}
	This means that $\nabla'_k\psi=:g\in L^2(\Omega,\C^3)$ and, being a gradient field, also $g\in H({\rm curl},\Omega)$. Hence, $w:=v-g\in W_k$ since $w\in L^2(\Omega,\C^3)$ and
	\begin{equation*}
	\begin{split}
	\int_\Omega w\cdot\overline{\nabla'_kf}&=\int_\Omega v\cdot\overline{\nabla'_kf}-\int_\Omega g\cdot\overline{\nabla'_kf}=\int_\Omega v\cdot\overline{\nabla'_kf}-\int_\Omega \nabla'_k\psi\cdot\overline{\nabla'_kf}=0.
	\end{split}
	\end{equation*}	
\end{proof}

\subsection{Regularity of \texorpdfstring{$\boldsymbol{p_j(\cdot,k)}$}{pj(.,k)} and of the maps \texorpdfstring{$\boldsymbol{k\mapsto\omega^2_j(k)}$}{k->omegaj(k)} and \texorpdfstring{$\boldsymbol{k\mapsto p_j(\cdot,k)}$}{k->pj(.,k)}}

We prove here some regularity results for the Bloch eigenvalues $(\omega^2_j(k))_j$, $k\in\B$, and eigenfunctions $(q_j(\cdot,k))_j$ and $(p_j(\cdot,k))_j$ (for problems \eqref{E:ev-eq-weak-vj} and \eqref{E:ev-eq-weak-uj}, respectively) described in Sec. \ref{S:spec-H}-\ref{S:spec-E}. In particular, we aim to show that the choice of our potential by (A1) and (A6), i.e. $0<\epsilon\in W^{2,\infty}(\cQ)$, $\Lambda$-periodic and with $\epsilon^{-1}\in L^\infty(\cQ)$, is sufficient to have for each $j\in \N$
\begin{enumerate}[a)]
	\item $\sup_{k\in\B}\|p_j(\cdot,k)\|_{H^2(\cQ)}<\infty$\; and\; $\sup_{k\in\B}\|p_j(\cdot,k)\|_{W^{2,\infty}(\cQ)}<\infty$\,;
	\item the map $\mathbb K\ni k\mapsto p_j(\cdot,k)\in H^2_\#(\cQ)$ is Lipschitz continuous, provided $\omega_j(k)$ is simple for all $k\in\mathbb K\subset\B$. \label{key2}
\end{enumerate}
To this aim, several lemmas will be needed. In the whole section, in addition to the notation introduced in Sec. \ref{Section:Linear}, we denote $\|\cdot\|_{L^2_\#(\cQ)}$ by $\|\cdot\|_2$. Our method of proof is inspired by that in \cite{CV97}.

\begin{lem}\label{eigv_Lip}
	The map $\B\ni k\mapsto\omega^2_j(k)\in [0,\infty)$ is Lipschitz continuous.
\end{lem}
\begin{proof}
	Recall that the sesquilinear form $a_k(\cdot,\cdot)$  of the $H$-eigenvalue problem is defined as $a_k(\varphi,\psi):=\int_\cQ\frac1\epsilon\nabla_k'\times \varphi\cdot\overline{\nabla_k'\times\psi}\dd x$ for $\varphi,\psi\in V_k$. Since for a fixed $\ko\in\B$ it is $\nabla'_k=\nabla'+\ri k=\nabla'_\ko+\ri(k-\ko)$, one has
	\begin{equation*}
	a_k(v,v)=a_\ko(v,v)+R(v,k,\ko),
	\end{equation*}
	with
	\begin{equation*}
	\begin{split}
	|R(v,k,\ko)|&\leq\|\epsilon^{-1}\|_\infty\left(\|v\|_2^2|k-\ko|^2+2\|\nabla'_\ko\times v\|_2\|v\|_2|k-\ko|\right)\\
	&\leq c\,|k-\ko|\left(\|v\|_2^2(|k|+|\ko|)+\|v\|_2\left(\|\nabla'\times v\|_2+|\ko|\|v\|_2\right)\right)\\
	&\leq c\,|k-\ko|\left(\|\nabla'\times v\|_2^2+\|v\|_2^2\right).
	\end{split}
	\end{equation*}
	Using the variational characterization of the eigenvalues
	$$\omega^2_j(k)=\min_{\dim S=j}\max_{v\in S}\frac{a_k(v,v)}{\|v\|_2^2},$$
	where $S$ is an arbitrary subspace of $V_k$, we infer
	\begin{equation*}
	\omega^2_j(k)-\omega^2_j(\ko)\leq\min_{\dim S=j}\max_{v\in S}\frac{R(v,k,\ko)}{\|v\|_2^2}\leq c|k-\ko|\min_{\dim S=j}\max_{v\in S}\left(\frac{\|\nabla'\times v\|_2^2}{\|v\|_2^2}+1\right)=c\,|k-\ko|(\omega_j^2(0)+1).
	\end{equation*}
	Interchanging $k$ and $\ko$, we finally get
	\begin{equation*}
	|\omega^2_j(k)-\omega^2_j(\ko)|\leq c\,|k-\ko|.
	\end{equation*}
\end{proof}

Notice that Lemma \ref{eigv_Lip} evidently implies
\begin{equation}\label{eigv_bdd}
\sup_{k\in\B}|\omega^2_j(k)|\leq C
\end{equation}
for all $j\in\N$.

\begin{lem}\label{Lemma_Reg_supH^2}
	For all $j\in\N$ and $p_j$ defined in \eqref{pj} one has
	\begin{equation}\label{Reg_est_supH^2}
	\sup_{k\in\B}\|p_j(\cdot,k)\|_{H^2(\cQ)}<\infty.
	\end{equation}
\end{lem}
\begin{proof}
	First, by the choice of the normalization of the Bloch eigenfunctions in \eqref{orthog_pn}, one has
	$$\|p_j(\cdot,k)\|^2_2\leq\|\epsilon^{-1}\|_\infty\int_\cQ\epsilon(x)|p_j(x,k)|^2\dd x=\|\epsilon^{-1}\|_\infty.$$
	Next, applying the divergence operator $\nabla'_k\cdot$ to \eqref{E:ev-eq-weak-uj}, one finds
	\begin{equation}\label{divk_pj}
	\nabla'_k\cdot p_j(x,k)=-\epsilon^{-1}(x)\nabla'_k\epsilon(x)\cdot p_j(x,k).
	\end{equation}
	Noticing that
	$$\|\nabla'_k\epsilon\|_\infty\leq \|\nabla\epsilon\|_\infty+
	\max_{k\in\B}(|k|+|\kappa|)\|\epsilon\|_\infty\leq c\,\|\epsilon\|_{W^{1,\infty}},$$
	one infers
	\begin{equation}\label{Reg_est:div}
	\sup_{k\in\B}\|\nabla'_k\cdot p_j(\cdot,k)\|_{L^2(\cQ)}\leq c\,\|\epsilon\|_{W^{1,\infty}}.
	\end{equation}
	To have a bound on the $H^1$-norm, we need to estimate also $\nabla'_k\times p_j(\cdot,k)$. We exploit the definition \eqref{pj} and equation \eqref{E:vj-eq-L2} that $q_j(\cdot,k)$ satisfies in the $L^2$-sense to get
	\begin{equation*}
	\nabla'_k\times p_j(\cdot,k)=\ri\nabla'_k\times\left(\frac1{\epsilon}\nabla'_k\times q_j(\cdot,k)\right)=\ri\omega_j^2(k)q_j(\cdot,k).
	\end{equation*}
	Therefore, from \eqref{eigv_bdd} and the normalization of the eigenfunctions $q_j(\cdot,k)$ we deduce
	\begin{equation}\label{Reg_est:curl}
	\sup_{k\in\B}\|\nabla'_k\times p_j(\cdot,k)\|_2\leq \sup_{k\in\B}|\omega_j^2(k)|\|q_j(\cdot,k)\|_2\leq c.
	\end{equation}
	We can thus conclude by \eqref{Reg_est:div}-\eqref{Reg_est:curl} that the same bound holds also in $H^1_\#(\cQ)$, i.e. 
	\begin{equation}\label{Reg_est:H1}
	\sup_{k\in\B}\|p_j(\cdot,k)\|_{H^1(\cQ)}\leq c.
	\end{equation}
	The $H^2$-norm is estimated similarly since
	\begin{equation}\label{H2_est}
	\|p_j(\cdot,k)\|_{H^2(\cQ)}\leq c\,\left(\|p_j(\cdot,k)\|_{L^2(\cQ)}+\|\nabla'_k\cdot p_j(\cdot,k)\|_{H^1(\cQ)}+\|\nabla'_k\times p_j(\cdot,k)\|_{H^1(\cQ)}\right).
	\end{equation}
	First, by \eqref{E:ev-eq-weak-uj} and \eqref{eigv_bdd},
	\begin{equation}\label{Reg_est:curlcurl}
	\begin{split}
	\|\nabla'_k\times p_j(\cdot,k)\|_{H^1(\cQ)}&\leq c\left(\|\nabla'_k\times p_j(\cdot,k)\|_2+\|\nabla'_k\times\nabla'_k\times p_j(\cdot,k)\|_2\right)\\
	&\leq c+\|\epsilon\|_\infty|\omega_j^2(k)|\|p_j(\cdot,k)\|_2\leq c.
	\end{split}
	\end{equation}
	Next, from \eqref{divk_pj} we deduce
	\begin{equation*}
	\nabla'_k\left(\nabla'_k\cdot p_j(\cdot,k)\right)=J'_k\left(\frac{\nabla'_k\epsilon}{\epsilon}\right)p_j(\cdot,k)+J'_k(p_j(\cdot,k))\frac{\nabla'_k\epsilon}{\epsilon},
	\end{equation*}
	where $J'_k(V)$ stands for the Jacobian of the vector field $V:\R^2\to\R^2$ with the derivatives $\partial_m$ replaced by the ``shifted'' derivatives $\partial_m+\ri k_m$ for $m\in\{1,2\}$. Hence from \eqref{Reg_est:H1} we have
	\begin{equation}\label{Reg_est:graddiv}
	\begin{split}
	\|\nabla'_k\left(\nabla'_k\cdot p_j(\cdot,k)\right)\|_{L^2(\cQ)}&\leq\left\|J'_k\left(\frac{\nabla'_k\epsilon}{\epsilon}\right)\right\|_\infty \|p_j(\cdot,k)\|_2+\|J'_k(p_j(\cdot,k))\|_2\left\|\frac{\nabla'_k\epsilon}\epsilon\right\|_\infty\\
	&\leq c\left(\|\epsilon\|_{W^{2,\infty}},\|\epsilon^{-1}\|_\infty\right)\|p_j(\cdot,k)\|_{H^1(\cQ)}\leq c,
	\end{split}
	\end{equation}
	for all $k\in\B$. Combining \eqref{H2_est} with \eqref{Reg_est:curlcurl} and \eqref{Reg_est:graddiv}, one infers \eqref{Reg_est_supH^2} and the proof is concluded.
\end{proof}

\begin{lem}\label{Lemma_Reg_supW^2infty}
	For all $j\in\N$ and $p_j$ defined in \eqref{pj} one has
	\begin{equation*}
	\sup_{k\in\B}\|p_j(\cdot,k)\|_{W^{2,\infty}(\cQ)}<\infty.
	\end{equation*}
\end{lem}
\begin{proof}
	By Lemma \ref{Lemma_Reg_supH^2} and the embedding $H^2(\cQ)\hookrightarrow L^\infty(\cQ)$ we infer $\sup_{k\in\B}\|p_j(\cdot,k)\|_\infty<\infty$. The upgrade to $W^{2,\infty}$-regularity is then accomplished by following the same steps as in the proof of Lemma \ref{Lemma_Reg_supH^2}.
\end{proof}

Next, we aim to prove \eqref{key2}. Let $\mathbb K$ be a connected and contractible subset of $\B$ such that $\omega_j(k)$ is simple for all $k\in\mathbb K$. Notice that we meet such a condition if $\mK=B_\delta(k^{(i)})$ with $j=n_*$, $k^{(i)}\in\{k^{(1)},\dots,k^{(N)}\}$ and $0<\delta\ll 1$ by assumption (A3). Indeed, the geometric simpleness of $\omega_{n_*}(k^{(i)})$ can be extended to $\omega_{n_*}(k)$ for $k$ in a whole neighbourhood of $k^{(i)}$, see \cite[Theorem IV.3.16]{Kato}.

As a first step, we prove the following.
\begin{lem}\label{eigf_Lip_L2q}
	The map $\mathbb K\ni k\mapsto q_j(\cdot,k)\in L^2_\#(\cQ)$ is $C^2$.
\end{lem}
In fact, Lipschitz continuity of the map in Lemma \ref{eigf_Lip_L2q} will be enough for our purposes in Sec. \ref{Section_approx}.
\begin{proof}
	Define the operator $A_0(k):=\nabla'_k\times\left(\frac1\epsilon\nabla'_k\times\right)+a_0I=L_k^{(H)}+a_0I$,
	where $a_0$ is a positive constant. Since the spectrum of $L_k^{(H)}$ is contained in the non-negative half-line (see Sec. \ref{S:spec-H}), the operator $A_0(k)$ is invertible, so in particular $A_0(k)^{-1}:L^2_\#(\cQ)\to V_k$, the latter space being the form domain of $L_k^{(H)}$ defined in \eqref{DomainVk}. Hence
	\begin{equation*}
	S_j(k):=A_0(k)^{-1}E-\nu_j(k)I,\qquad k\in\mathbb K,
	\end{equation*}
	where $\nu_j(k):=\left(a_0+\omega_j^2(k)\right)^{-1}$ and $E:H^1_\#(\cQ)\to L^2_\#(\cQ)$ is the identical embedding, is a well-defined Fredholm operator on $L^2_\#(\cQ)$ which depends on $k$ in a $C^2$ fashion. Indeed, $E$ is a compact embedding (see e.g. \cite[Theorems 3.5,3.7]{Agmon}) and so $S_j(k)$ is a compact perturbation of (a multiple of) the identity. The $C^2$-regularity is a consequence of the same property that the map $k\mapsto\omega_j(k)$ enjoys, see assumption (A4). Moreover, it is easy to see that $\ker S_j(k)$ coincides with the $j$-th eigenspace of $L_k^{(H)}$ and so, by the geometric simpleness of $\omega_j(k)$, it is of dimension one for all $k\in\mathbb K$. This yields the structure of a vector bundle to $\ker S_j:=\bigcup_{k\in\mathbb K}\ker S_j(k)$ over $\mathbb K$, see \cite[p.62]{BB}. Moreover, we claim that the map $k\mapsto\ker S_j(k)$ is $C^2$.
	
	To this aim let $\ko\in\mK$. Since $S_j(k)$ is a self-adjoint Fredholm operator with a nontrivial kernel for all $k\in\mK$, there exists an interval $[-\delta,\delta]\subset\R$, such that $\sigma(S_j(k))\cap[-\delta,\delta]=\{0\}$\footnote{If $A$ is a self-adjoint operator on a Hilbert space and $\lambda\in\C$, then $A-\lambda$ is Fredholm if and only if $\lambda$ is a discrete eigenvalue of finite multiplicity or lies in the resolvent of $A$. See also \cite[Chp. XVII Theorem 2.1]{GGK}.}. 
	Since $k\mapsto S_j(k)$ is continuous, by spectral continuity \cite[Chapter II, Theorem 4.2]{GGK}, such $\delta$ can be chosen independent of $k$ for all $k\in B_r(\ko)$. Consider therefore the map  
	\begin{equation*}
	S_j(k)\mapsto P(k):=\frac1{2\pi\ri}\oint_\Gamma\left(S_j(k)-\lambda\right)^{-1}\dd\lambda,
	\end{equation*}
	where $\Gamma$ is a closed curve in $\C$ that isolates $\{0\}$ from the rest of the spectrum. Then $P(k)$ is a projection onto the eigenspace of the $0$ eigenvalue for all $k\in\mK$, i.e. $P(k)=\Pi_{\ker(S_j(k))}$, see \cite[Sec.6.4]{Kato} or \cite[Theorems XII.5-6]{RS}. It is clear then that $k\mapsto P(k)$ is $C^2$ relying on the same property of $k\mapsto S_j(k)$. Therefore the map $k\mapsto\Imag P(k)=\ker(S_j(k))$ shares the same regularity too, and the claim is proved.
	
	$\mathbb K$ being contractible, the vector bundle $\ker S_j$ is $C^2$-diffeomorphic to the trivial bundle $\mK\times\C$, see e.g. \cite[Ex.2 Chapter 4.1]{Hirsch}, which clearly has a constant section $\tilde s_j:\mK\to\mK\times\C$ such that $\tilde s_j(k)=(k,1)$. Then, calling such diffeomorphism $\varphi_j$, the map $s_j:\mathbb K\to\ker S_j$ defined as $s_j:=\varphi_j^{-1}\circ\tilde s_j$ is a $C^2$ section over $\ker S_j$. This means that, up to a multiplication by a unitary complex function, it is possible to redefine the $j$-th eigenfunction $q_j(\cdot,k)$ normalized as in \eqref{orthog_pn} and such that the map $\mathbb K\ni k\mapsto q_j(\cdot,k)\in L^2_\#(\cQ)$ is $C^2$.
\end{proof}

Before transferring such a property to the eigenfunctions $p_j(\cdot,k)$, we need a stronger results on $q_j(\cdot,k)$.

\begin{lem}\label{eigf_Lip_H1q}
	The map $\mathbb K\ni k\mapsto q_j(\cdot,k)\in H^1_\#(\cQ)$ is Lipschitz continuous.
\end{lem}
\begin{proof}
	In other words, we aim to prove that for an arbitrary $\ko\in\mathbb K$ there exists a suitable constant $c(\ko)>0$ such that
	\begin{equation}\label{eigf_Lip_H1q_estimate}
	\|q_j(\cdot,k)-q_j(\cdot,\ko)\|_{H^1(\cQ)}\leq c\,|k-\ko|,\qquad\mbox{for all}\;\,k\in\mathbb K.
	\end{equation}
	Noticing that the Helmholtz decomposition of $H^1_\#(\cQ)$ of Lemma \ref{L:helmh} holds with the operators $\nabla'\times$ and $\nabla'\cdot$, as the particular case when $k=0$, we estimate separately $\|\nabla'\cdot q_j(\cdot,k)-\nabla'\cdot q_j(\cdot,\ko)\|_2$ and $\|\nabla'\times q_j(\cdot,k)-\nabla'\times q_j(\cdot,\ko)\|_2$. Since $\nabla'_k\cdot q_j(\cdot,k)=0$ for all $k\in\B$, one has
	\begin{equation}\label{eigf_Lip_H1q_div}
	\begin{split}
	\|\nabla'\cdot q_j(\cdot,k)-\nabla'\cdot q_j(\cdot,\ko)\|_2&=\|k\cdot q_j(\cdot,k)-\ko\cdot q_j(\cdot,\ko)\|_2\\
	&\leq|k|\|q_j(\cdot,k)-q_j(\cdot,\ko)\|_2-|k-\ko|\|q_j(\cdot,\ko)\|_2\\
	&\leq c\,|k-\ko|,
	\end{split}
	\end{equation}
	due to Lemma \ref{eigf_Lip_L2q}. The estimate for the difference of the curls' is more involved and is based on equation \eqref{E:ev-eq-weak-vj} which the eigenfunctions satisfy. First,
	\begin{equation}\label{eigf_Lip_H1q_curl}
	\|\nabla'\times q_j(\cdot,k)-\nabla'\times q_j(\cdot,\ko)\|_2=\|\nabla'_k\times q_j(\cdot,k)-\nabla'_\ko\times q_j(\cdot,\ko)\|_2+\|k\times q_j(\cdot,k)-\ko\times q_j(\cdot,\ko)\|_2,
	\end{equation}
	where the second term is estimated like above. Noticing that $\nabla'_k=\nabla'_\ko+\ri(k-\ko)$, we write
	\begin{equation}
	\begin{split}
	\|\nabla'_k\times q_j(\cdot,k)&-\nabla'_\ko\times q_j(\cdot,\ko)\|_2^2\\
	&=\int_\cQ\nabla'_k\times q_j(\cdot,k)\cdot\overline{\left(\nabla'_k\times q_j(\cdot,k)-\left(\nabla'_k-\ri(k-\ko)\right)\times q_j(\cdot,\ko)\right)}\\
	&\quad-\int_\cQ\nabla'_\ko\times q_j(\cdot,\ko)\cdot\overline{\left(\left(\nabla'_\ko+\ri(k-\ko)\right)\times q_j(\cdot,k)-\nabla'_\ko\times q_j(\cdot,\ko)\right)}\\
	&=M_1-\ri M_2,
	\end{split}
	\end{equation}
	where
	\begin{equation}
	M_1:=\int_\cQ\nabla'_k\times q_j(\cdot,k)\cdot\overline{\nabla'_k\times(q_j(\cdot,k)-q_j(\cdot,\ko))}-\int_\cQ\nabla'_\ko\times q_j(\cdot,\ko)\cdot\overline{\nabla'_\ko\times(q_j(\cdot,k)-q_j(\cdot,\ko))}
	\end{equation}
	and
	\begin{equation}
	M_2:=\int_\cQ\nabla'_k\times q_j(\cdot,k)\cdot\overline{(k-\ko)\times q_j(\cdot,\ko)}-\int_\cQ\nabla'_\ko\times q_j(\cdot,\ko)\cdot\overline{(k-\ko)\times q_j(\cdot,k)}.
	\end{equation}
	We estimate $M_1$ and $M_2$ separately. First, using \eqref{E:ev-eq-weak-vj},
	\begin{equation}
	\begin{split}
	M_1&\leq\|\epsilon^{-1}\|_\infty\left|\int_\cQ\frac1\epsilon\nabla'_k\times q_j(\cdot,k)\cdot\overline{\nabla'_k\times(q_j(\cdot,k)-q_j(\cdot,\ko))}\right.\\
	&\qquad \left. -\!\int_\cQ\frac1\epsilon\nabla'_\ko\times q_j(\cdot,\ko)\cdot\overline{\nabla'_\ko\times(q_j(\cdot,k)-q_j(\cdot,\ko))}\,\right|\\
	&\leq c\left|\omega_j^2(k)\int_\cQ q_j(\cdot,k)\cdot\overline{(q_j(\cdot,k)-q_j(\cdot,\ko))}-\omega_j^2(\ko)\int_\cQ q_j(\cdot,\ko)\cdot\overline{(q_j(\cdot,k)-q_j(\cdot,\ko))}\,\right|\\
	&\leq c\,|\omega_j^2(k)-\omega_j^2(\ko)|\|q_j(\cdot,\ko)\|_2\|q_j(\cdot,k)-q_j(\cdot,\ko)\|_2+c\sup_{k\in\B}|\omega_j^2(k)|\|q_j(\cdot,k)-q_j(\cdot,\ko)\|_2^2\\
	&\leq c\,|k-\ko|^2,
	\end{split}
	\end{equation}
	where in the last inequality we make use of \eqref{eigv_bdd} and Lemmas \ref{eigv_Lip} and \ref{eigf_Lip_L2q}. Similarly, we also get
	\begin{equation}\label{M2}
	\begin{split}
	M_2&\leq\|\nabla'_k\times q_j(\cdot,k)\|_2|k-\ko|\|q_j(\cdot,k)-q_j(\cdot,\ko)\|_2\\
	&\qquad +\|\nabla'_k\times q_j(\cdot,k)-\nabla'_\ko\times q_j(\cdot,\ko)\|_2|k-\ko|\| q_j(\cdot,\ko)\|_2\\
	&\leq c\sup_{k\in\B}\|\nabla'_k\times q_j(\cdot,k)\|_2|k-\ko|^2+\delta\|\nabla'_k\times q_j(\cdot,k)-\nabla'_\ko\times q_j(\cdot,\ko)\|_2^2+\frac{|k-\ko|^2}{4\delta}
	\end{split}
	\end{equation}
	for a small $\delta>0$. Therefore, combining equations \eqref{eigf_Lip_H1q_div}-\eqref{M2}, we finally infer \eqref{eigf_Lip_H1q_estimate}.
\end{proof}

We are now in the position to prove \eqref{key2}.
\begin{lem}\label{eigf_Lip_H2p}
	The map $\mathbb K\ni k\mapsto p_j(\cdot,k)\in H^2_\#(\cQ)$ is Lipschitz continuous.
\end{lem}
\begin{proof}
	Notice that Lemmas \ref{eigv_Lip} and \ref{eigf_Lip_H1q} and the definition $p_j(x,k)=\frac\ri{\epsilon(x)\omega_j(k)}\nabla'_k\times q_j(x,k)$ already imply that the above map with values in $L^2_\#(\cQ)$ is Lipschitz continuous.
	
	First we show that $\mathbb K\ni k\mapsto p_j(\cdot,k)\in H^1_\#(\cQ)$ is Lipschitz continuous. Once again we consider the $\nabla'$-Helmholtz decomposition and estimate separately $\|\nabla'\cdot p_j(\cdot,k)-\nabla'\cdot p_j(\cdot,\ko)\|_2$ and $\|\nabla'\times p_j(\cdot,k)-\nabla'\times p_j(\cdot,\ko)\|_2$. Similarly to \eqref{eigf_Lip_H1q_curl} we may confine ourselves to estimate $\|\nabla'_k\cdot p_j(\cdot,k)-\nabla'_\ko\cdot p_j(\cdot,\ko)\|_2$ and $\|\nabla'_k\times p_j(\cdot,k)-\nabla'_\ko\times p_j(\cdot,\ko)\|_2$.
	
	First, by \eqref{divk_pj} we have
	\begin{equation*}
	\begin{split}
	\|\nabla'_k\cdot p_j(\cdot,k)-\nabla'_\ko\cdot p_j(\cdot,\ko)\|_2&\leq\|\epsilon^{-1}\|_\infty\|\nabla'_k\epsilon\cdot p_j(\cdot,k)-\nabla'_\ko\epsilon\cdot p_j(\cdot,\ko)\|_2\leq c|k-\ko|,
	\end{split}
	\end{equation*}
	applying the triangular inequality and Lemmas \ref{eigv_Lip} and \ref{eigf_Lip_H1q}. Next, noticing that $\nabla'_k\times p_j(\cdot,k)=\ri\omega_j(k)q_j(\cdot,k)$ by \eqref{E:vj-eq-L2}, we may similarly infer the estimate $\|\nabla'_k\times p_j(\cdot,k)-\nabla'_\ko\times p_j(\cdot,\ko)\|_2\leq c|k-\ko|$ and, in turn, the Lipschitz continuity in the $H^1$-norm.
	
	The upgrade to the $H^2$-norm can be deduced analogously, combining the estimates above with \eqref{H2_est}-\eqref{Reg_est:graddiv}.
\end{proof}

\section*{Acknowledgement}
This research is supported by the German Research Foundation, DFG grant No. DO1467/4-1. The authors thank Michael Plum, KIT Karlsruhe, for fruitful discussions.

%%%%%%%%%%%%%%%%%%%%%%%%%%%%%%%%%%%%%%%%%%%%%%%%%%%%%%%%%%%%%%%%%%%%%%%%%%%%%%%%%%%%%%%%%%%
% Bibliography
%%%%%%%%%%%%%%%%%%%%%%%%%%%%%%%%%%%%%%%%%%%%%%%%%%%%%%%%%%%%%%%%%%%%%%%%%%%%%%%%%%%%%%%%%%%
\bibliographystyle{plain}
\bibliography{bibliography_just}

\end{document}